\documentclass[12pt,draftcls,onecolumn]{IEEEtran}

\def\bsg{{\boldsymbol{g}}}

\def\bsv{{\boldsymbol{v}}}

\def\bsx{{\boldsymbol{x}}}

\def\bsH{{\boldsymbol{H}}}

\def\bsK{{\boldsymbol{K}}}
\def\bsL{{\boldsymbol{L}}}

\def\bsM{{\boldsymbol{M}}}

\def\bsQ{{\boldsymbol{Q}}}

\def\calG{{\mathcal{G}}}

\def\calL{{\mathcal{L}}}

\let\labelindent\relax
\usepackage[cmex10]{amsmath}
\usepackage{amssymb}
\usepackage{algorithmic}
\usepackage{algorithm}

\usepackage{tikz}
\usepackage{tkz-berge}
\usepackage{lipsum}
\usetikzlibrary{calc}
\usetikzlibrary{shapes,fit}
\usepackage{verbatim}
\usepackage{setspace}
\usetikzlibrary{shapes,arrows}
\usepackage[short]{optidef}
\usepackage{pifont}

\usepackage{amsthm}

\usepackage{hyperref}

\usepackage{xcolor, colortbl}
\definecolor{LightGray}{gray}{0.9}
\definecolor{LightCyan}{rgb}{0.88,1,0.8}

\usepackage{color}
\definecolor{gray}{RGB}{128,128,128}

\usepackage{cite}
\usepackage{graphicx}
\graphicspath{{figure/}}
\usepackage{epstopdf}
\usepackage{tikz}
\usetikzlibrary{arrows,shapes,backgrounds}

\usepackage{bm}

\usepackage{array}
\newcolumntype{M}[1]{>{\centering\arraybackslash}m{#1}}
\newcolumntype{N}{@{}m{0pt}@{}}

\usepackage{multirow}

\usepackage{stfloats}
\usepackage{caption}
\usepackage{subcaption}

\usepackage{enumitem}
\usepackage{url}
\usepackage{adjustbox}
\usepackage{booktabs}
\usepackage{makecell}

\newtheorem{theorem}{Theorem}
\newtheorem{assumption}{Assumption}

\newtheorem{lemma}{Lemma}
\newtheorem{corollary}{Corollary}
\newtheorem{definition}{Definition}
\newtheorem{remark}{Remark}

\DeclareMathOperator{\nullrank}{null}

\DeclareMathOperator{\col}{col}

\DeclareMathOperator{\diag}{diag}
\DeclareMathOperator{\Deg}{Deg}

\newenvironment{proof}[1][Proof]%
  {\smallskip\par\noindent\textbf{#1\,:\ }}%
  {\hspace*{\fill} \rule{6pt}{6pt}\smallskip}
\newenvironment{proof*}[1][Proof]%
  {\smallskip\par\noindent\textbf{#1\,:\ }}%

\newtheorem{assumptionalt}{Assumption}[assumption]
\newenvironment{assumptionp}[1]{
  \renewcommand\theassumptionalt{#1}
  \assumptionalt
}{\endassumptionalt}

\newlength{\fwidth}

\allowdisplaybreaks
\begin{document}
\IEEEoverridecommandlockouts
\title{Zeroth-Order Algorithms for Stochastic Distributed Nonconvex Optimization}
\author{Xinlei Yi, Shengjun Zhang, Tao Yang, and Karl H. Johansson
\thanks{This work was supported by the
Knut and Alice Wallenberg Foundation, the  Swedish Foundation for Strategic Research, the Swedish Research Council, and the National Natural Science Foundation of China under grants 62133003 and 61991403.}
\thanks{X. Yi and K. H. Johansson are with Division of Decision and Control Systems, School of Electrical Engineering and Computer Science, KTH Royal Institute of Technology, and they are also affiliated with Digital Futures, 10044, Stockholm, Sweden. {\tt\small \{xinleiy, kallej\}@kth.se}.}%
\thanks{S. Zhang is with the Department of Electrical Engineering, University of North Texas, Denton, TX 76203 USA. {\tt\small  ShengjunZhang@my.unt.edu}.}
\thanks{T. Yang is with the State Key Laboratory of Synthetical Automation for Process Industries, Northeastern University, 110819, Shenyang, China. {\tt\small yangtao@mail.neu.edu.cn}.}
}

\maketitle

\begin{abstract}                
In this paper, we consider a stochastic distributed nonconvex optimization problem with the cost function being distributed over $n$ agents having access only to zeroth-order (ZO) information of the cost. This problem has various machine learning applications. As a solution, we propose two distributed ZO algorithms, in which at each iteration each agent samples the local stochastic ZO oracle at two points with a time-varying smoothing parameter. We show that the proposed algorithms achieve the linear speedup convergence rate $\mathcal{O}(\sqrt{p/(nT)})$ for smooth cost functions under the state-dependent variance assumptions which are more general than the commonly used bounded variance and Lipschitz assumptions, and $\mathcal{O}(p/(nT))$ convergence rate when the global cost function additionally satisfies the Polyak--{\L}ojasiewicz (P--{\L}) condition in addition, where $p$ and $T$ are the dimension of the decision variable and the total number of iterations, respectively. To the best of our knowledge, this is the first linear speedup result for distributed ZO algorithms, which enables systematic processing performance improvements by adding more agents. We also show that the proposed algorithms converge linearly under the relative bounded second moment assumptions and the P--{\L} condition.
We demonstrate through numerical experiments the efficiency of our algorithms on generating adversarial examples from deep neural networks in comparison with baseline and recently proposed centralized and distributed ZO algorithms.

\end{abstract}


\section{Introduction}

We consider stochastic distributed nonconvex optimization with zeroth-order (ZO) information feedback. Specifically, consider a network of $n$ agents/machines collaborating to solve the following optimization problem
\begin{align}\label{zerosg:eqn:xopt}
 \min_{x\in \mathbb{R}^p} f(x):=\frac{1}{n}\sum_{i=1}^n\mathbf{E}_{\xi_i}[F_i(x,\xi_i)],
\end{align}
where $x\in \mathbb{R}^p$ is the decision variable, $\xi_i$ is a random variable, and $F_i(\cdot,\xi_i): \mathbb{R}^{p}\mapsto \mathbb{R}$ is a stochastic component function (not necessarily convex). Each agent $i$ only has information about its own stochastic ZO oracle $F_i(x,\xi_i)$. In other words, for any given $x$ and $\xi_i$, each agent $i$ can sample $F_i(x,\xi_i)$ as a stochastic approximation of the true local cost function value $f_i(x)=\mathbf{E}_{\xi_i}[F_i(x,\xi_i)]$, but other information such as the first-order oracle cannot be observed. Agents can communicate with their neighbors through an underlying communication network. The network is modeled by an undirected graph $\mathcal G=(\mathcal V,\mathcal E)$, where $\mathcal V =\{1,\dots,n\}$ is the agent set, $\mathcal E
\subseteq \mathcal V \times \mathcal V$ the edge set, and $(i,j)\in \mathcal E$ if agents $i$ and $j$ can communicate with each other. The neighboring set of agent $i$ is denoted by $\mathcal{N}_i=\{j\in \mathcal V:~ (i,j)\in \mathcal E\}$.
The ZO information feedback setting has wide usage in applications, particularly when explicit expressions of the gradients are unavailable or difficult to obtain \cite{conn2009introduction,audet2017derivative,larson2019derivative}. For example, the cost functions of many big data problems that deal with complex data generating processes cannot be explicitly defined \cite{chen2017zoo}. Moreover, the distributed setting is a core aspect of many important applications in view of flexibility and scalability to large-scale datasets and systems, data privacy and locality, communication reduction to the central entity, and robustness to potential failures of the central entity \cite{nedic2018distributed,Koloskova2019decentralized,yang2019survey}.

\subsection{Literature Review}
The study of gradient-free (derivative-free) optimization has a long history, which can be traced back  at least to the 1960's \cite{hooke1961direct,matyas1965random,nelder1965simplex}. It has recently gained renewed attention by the machine learning community. Classical gradient-free optimization methods can be classified into direct-search and model-based methods. For example, 
stochastic direct-search and model-based trust-region algorithms have been proposed in \cite{bergou2019stochastic,bibi2019stochastic,gorbunov2019stochastic,golovin2019gradientless}
and \cite{marazzi2002wedge,conn2009global,scheinberg2010self}, respectively. 
In recent years, the more popular gradient-free optimization methods are ZO methods, which are gradient-free counterparts of first-order optimization methods and thus  easy to implement. In ZO optimization methods, the full or stochastic gradients are approximated by directional derivatives, which are calculated through sampled function values. A commonly used method to calculate directional derivatives is simply using the function difference at two points \cite{duchi2015optimal,shamir2017optimal,nesterov2017random}.

Various ZO optimization methods have been proposed, e.g., ZO (stochastic) gradient descent algorithms \cite{shamir2013complexity,ghadimi2013stochastic,bach2016highly,nesterov2017random,
balasubramanian2018zeroth,jin2018local,
ye2018hessian,vlatakis2019efficiently,kozak2020stochastic,liu2018zerothieee,
liu2019signsgd,zhang2020improving}, ZO stochastic coordinate descent algorithms \cite{lian2016Comprehensive}, ZO (stochastic) variance reduction algorithms \cite{liu2018zerothieee,liu2019signsgd,balasubramanian2018zeroth,jin2018local,
ghadimi2016mini,gao2018information,kazemi2018proximal,gu2018faster,fang2018spider,
liu2018zeroth,gorbunov2018accelerated,liu2018stochastic,
huang2019faster,pmlr-v97-ji19a,Huang2020Accelerated,chen2020accelerated,gao2020can}, ZO (stochastic) proximal algorithms \cite{huang2019faster,cai2020zeroth,ghadimi2016mini,nazari2020adaptive}, ZO Frank-Wolfe algorithms \cite{balasubramanian2018zeroth,sahu2019towards,Huang2020Accelerated,gao2020can}, ZO mirror descent algorithms \cite{duchi2015optimal,wang2018stochastic,gorbunov2018accelerated}, ZO adaptive momentum methods \cite{chen2019zo,nazari2020adaptive}, ZO methods of multipliers \cite{gao2018information,kazemi2018proximal,huang2019zeroth,huang2019nonconvex}, ZO stochastic path-integrated differential estimator \cite{pmlr-v97-ji19a,fang2018spider,huang2019nonconvex}.
Convergence properties of these algorithms have been analyzed in detail. For instance, the typical convergence result for deterministic centralized optimization problems is that first-order stationary points can be found at an $\mathcal{O}(p/T)$ convergence rate by the two-point sampling based ZO algorithms \cite{nesterov2017random,kozak2020stochastic}, while under stochastic settings the convergence rate is reduced to $\mathcal{O}(\sqrt{p/T})$ \cite{ghadimi2013stochastic,lian2016Comprehensive}, where $T$ is the total number of iterations.

Aforementioned ZO optimization algorithms are all centralized and thus not suitable to solve distributed optimization problems. Recently distributed ZO algorithms have been proposed, e.g., distributed ZO gradient descent algorithms \cite{yuan2014randomized,sahu2018distributed,sahu2018communication,wang2019distributed,pang2019randomized,tang2020distributedzero}, distributed ZO push-sum algorithm \cite{yuan2015gradient}, distributed ZO mirror descent algorithm \cite{yu2019distributed}, distributed ZO gradient tracking algorithm \cite{tang2020distributedzero}, distributed ZO primal--dual algorithms \cite{hajinezhad2018gradient,hajinezhad2019zone,yi2021linear}, distributed ZO sliding algorithm \cite{beznosikov2019derivative}, privacy-preserving distributed ZO algorithm \cite{gratton2020privacy}, distributed ZO Frank-Wolfe algorithm \cite{sahu2020decentralized}. Among these algorithms, the algorithms in \cite{yuan2014randomized,yuan2015gradient,sahu2018distributed,sahu2018communication,
yu2019distributed,yi2021linear,
tang2020distributedzero} are suitable to solve the deterministic form of \eqref{zerosg:eqn:xopt}, while the algorithm in \cite{hajinezhad2019zone} can be directly applied to solve the stochastic optimization problem~\eqref{zerosg:eqn:xopt}. However, the algorithm in \cite{hajinezhad2019zone} requires each agent to have an $\mathcal{O}(T)$ sampling size per iteration, which is not favorable in practice, although it was shown that first-order stationary points can be found at an $\mathcal{O}(p^2n/T)$ convergence rate.

From the discussion above,  three core theoretical questions arise  when considering stochastic distributed optimization problems:


(Q1) Can distributed ZO algorithms achieve similar convergence properties as centralized ZO algorithms? For instance, can distributed ZO algorithms based on two-point sampling have an $\mathcal{O}(\sqrt{p/T})$ convergence rate as their centralized counterparts in \cite{ghadimi2013stochastic,lian2016Comprehensive}?

(Q2) As shown in \cite{lian2017can}, distributed stochastic gradient descent (SGD) algorithms can achieve linear speedup  with respect to the number of agents compared with centralized SGD algorithms. Can distributed ZO algorithms also achieve linear speedup? In particular, can distributed ZO algorithms based on two-point sampling achieve the linear speedup convergence rate $\mathcal{O}(\sqrt{p/nT})$?

(Q3) For deterministic optimization problems, centralized and distributed ZO algorithms can achieve faster convergence rates under more stringent conditions such as strong convexity or Polyak--{\L}ojasiewicz (P--{\L}) conditions, as shown in \cite{nesterov2017random,ye2018hessian,pmlr-v97-ji19a,kozak2020stochastic,chen2020accelerated,
cai2020zeroth} and \cite{yi2021linear,tang2020distributedzero}, respectively. For stochastic optimization problems, can ZO algorithms also achieve faster convergence rates under strong convexity or P--{\L} conditions?

\subsection{Main Contributions}
This paper provides positive answers to the above three questions. We propose two distributed ZO algorithms, one primal--dual and one primal algorithm, to solve the stochastic optimization problem~\eqref{zerosg:eqn:xopt}.
In both algorithms, at each iteration each agent communicates its local primal variables to its neighbors through an arbitrarily connected communication network. Moreover, each agent samples its local stochastic ZO oracle at two points with a time-varying smoothing parameter. The contributions of this paper are summarized as follows.

(C1) We show in Theorems~\ref{zerosg:thm-sg-smT} and \ref{zerosg-p:thm-sg-smT} that our algorithms find a stationary point with the linear speedup convergence rate $\mathcal{O}(\sqrt{p/(nT)})$ for nonconvex but smooth cost functions under the state-dependent variance assumptions which are more general than the commonly used bounded variance and Lipschitz assumptions. This rate is faster than that achieved by the centralized ZO algorithms in \cite{ghadimi2013stochastic,lian2016Comprehensive,zhang2020improving,liu2018zerothieee,
liu2019signsgd,balasubramanian2018zeroth,chen2019zo} and the distributed ZO algorithm in \cite{tang2020distributedzero}. To the best of our knowledge, this is the first linear speedup result for distributed ZO algorithms; thus (Q1) and (Q2) are answered.

(C2) We show in Theorems~\ref{zerosg:thm-sg-diminishingt}, \ref{zerosg-a5:thm-sg-diminishingt}, \ref{zerosg-p:thm-sg-diminishingt}, and \ref{zerosg-p-a5:thm-sg-diminishingt} that our proposed algorithms find a global optimum with an $\mathcal{O}(p/(nT))$ convergence rate when the global cost function satisfies the P--{\L} condition in addition. This rate is faster than that achieved by the centralized ZO algorithms in \cite{shamir2013complexity,bach2016highly} and the distributed ZO algorithms in \cite{sahu2018distributed,tang2020distributedzero}, although  \cite{shamir2013complexity,bach2016highly,sahu2018distributed} assumed strongly convex cost functions and only considered additive sampling noise, and \cite{tang2020distributedzero} only considered the deterministic problem. This paper presents the first performance analysis for ZO algorithms to solve stochastic optimization problems under P--{\L}  or strong convexity assumptions; thus (Q3) is answered.

(C3) We show in Theorems~\ref{zerosg:thm-random-pd-fixed} and \ref{zerosg-p:thm-random-pd-fixed} that our algorithms  with constant algorithm parameters linearly converges to a neighborhood of a global optimum under the P--{\L} condition. Moreover, a precise global optimum can be linearly found if the relative bounded second moment assumptions also hold, see Corollaries~\ref{zerosg:thm-random-pd-fixed-coro1} and \ref{zerosg-p:thm-random-pd-fixed-coro1}. It should be mentioned that the P--{\L} constant is not used to design the algorithm parameters when showing these results. Compared with \cite{nesterov2017random,ye2018hessian,pmlr-v97-ji19a,kozak2020stochastic,chen2020accelerated,
cai2020zeroth} which also achieved linear convergence, we use weaker assumptions on the cost function and/or less samplings per iteration.

The detailed comparison between this paper and the literature is summarized in TABLE~\ref{zerosg:table}.

\begin{table*}[htbp]
\caption{Summary of existing works on ZO optimization, where NoSPPI denotes the number of sampled points per iteration, and the sampling complexity is the total number of function samplings to achieve $\mathbf{E}[\|\nabla f(x_T)\|^2]\le\epsilon$ for nonconvex problems or $\mathbf{E}[f(x_T)-f^*]\le\epsilon$ for (strongly) convex  problems or problems satisfying the P--{\L} condition.}
\label{zerosg:table}
\vskip 0.05in
\begin{center}
\tiny
\scalebox{1}{
\begin{tabular}{M{0.8cm}|M{8.0cm}|M{1.15cm}|M{1.6cm}|M{2.5cm}N}
\hline
Reference&Problem settings&NoSPPI &Convergence rate&Sampling complexity&\\[5pt]

\hline
\multirow{2}{*}[-2pt]{\parbox{0.8cm}{\centering \cite{nesterov2017random}}}&Deterministic, centralized, unconstrained, nonconvex, smooth&\multirow{2}{*}[-1pt]{\parbox{1.1cm}{\centering Two}}&$\mathcal{O}(p/T)$&$\mathcal{O}(p/\epsilon)$&\\[5pt]
\cline{2-2}\cline{4-5}
&Strongly convex in addition&&Linear&$\mathcal{O}(p\log(1/\epsilon))$&\\[5pt]


\hline
\cite{ye2018hessian}&Deterministic, centralized, strongly convex, unconstrained, smooth, Lipschitz Hessian&$p$&Linear&$\mathcal{O}(p\log(1/\epsilon))$&\\[5pt]


\hline
\multirow{2}{*}[-2pt]{\cite{kozak2020stochastic}}&Deterministic, centralized, unconstrained, nonconvex, smooth&\multirow{2}{*}[-1pt]{\parbox{1.1cm}{\centering Two}}&$\mathcal{O}(p/T)$&$\mathcal{O}(p/\epsilon)$&\\[5pt]
\cline{2-2}\cline{4-5}
&P--{\L} condition in addition &&Linear&$\mathcal{O}(p\log(1/\epsilon))$&\\[5pt]

\hline
\cite{cai2020zeroth}&Deterministic, centralized, restricted strongly convex, unconstrained, smooth, $s$-sparse gradient &$4s\log(p/s)$&Linear&$\mathcal{O}(s\log(p/s)\log(1/\epsilon))$&\\[5pt]

\hline
\rowcolor{LightGray}
\cite{shamir2013complexity}&Deterministic, centralized, quadratic, unconstrained, additive sampling noise&One&$\mathcal{O}(p^2/T)$&$\mathcal{O}(p^2/\epsilon)$&\\[5pt]

\hline
\rowcolor{LightGray}
\cite{bach2016highly}&Deterministic, centralized, strongly convex, unconstrained, smooth, additive sampling noise &Two &$\mathcal{O}(p/\sqrt{T})$&$\mathcal{O}(p^2/\epsilon^2)$&\\[5pt]

\hline
&Deterministic, centralized, unconstrained, nonconvex, Lipschitz, smooth&&$\mathcal{O}(p^2/T^{2/3})$&$\mathcal{O}(p^3/\epsilon^{3/2})$&\\[5pt]
\cline{2-2}\cline{4-5}
\rowcolor{LightCyan}
\multirow{2}{*}[9.5pt]{\cite{zhang2020improving}}&Stochastic, centralized, unconstrained, nonconvex, Lipschitz, smooth&\multirow{2}{*}[10.5pt]{\parbox{1.1cm}{\centering One}}&$\mathcal{O}(p^{4/3}/T^{1/3})$&$\mathcal{O}(p^4/\epsilon^{3})$&\\[5pt]


\hline
\rowcolor{LightCyan}
\cite{ghadimi2013stochastic,lian2016Comprehensive}&Stochastic, centralized, unconstrained, nonconvex, smooth &Two & $\mathcal{O}(\sqrt{p/T})$&$\mathcal{O}(p/\epsilon^2)$&\\[5pt]

\hline
\rowcolor{LightCyan}
\cite{balasubramanian2018zeroth}&Stochastic, centralized, unconstrained, nonconvex, smooth, $s$-sparse gradient &Two &$\mathcal{O}(s\log(p)/\sqrt{T})$ &$\mathcal{O}((s\log(p))^2/\epsilon^2)$&\\[5pt]

\hline
\rowcolor{LightCyan}
\cite{ghadimi2016mini}&Stochastic, centralized, constrained, nonconvex, Lipschitz, smooth& $\mathcal{O}(pT)$ &$\mathcal{O}(1/T)$&$\mathcal{O}(p/\epsilon^2)$&\\[5pt]


\hline
\rowcolor{LightCyan}
\cite{chen2019zo}&Stochastic, centralized, constrained, nonconvex, Lipschitz, smooth& Two &$\mathcal{O}(p/\sqrt{T})$ &$\mathcal{O}(p^2/\epsilon^2)$&\\[5pt]

\hline
\cite{liu2018zerothieee}&Deterministic, finite-sum, nonconvex, constrained, Lipschitz, smooth &$\mathcal{O}(\sqrt{T})$ &$\mathcal{O}(p/\sqrt{T})$ &$\mathcal{O}(p^{3}/\epsilon^3)$&\\[5pt]

\hline
\cite{liu2019signsgd}&Deterministic, finite-sum, nonconvex, unconstrained, Lipschitz, smooth&$\mathcal{O}(pT)$&$\mathcal{O}(\sqrt{p/T})$&$\mathcal{O}(p^3/\epsilon^4)$&\\[5pt]

\hline
\cite{gu2018faster}&Deterministic, finite-sum, nonconvex, unconstrained, Lipschitz, smooth, the original and mixture gradients are proportional &$2b$&$\mathcal{O}(pn^\theta/(bT))$ &$\mathcal{O}(pn^\theta/\epsilon)$, $\forall \theta\in(0,1)$&\\[5pt]

\hline
\cite{fang2018spider}&Deterministic, finite-sum, nonconvex, unconstrained, smooth, bounded variance &$\mathcal{O}(pn^{1/2})$ &$\mathcal{O}(1/T)$&$\mathcal{O}(pn^{1/2}/\epsilon)$&\\[5pt]

\hline
\cite{liu2018zeroth}&Deterministic, finite-sum, nonconvex, unconstrained, smooth, bounded variance &$2n$ &$\mathcal{O}(p/T)$ &$\mathcal{O}(pn/\epsilon)$&\\[5pt]

\hline
\cite{huang2019faster}&Deterministic, finite-sum, nonconvex, unconstrained, Lipschitz, smooth &$\mathcal{O}(pn^{2/3})$ &$\mathcal{O}(p/T)$ &$\mathcal{O}(p^2n^{2/3}/\epsilon)$&\\[5pt]

\hline
\multirow{2}{*}[-2pt]{\cite{pmlr-v97-ji19a}}&Deterministic, finite-sum, nonconvex, unconstrained, smooth, bounded variance&\multirow{2}{*}[-1pt]{\parbox{1.1cm}{\centering $\mathcal{O}(pn^{1/2})$}}&$\mathcal{O}(1/T)$&$\mathcal{O}(pn^{1/2}/\epsilon)$&\\[5pt]
\cline{2-2}\cline{4-5}
&P--{\L} condition in addition&&Linear&$\mathcal{O}(pn^{1/2}\log(1/\epsilon))$&\\[5pt]

\hline
\cite{chen2020accelerated}&Deterministic, finite-sum, strongly convex, unconstrained, smooth &Four &Linear &$\mathcal{O}(pn\log(p/\epsilon))$&\\[5pt]

\hline
\rowcolor{LightCyan}
\cite{kazemi2018proximal}&Stochastic, finite-sum, nonconvex, constrained, Lipschitz, smooth &$\mathcal{O}(nT)$ &$\mathcal{O}(p/T)$&$\mathcal{O}(p^2n/\epsilon^2)$&\\[5pt]

\hline
\rowcolor{LightCyan}
\cite{liu2018stochastic}&Stochastic, finite-sum, nonconvex, unconstrained, smooth &Four &$\mathcal{O}(p^{1/3}n^{2/3}/T)$ &$\mathcal{O}(p^{1/3}n^{2/3}/\epsilon)$&\\[3.2pt]

\hline
\cite{yuan2014randomized}&Deterministic, distributed, convex, constrained, Lipschitz &$2n$ &Asymptotic& ---- &\\[5pt]

\hline
\cite{yuan2015gradient}&Deterministic, distributed, convex, unconstrained, Lipschitz&$2n$&$\mathcal{O}(p^3n^2/\sqrt{T})$&$\mathcal{O}(p^6n^5/\epsilon^2)$&\\[5pt]

\hline
\multirow{2}{*}[-2pt]{\cite{yu2019distributed}}&Deterministic, distributed, convex, compact constrained, Lipschitz&\multirow{2}{*}[-1pt]{\parbox{1.0cm}{\centering $2n$}}&$\mathcal{O}(p\sqrt{n/T})$&$\mathcal{O}(p^2n^2/\epsilon^2)$&\\[5pt]
\cline{2-2}\cline{4-5}
&Deterministic, distributed, strongly convex, constrained, Lipschitz&&$\mathcal{O}(p^2n^2/T)$&$\mathcal{O}(p^2n^3/\epsilon)$&\\[5pt]

\hline
\rowcolor{LightGray}
\cite{sahu2018distributed}&Deterministic, distributed, strongly convex, unconstrained, smooth, additive sampling noise &$2n$ &$\mathbb{O}(pn^2/\sqrt{T})$&$\mathcal{O}(p^2n^5/\epsilon^2)$&\\[5pt]

\hline
\rowcolor{LightGray}
\cite{wang2019distributed}&Deterministic, distributed, convex, compact constrained, Lipschitz, additive sampling noise &$2n$ &$\mathcal{O}(1/\sqrt{T})$&$\mathcal{O}(n/\epsilon^2)$&\\[5pt]

\hline
\rowcolor{LightCyan}
\cite{beznosikov2019derivative}&Stochastic, distributed, convex, compact constrained, Lipschitz &$\mathcal{O}(pnT)$ &$\mathcal{O}(1/T)$&$\mathcal{O}(pn/\epsilon^2)$&\\[5pt]

\hline
\multirow{4}{*}[-6pt]{\cite{tang2020distributedzero}}&Deterministic, distributed, nonconvex, unconstrained, Lipschitz, smooth&\multirow{2}{*}[-5pt]{\parbox{1.1cm}{\centering $2n$}}&$\mathcal{O}(\sqrt{p/T})$&$\mathcal{O}(pn/\epsilon^2)$&\\[5pt]
\cline{2-2}\cline{4-5}
&Deterministic, distributed, nonconvex, unconstrained, smooth, P--{\L} condition&&$\mathcal{O}(p/T)$&$\mathcal{O}(pn/\epsilon)$&\\[5pt]
\cline{2-5}
&Deterministic, distributed, nonconvex, unconstrained, smooth&\multirow{2}{*}[-5pt]{\parbox{1.1cm}{\centering $2pn$}}&$\mathcal{O}(1/T)$&$\mathcal{O}(pn/\epsilon)$&\\[5pt]
\cline{2-2}\cline{4-5}
&Deterministic, distributed, nonconvex, unconstrained, smooth, P--{\L} condition&&Linear&$\mathcal{O}(pn\log(1/\epsilon))$&\\[5pt]

\hline
\multirow{2}{*}[-2.5pt]{\cite{yi2021linear}}&Deterministic, distributed, nonconvex, unconstrained, smooth&\multirow{2}{*}[-2.5pt]{\parbox{1.1cm}{\centering $(p+1)n$}}&$\mathcal{O}(1/T)$&$\mathcal{O}(pn/\epsilon)$&\\[5pt]
\cline{2-2}\cline{4-5}
&P--{\L} condition in addition (without using the P--{\L} constant) &&Linear&$\mathcal{O}(pn\log(1/\epsilon))$&\\[5pt]

\hline
\rowcolor{LightCyan}
\cite{hajinezhad2019zone}&Stochastic, distributed, nonconvex, unconstrained, Lipschitz, smooth&$\mathcal{O}(nT)$&$\mathcal{O}(p^2n/T)$&$\mathcal{O}(p^4n^3/\epsilon^2)$&\\[5pt]

\hline
\rowcolor{LightCyan}
&Stochastic, distributed, nonconvex, unconstrained, smooth, state-dependent variance &&$\mathcal{O}(\sqrt{p/(nT)})$&$\mathcal{O}(p/\epsilon^2)$&\\[5pt]
\cline{2-2}\cline{4-5}
\rowcolor{LightCyan}
&Stochastic, distributed, nonconvex, unconstrained, smooth, P--{\L} condition&\multirow{2}{*}[10.5pt]{\parbox{1.1cm}{\centering $2n$}}&$\mathcal{O}(p/(nT))$&$\mathcal{O}(p/\epsilon)$&\\[5pt]
\cline{2-5}
\rowcolor{LightCyan}
&Stochastic, distributed, nonconvex, unconstrained, smooth, relative bounded second moment &&$\mathcal{O}(p/T)$&$\mathcal{O}(np/\epsilon)$&\\[5pt]
\cline{2-2}\cline{4-5}
\rowcolor{LightCyan}
\multirow{4}{*}[30pt]{\parbox{0.8cm}{\centering This paper}}&P--{\L} condition in addition (without using the P--{\L} constant) &\multirow{2}{*}[10.5pt]{\parbox{1.1cm}{\centering $2n$}} &Linear &$\mathcal{O}(pn\log(1/\epsilon))$&\\[5pt]

\hline
\end{tabular}
}

\end{center}
\vskip -0.1in
\end{table*}

\subsection{Outline}
The rest of this paper is organized as follows.
Section~\ref{zerosg:sec-preliminary} introduces some preliminaries. Sections~\ref{zerosg:sec-main-random} and \ref{zerosg-p:sec-main-random} provide the distributed primal--dual and primal ZO algorithms, respectively, and analyze their convergence properties.  Numerical evaluations for an image classification problem from the literature are given in Section~\ref{zerosg:sec-simulation}. Finally, concluding remarks are offered in Section~\ref{zerosg:sec-conclusion}. All the proofs are given in Appendix.

\noindent {\bf Notations}: $\mathbb{N}_+$ denotes the set of positive integers. $[n]$ denotes the set $\{1,\dots,n\}$ for any $n\in\mathbb{N}_+$.
$\|\cdot\|$ represents the Euclidean norm for vectors or the induced 2-norm for matrices.
$\mathbb{B}^p$ and $\mathbb{S}^p$ are the unit ball and sphere centered around the origin in $\mathbb{R}^p$ under Euclidean norm, respectively. Given a differentiable function $f$, $\nabla f$ denotes its gradient.

\section{Preliminaries}\label{zerosg:sec-preliminary}
In this section, we introduce the P--{\L} condition, the random gradient
estimator, and the assumptions used in this paper.

\subsection{Polyak--{\L}ojasiewicz Condition}
\begin{definition} \cite{karimi2016linear}
A differentiable  function $f(x):~\mathbb{R}^p\mapsto\mathbb{R}$ satisfies the Polyak--{\L}ojasiewicz (P--{\L}) condition with constant  $\nu>0$ if $f^*>-\infty$, where $f^*=\min_{x\in\mathbb{R}^p}f(x)$, and
\begin{align}
\frac{1}{2}\|\nabla f(x)\|^2\ge \nu( f(x)-f^*),~\forall x\in \mathbb{R}^p.\label{nonconvex:equ:plc}
\end{align}
\end{definition}
It is straightforward to see that every (essentially, weakly, or restricted) strongly convex function satisfies the P--{\L} condition.
The P--{\L} condition implies that every stationary point is a global minimizer. But unlike (essentially, weakly, or restricted) strong convexity, the P--{\L} condition  alone does not imply convexity of $f$. Moreover, it does not imply that the set of global minimizers is a singleton \cite{karimi2016linear,zhang2015restricted}.
Examples of nonconvex functions which satisfy the P--{\L} condition can be found in \cite{karimi2016linear,zhang2015restricted}.


\subsection{Gradient Estimator}
Let $f(x):~\mathbb{R}^p\mapsto\mathbb{R}$ be a function. The authors of \cite{duchi2015optimal} proposed the following random gradient estimator:
\begin{align}
\hat{\nabla}_2f(x,\delta,u)=\frac{p}{\delta}(f(x+\delta u)-f(x))u,\label{dbco:gradient:model2}
\end{align}
where $\delta>0$ is the smoothing/exploration parameter and $u\in\mathbb{S}^p$ is a uniformly distributed random vector. This gradient estimator can be calculated by sampling the function $f$ at two points (e.g., $x$ and $x+\delta u$). The intuition of this estimator follows from directional derivatives \cite{duchi2015optimal}. From a practical point of view, the larger the smoothing parameter $\delta$ the better, since in this case it is easier to distinguish the two sampled function values.

\subsection{Assumptions}
The following assumptions are made.

\begin{assumption}\label{zerosg:ass:graph}
The undirected graph $\mathcal G$ is connected.
\end{assumption}

\begin{assumption}\label{zerosg:ass:optset}
The optimal set $\mathbb{X}^*$ is nonempty and $f^*>-\infty$, where $\mathbb{X}^*$ and $f^*$ are the optimal set and the minimum function value of the optimization problem \eqref{zerosg:eqn:xopt}, respectively.
\end{assumption}


\begin{assumption}\label{zerosg:ass:zeroth-smooth}
For almost all $\xi_i$, the stochastic ZO oracle $F_i(\cdot,\xi_i)$ is smooth with constant $L_f>0$.
\end{assumption}

\begin{assumption}\label{zerosg:ass:zeroth-variance}
Each stochastic gradient $\nabla_xF_i(x,\xi_i)$ has state-dependent variance, i.e., there exist two constants $\sigma_0$ and $\sigma_1$ such that $\mathbf{E}_{\xi_i}[\|\nabla_xF_i(x,\xi_i)-\nabla f_i(x)\|^2]\le\sigma^2_0\|\nabla f_i(x)\|^2+\sigma^2_1,~\forall i\in[n],~\forall x\in\mathbb{R}^p$.
\end{assumption}

\begin{assumption}\label{zerosg:ass:fig}
Each local gradient $\nabla f_i(x)$ has state-dependent variance, i.e.,
there exists two constants $\tilde{\sigma}_0$ and $\sigma_2$ such that $\|\nabla f_i(x)-\nabla f(x)\|^2\le\tilde{\sigma}^2_0\|\nabla f(x)\|^2+\sigma^2_2,~\forall i\in[n],~\forall x\in\mathbb{R}^p$. Here $\nabla f_i(x)$ can be viewed as a stochastic gradient of $\nabla f(x)$ by randomly picking an index $i\in[n]$.
\end{assumption}

\begin{assumption}\label{zerosg:ass:fil} The global cost function $f(x)$ satisfies the P--{\L} condition with constant $\nu>0$.
\end{assumption}

\begin{remark}
It should be highlighted that no convexity assumptions are made.
Assumption~\ref{zerosg:ass:graph} is common in distributed optimization, e.g., \cite{shi2015extra,nedic2017geometrically,qu2018harnessing,
qu2017accelerated,hajinezhad2019zone,beznosikov2019derivative,tang2020distributedzero}. Assumption~\ref{zerosg:ass:optset} is basic. Assumption~\ref{zerosg:ass:zeroth-smooth} is standard in stochastic optimization with ZO information feedback, e.g., \cite{ghadimi2013stochastic,ghadimi2016mini,lian2016Comprehensive,gao2018information,
balasubramanian2018zeroth,gorbunov2018accelerated,liu2018stochastic,kazemi2018proximal,
hajinezhad2019zone}. When $\sigma_0=0$, Assumption~\ref{zerosg:ass:zeroth-variance} recovers the bounded variance assumption, which is commonly used in the literature studying stochastic ZO optimization, e.g., \cite{ghadimi2013stochastic,balasubramanian2018zeroth,lian2016Comprehensive,ghadimi2016mini,gao2018information,
kazemi2018proximal,gorbunov2018accelerated,sahu2019towards,hajinezhad2019zone}. Therefore, Assumption~\ref{zerosg:ass:zeroth-variance} is more general.
When $\tilde{\sigma}_0=0$, Assumption~\ref{zerosg:ass:fig} becomes the bounded variance assumption, i.e., $\|\nabla f_i(x)-\nabla f(x)\|^2$ is globally bounded, and is weaker than the Lipschitz assumption, i.e., $\|\nabla f_i(x)\|$ is globally bounded. Both the bounded variance and Lipschitz assumptions  are normally used in the literature studying ZO optimization, e.g., \cite{fang2018spider,liu2018zeroth,pmlr-v97-ji19a} and \cite{duchi2015optimal,yuan2014randomized,gu2018faster,yuan2015gradient,liu2018zerothieee,kazemi2018proximal,
hajinezhad2019zone,liu2019signsgd,huang2019faster,gao2018information,
huang2019zeroth,huang2019nonconvex,tang2020distributedzero}, respectively. However, the Lipschitz assumption is too restricted since the simple quadratic functions are not Lipschitz. Moreover, the bounded variance assumption is also restricted, for instance it is impractical to assume this assumption for distributed learning problems with local cost functions being constructed by heterogenous data collected locally by agents. In contrast, Assumption~\ref{zerosg:ass:fig} is more general due to the state-dependent term $\tilde{\sigma}^2_0\|\nabla f(x)\|^2$, and it is not needed when Assumption~\ref{zerosg:ass:fil} holds and the constant $\nu$ is known in advance as shown in Theorems~\ref{zerosg-a5:thm-sg-diminishingt} and \ref{zerosg-p-a5:thm-sg-diminishingt} in the later sections. Assumption~\ref{zerosg:ass:fil}  is  weaker than the assumption that the global or each local cost function is (restricted) strongly convex. It plays a key role to guarantee that a global optimum can be found and to show that faster convergence rate can be achieved.
\end{remark}

To end this section, we introduce the stronger alternatives of Assumptions~\ref{zerosg:ass:zeroth-variance} and \ref{zerosg:ass:fig}, which are the key to show faster convergence for the proposed algorithms as shown in the later sections.

\begin{assumptionp}{\ref{zerosg:ass:zeroth-variance}$'$}
The second moment of each stochastic gradient $\nabla_xF_i(x,\xi_i)$ is relative bounded, i.e., there exists a constant $\breve{\sigma}_0$ such that $\mathbf{E}_{\xi_i}[\|\nabla_xF_i(x,\xi_i)\|^2]\le\breve{\sigma}^2_0\|\nabla f_i(x)\|^2,~\forall i\in[n],~\forall x\in\mathbb{R}^p$.
\end{assumptionp}

\begin{assumptionp}{\ref{zerosg:ass:fig}$'$}
The second moment of each local gradient $\nabla f_i(x)$ is relative bounded, i.e.,
there exists a constant $\hat{\sigma}_0$ such that $\|\nabla f_i(x)\|^2\le\hat{\sigma}^2_0\|\nabla f(x)\|^2,~\forall i\in[n],~\forall x\in\mathbb{R}^p$.
\end{assumptionp}

It is straightforward to check that Assumption~\ref{zerosg:ass:zeroth-variance}$'$ (Assumption~\ref{zerosg:ass:fig}$'$) is equivalent to  Assumption~\ref{zerosg:ass:zeroth-variance} (Assumption~\ref{zerosg:ass:fig}) when $\sigma_1=0$ ($\sigma_2=0$). Assumption~\ref{zerosg:ass:zeroth-variance}$'$ is satisfied trivially when the deterministic ZO information is available. Assumption~\ref{zerosg:ass:fig}$'$ holds when $\nabla f_i(x)$ is proportional to $\nabla f(x)$, for example when all the random variables $\xi_i$ have a common probability distribution and the local stochastic  component functions are the same, which is a common setup in distributed empirical risk minimization problems. Moreover, for deterministic centralized optimization problems, Assumption~\ref{zerosg:ass:zeroth-variance}$'$ and \ref{zerosg:ass:fig}$'$ hold trivially.

\section{Distributed ZO primal--dual Algorithm}\label{zerosg:sec-main-random}
In this section, we propose a distributed ZO primal--dual algorithm and analyze its convergence properties.

When gradient information is available, in \cite{yi2021linear} the following distributed first-order primal--dual algorithm was proposed to solve \eqref{zerosg:eqn:xopt}:
\begin{subequations}\label{nonconvex:kia-algo-dc}
\begin{align}
x_{i,k+1} &= x_{i,k}-\eta\Big(\alpha\sum_{j=1}^nL_{ij}x_{j,k}+\beta v_{i,k}+\nabla f_i(x_{i,k})\Big), \label{nonconvex:kia-algo-dc-x}\\
v_{i,k+1} &=v_{i,k}+ \eta\beta\sum_{j=1}^n L_{ij}x_{j,k},~\forall x_{i,0}\in\mathbb{R}^p, ~\sum_{j=1}^nv_{j,0}={\bm 0}_p,~
\forall i\in[n], \label{nonconvex:kia-algo-dc-q}
\end{align}
\end{subequations}
where $\alpha$, $\beta$, and $\eta$ are positive algorithm parameters, and $L=[L_{ij}]$ is the weighted Laplacian matrix associated with the undirected communication graph $\calG$. As pointed out in \cite{yi2021linear} the distributed first-order algorithm \eqref{nonconvex:kia-algo-dc} is a special form of several existing first-order algorithms in the literature, e.g., \cite{shi2015extra,jakovetic2020primal}, and it has been shown that this algorithm can find a stationary point with an $\mathcal{O}(1/k)$ convergence rate.

Noting that we consider the scenario where only stochastic ZO oracles rather than the explicit expressions of the gradients are available, we need to estimate the gradients used in the distributed first-order algorithm \eqref{nonconvex:kia-algo-dc}. Inspired by \eqref{dbco:gradient:model2}, we introduce
\begin{align}
&g^e_{i,k}=\frac{p(F_i(x_{i,k}+\delta_{i,k} u_{i,k},\xi_{i,k})-F_i(x_{i,k},\xi_{i,k}))}{\delta_{i,k}}u_{i,k},
\label{dbco:gradient:model2-st}
\end{align}
where $\delta_{i,k}>0$ is a time-varying smoothing parameter and $u_{i,k}\in\mathbb{S}^p$ is a uniformly distributed random vector chosen by agent $i$ at iteration $k$; $\xi_{i,k}$ is a random variable sampled by agent $i$ at iteration $k$ according to the distribution of $\xi_i$; and $F_i(x_{i,k}+\delta_{i,k} u_{i,k},\xi_{i,k})$ and $F_i(x_{i,k},\xi_{i,k})$ are the values sampled by agent $i$ at iteration $k$.
We replace the gradient and fixed algorithm parameters in \eqref{nonconvex:kia-algo-dc} with the stochastic gradient estimator \eqref{dbco:gradient:model2-st} and time-varying parameters, respectively. Then we get the following ZO algorithm:
\begin{subequations}\label{zerosg:alg:random-pd}
\begin{align}
x_{i,k+1} &= x_{i,k}-\eta_k\Big(\alpha_k\sum_{j=1}^nL_{ij}x_{j,k}+\beta_k v_{i,k}+g^e_{i,k}\Big), \label{zerosg:alg:random-pd-x}\\
v_{i,k+1} &=v_{i,k}+ \eta_k\beta_k\sum_{j=1}^n L_{ij}x_{j,k},~\forall x_{i,0}\in\mathbb{R}^p, ~\sum_{j=1}^nv_{j,0}={\bm 0}_p,~
\forall i\in[n].  \label{zerosg:alg:random-pd-q}
\end{align}
\end{subequations}

We write the distributed ZO algorithm \eqref{zerosg:alg:random-pd} in pseudo-code as Algorithm~\ref{zerosg:algorithm-random-pd}. In this algorithm, from the way to generate $u_{i,k}$ and $\xi_{i,k}$, we know that $u_{i,k},~\xi_{j,l},~\forall i,j\in[n],~k,l\in\mathbb{N}_+$ are mutually independent. Let $\mathfrak{L}_k$ denote the $\sigma$-algebra generated by the independent random variables $u_{1,k},\dots,u_{n,k},\xi_{1,k},\dots,\xi_{n,k}$ and let $\calL_k=\bigcup_{t=0}^{k}\mathfrak{L}_t$. From the independence property of $u_{i,k}$ and $\xi_{i,l}$, we can see that $x_{i,k}$ and $v_{i,k+1},~i\in[n]$ depend on $\calL_{k-1}$ and are independent of $\mathfrak{L}_t$ for all $t\ge k$.
\begin{algorithm}[tb]
\caption{Distributed  ZO Primal--Dual Algorithm}
\label{zerosg:algorithm-random-pd}
\begin{algorithmic}[1]

\STATE \textbf{Input}: positive sequences $\{\alpha_k\}$, $\{\beta_k\}$, $\{\eta_k\}$, and $\{\delta_{i,k}\}$.
\STATE \textbf{Initialize}: $ x_{i,0}\in\mathbb{R}^p$ and $v_{i,0}={\bm 0}_p,~
\forall i\in[n]$.
\FOR{$k=0,1,\dots$}
\FOR{$i=1,\dots,n$  in parallel}
\STATE  Broadcast $x_{i,k}$ to $\mathcal{N}_i$ and receive $x_{j,k}$ from $j\in\mathcal{N}_i$;
\STATE Generate $u_{i,k}\in\mathbb{S}^{p}$ independently and uniformly at random;
\STATE Generate $\xi_{i,k}$ independently and randomly according to the distribution of $\xi_i$;
\STATE Sample $F_i(x_{i,k},\xi_{i,k})$ and $F_i(x_{i,k}+\delta_{i,k}u_{i,k},\xi_{i,k})$;
\STATE  Update $x_{i,k+1}$ by \eqref{zerosg:alg:random-pd-x};
\STATE  Update $v_{i,k+1}$ by \eqref{zerosg:alg:random-pd-q}.
\ENDFOR
\ENDFOR
\STATE  \textbf{Output}: $\{\bsx_{k}\}$.
\end{algorithmic}
\end{algorithm}

\begin{remark}
In Algorithm~\ref{zerosg:algorithm-random-pd}, each agent $i$ maintains two local sequences, i.e., the local primal and dual variable sequences $\{x_{i,k}\}$ and $\{v_{i,k}\}$, and communicates its local primal variables to its neighbors through the network. Moreover, at each iteration each agent samples its local stochastic ZO oracle at two  points to estimate the gradient of its local cost function. It should be highlighted that the agent-wise smoothing parameter $\delta_{i,k}$ is time-varying. It can in many situations be chosen larger than the fixed smoothing parameter used in existing ZO algorithms. For example, in the following we use an $\mathcal{O}(1/k^{1/4})$ smoothing parameter, which is larger than the $\mathcal{O}(1/T^{1/2})$ smoothing parameter used in \cite{ghadimi2013stochastic}.
\end{remark}

\subsection{Find stationary points}
Let us consider the case when Algorithm~\ref{zerosg:algorithm-random-pd} is able to find stationary points. We first have the following convergence result.
\begin{theorem}\label{zerosg:thm-random-pd-sm}
Suppose Assumptions~\ref{zerosg:ass:graph}--\ref{zerosg:ass:fig} hold. Let $\{\bsx_k\}$ be the sequence generated by Algorithm~\ref{zerosg:algorithm-random-pd} with
\begin{align}\label{zerosg:step:eta1-sm}
&\alpha_k=\kappa_1\beta_k,~\beta_k=\kappa_0(k+t_1)^{\theta},
~ \eta_k=\frac{\kappa_2}{\beta_k},~\delta_{i,k}\le \frac{\kappa_\delta\sqrt{p\eta_k}}{\sqrt{n+p}},~\forall k\in\mathbb{N}_0,
\end{align}
where $\kappa_1>c_1$, $\kappa_2\in(0,c_2(\kappa_1))$, $\theta\in(0.5,1)$, $t_1\ge(\sqrt{p}c_3(\kappa_1,\kappa_2))^{\frac{1}{\theta}}$, $\kappa_0\ge \frac{c_0(\kappa_1,\kappa_2)}{t_1^\theta}$, and $\kappa_\delta>0$ with $c_0(\kappa_1,\kappa_2)$, $c_1$, $c_2(\kappa_1)$, and $c_3(\kappa_1,\kappa_2)$ being given in Appendix~\ref{zerosg:proof-thm-random-pd-sm}.
Then, for any $T\in\mathbb{N}_+$,
\begin{subequations}
\begin{align}
&\frac{1}{T}\sum_{k=0}^{T-1}\mathbf{E}[\|\nabla f(\bar{x}_k)\|^2]
=\mathcal{O}\Big(\frac{\sqrt{p}}{T^{1-\theta}}+\frac{p}{T}\Big),\label{zerosg:thm-sg-sm-equ1}\\
&\mathbf{E}[f(\bar{x}_{T})]-f^*=\mathcal{O}(1),\label{zerosg:thm-sg-sm-equ2}\\
&\mathbf{E}\Big[\frac{1}{n}\sum_{i=1}^{n}\|x_{i,T}-\bar{x}_T\|^2\Big]
=\mathcal{O}\Big(\frac{1}{T^{2\theta}}\Big), \label{zerosg:thm-sg-sm-equ1.1bounded}\\
&\lim_{T\rightarrow+\infty}\mathbf{E}[\|\nabla f(\bar{x}_T)\|^2]=0,\label{zerosg:thm-sg-sm-equ1bounded}
\end{align}
\end{subequations}
where $\bar{x}_k=\frac{1}{n}\sum_{i=1}^{n}x_{i,k}$.
\end{theorem}
\begin{proof}
The the proof is given in Appendix~\ref{zerosg:proof-thm-random-pd-sm}.
\end{proof}

If the total number of iterations $T$ and the number of agents $n$ are known in advance, then, as shown in the following, Algorithm~\ref{zerosg:algorithm-random-pd} can find a stationary point of \eqref{zerosg:eqn:xopt} with an $\mathcal{O}(\sqrt{p/(nT)})$ convergence rate, and thus achieves linear speedup with respect to the number of agents compared to the $\mathcal{O}(\sqrt{p/T})$ convergence rate achieved by the centralized stochastic ZO algorithms in \cite{ghadimi2013stochastic,lian2016Comprehensive}. The linear speedup property enables us to scale up the computing capability by adding more agents into the algorithm \cite{Yu2019on}.
\begin{theorem}[Linear speedup]\label{zerosg:thm-sg-smT}
Suppose Assumptions~\ref{zerosg:ass:graph}--\ref{zerosg:ass:fig} hold. For any given $T\ge\max\{\frac{n(\tilde{c}_0(\kappa_1,\kappa_2))^2}{p\kappa_2^2},~\frac{n^3}{p}\}$, let $\{\bsx_k,k=0,\dots,T\}$ be the output generated by Algorithm~\ref{zerosg:algorithm-random-pd} with
\begin{align}\label{zerosg:step:eta2-sm}
&\alpha_k=\kappa_1\beta_k,~\beta_k=\beta=\frac{\kappa_2\sqrt{pT}}{\sqrt{n}},~ \eta_k=\frac{\kappa_2}{\beta_k},
~\delta_{i,k}\le\frac{p^{\frac{1}{4}}n^{\frac{1}{4}}\kappa_\delta}
{\sqrt{n+p}(k+1)^{\frac{1}{4}}},~\forall k\le T,
\end{align}
where $\tilde{c}_0(\kappa_1,\kappa_2)$ is given in Appendix~\ref{zerosg:proof-thm-sg-smT}, $\kappa_1>c_1$, $\kappa_2\in(0,c_2(\kappa_1))$,  and $\kappa_\delta>0$ with $c_1$ and $c_2(\kappa_1)$ being given in Appendix~\ref{zerosg:proof-thm-random-pd-sm}.
Then,
\begin{subequations}
\begin{align}
&\frac{1}{T}\sum_{k=0}^{T-1}\mathbf{E}[\|\nabla f(\bar{x}_k)\|^2]
=\mathcal{O}\Big(\frac{\sqrt{p}}{\sqrt{nT}}\Big)+\mathcal{O}\Big(\frac{n}{T}\Big),\label{zerosg:coro-sg-sm-equ3}\\
&\mathbf{E}[f(\bar{x}_{T})]-f^*=\mathcal{O}(1),\label{zerosg:coro-sg-sm-equ4}\\
&\mathbf{E}\Big[\frac{1}{n}\sum_{i=1}^{n}\|x_{i,T}-\bar{x}_T\|^2\Big]
=\mathcal{O}\Big(\frac{n}{T}\Big),\label{zerosg:coro-sg-sm-equ3.1}\\
&\lim_{T\rightarrow+\infty}\mathbf{E}[\|\nabla f(\bar{x}_T)\|^2]=0.\label{zerosg:coro-sg-sm-equ3bounded}
\end{align}
\end{subequations}
\end{theorem}
\begin{proof}
The proof is given in Appendix~\ref{zerosg:proof-thm-sg-smT}. It should be highlighted that the omitted constants in the first term on the right-hand side of \eqref{zerosg:coro-sg-sm-equ3} do not depend on any parameters related to the communication network.
\end{proof}

\begin{remark}
To the best of our knowledge, Theorem~\ref{zerosg:thm-sg-smT} is the first result to establish linear speedup for a distributed ZO algorithm to solve stochastic optimization problems. The achieved rate is faster than that achieved by the centralized ZO algorithms in \cite{ghadimi2013stochastic,lian2016Comprehensive,zhang2020improving,liu2018zerothieee,
liu2019signsgd,balasubramanian2018zeroth,chen2019zo} and the distributed ZO gradient descent algorithm in \cite{tang2020distributedzero}. The rate is slower than that achieved by the centralized ZO algorithms in \cite{gu2018faster,fang2018spider,liu2018zeroth,
liu2018stochastic,huang2019faster,pmlr-v97-ji19a,ghadimi2016mini,kazemi2018proximal}, which is reasonable since these algorithms not only are centralized but also use variance reduction techniques.
The distributed ZO gradient tracking algorithm in \cite{tang2020distributedzero} and the distributed ZO primal--dual algorithms in \cite{hajinezhad2019zone,yi2021linear} also achieved faster convergence rates than ours. However, in \cite{gu2018faster,fang2018spider,liu2018zeroth,
huang2019faster,pmlr-v97-ji19a,yi2021linear,tang2020distributedzero}, the considered problems are deterministic;  in \cite{yi2021linear,tang2020distributedzero}, the sampling size of each agent at each iteration is $\mathcal{O}(p)$, which results in a heavy sampling burden when $p$ is large; in \cite{ghadimi2016mini,kazemi2018proximal,hajinezhad2019zone}, the sampling size of each agent at each iteration is $\mathcal{O}(T)$, which is difficult to execute in practice. One of our future research directions is to establish faster convergence with reduced sampling complexity by using variance reduction techniques.
\end{remark}

\subsection{Find global optimum}
Let us next consider cases when Algorithm~\ref{zerosg:algorithm-random-pd} finds global optimum.

\begin{theorem}\label{zerosg:thm-sg-diminishing}
Suppose Assumptions~\ref{zerosg:ass:graph}--\ref{zerosg:ass:fil} hold. Let $\{\bsx_k\}$ be the sequence generated by Algorithm~\ref{zerosg:algorithm-random-pd} with
\begin{align}\label{zerosg:step:eta1}
&\alpha_k=\kappa_1\beta_k,~\beta_k=\kappa_0(k+t_1)^{\theta},
~ \eta_k=\frac{\kappa_2}{\beta_k},~\delta_{i,k}\le \frac{\kappa_\delta\sqrt{p\eta_k}}{\sqrt{n+p}},~\forall k\in\mathbb{N}_0,
\end{align}
where $\kappa_1>c_1$, $\kappa_2\in(0,c_2(\kappa_1))$, $\theta\in(0,1)$, $t_1\in[(pc_3(\kappa_1,\kappa_2))^{\frac{1}{\theta}},(pc_4c_3(\kappa_1,\kappa_2))^{\frac{1}{\theta}}]$, $\kappa_0\ge \frac{c_0(\kappa_1,\kappa_2)}{t_1^\theta}$, $\kappa_\delta>0$, and $c_4\ge1$ with $c_0(\kappa_1,\kappa_2)$, $c_1$, $c_2(\kappa_1)$, and $c_3(\kappa_1,\kappa_2)$ being given in Appendix~\ref{zerosg:proof-thm-random-pd-sm}.
Then, for any $T\in\mathbb{N}_+$,
\begin{subequations}
\begin{align}
&\mathbf{E}\Big[\frac{1}{n}\sum_{i=1}^{n}\|x_{i,T}-\bar{x}_T\|^2\Big]
=\mathcal{O}\Big(\frac{p}{T^{2\theta}}\Big), \label{zerosg:thm-sg-diminishing-equ1.1bounded}\\
&\mathbf{E}[f(\bar{x}_{T})-f^*]
=\mathcal{O}\Big(\frac{p}{nT^\theta}\Big)+\mathcal{O}\Big(\frac{p}{T^{2\theta}}\Big).
\label{zerosg:thm-sg-diminishing-equ1bounded}
\end{align}
\end{subequations}
\end{theorem}
\begin{proof}
The proof is given in Appendix~\ref{zerosg:proof-thm-sg-diminishing}. It should be highlighted that the omitted constants in the first term on the right-hand side of \eqref{zerosg:thm-sg-diminishing-equ1bounded} do not depend on any parameters related to the communication network.
\end{proof}

From Theorem~\ref{zerosg:thm-sg-diminishing}, we see that the convergence rate is strictly slower than $\mathcal{O}(p/(nT))$. In the following we show that the $\mathcal{O}(p/(nT))$ convergence rate can be achieved if the P--{\L} constant $\nu$ is known in advance. Information about the total number of iterations $T$ is not needed.
\begin{theorem}[Linear speedup]\label{zerosg:thm-sg-diminishingt}
Suppose Assumptions~\ref{zerosg:ass:graph}--\ref{zerosg:ass:fil} hold and the P--{\L} constant $\nu$ is known in advance. Let $\{\bsx_{k}\}$ be the sequence generated by Algorithm~\ref{zerosg:algorithm-random-pd} with
\begin{align}\label{zerosg:step:eta1t1}
&\alpha_k=\kappa_1\beta_k,~\beta_k=\kappa_0(k+t_1),~ \eta_k=\frac{\kappa_2}{\beta_k},~\delta_{i,k}\le \frac{\kappa_\delta\sqrt{p\eta_k}}{\sqrt{n+p}},~\forall k\in\mathbb{N}_0,
\end{align}
where $\kappa_1>c_1$, $\kappa_2\in(0,c_2(\kappa_1))$, $\kappa_0\in[\frac{3\hat{c}_0\nu\kappa_2}{16},\frac{3\nu\kappa_2}{16})$, $t_1\ge\hat{c}_3(\kappa_0,\kappa_1,\kappa_2)$, $\kappa_\delta>0$, and $\hat{c}_0\in(0,1)$ with $c_1$ and $c_2(\kappa_1)$ being given in Appendix~\ref{zerosg:proof-thm-random-pd-sm}, and $\hat{c}_3(\kappa_0,\kappa_1,\kappa_2)$ being given in Appendix~\ref{zerosg:proof-thm-sg-diminishingt}.
Then, for any $T\in\mathbb{N}_+$,
\begin{subequations}
\begin{align}
&\mathbf{E}\Big[\frac{1}{n}\sum_{i=1}^{n}\|x_{i,T}-\bar{x}_T\|^2\Big]
=\mathcal{O}\Big(\frac{p}{T^{2}}\Big), \label{zerosg:thm-sg-diminishing-equ2.1bounded}\\
&\mathbf{E}[f(\bar{x}_{T})-f^*]
=\mathcal{O}\Big(\frac{p}{nT}\Big)+\mathcal{O}\Big(\frac{p}{T^{2}}\Big).
\label{zerosg:thm-sg-diminishing-equ2bounded}
\end{align}
\end{subequations}
\end{theorem}
\begin{proof}
The proof is given  in Appendix~\ref{zerosg:proof-thm-sg-diminishingt}. It should be highlighted that the omitted constants in the first term on the right-hand side of \eqref{zerosg:thm-sg-diminishing-equ2bounded} do not depend on any parameters related to the communication network.
\end{proof}

Although Assumption~\ref{zerosg:ass:fig} is weaker than the bounded gradient assumption, it can be further relaxed by a mild assumption. Specifically, if each $f_i^*>-\infty$, where $f_i^*=\min_{x\in\mathbb{R}^p}f_i(x)$, then without Assumption~\ref{zerosg:ass:fig}, the convergence results stated in \eqref{zerosg:thm-sg-diminishing-equ2.1bounded} and \eqref{zerosg:thm-sg-diminishing-equ2bounded} still hold, as shown in the following.
\begin{theorem}[Linear speedup]\label{zerosg-a5:thm-sg-diminishingt}
Suppose Assumptions~\ref{zerosg:ass:graph}--\ref{zerosg:ass:zeroth-variance} and \ref{zerosg:ass:fil} hold, and the P--{\L} constant $\nu$ is known in advance, and each $f_i^*>-\infty$. Let $\{\bsx_{k}\}$ be the sequence generated by Algorithm~\ref{zerosg:algorithm-random-pd} with
\begin{align}\label{zerosg-a5:step:eta1t1}
&\alpha_k=\kappa_1\beta_k,~\beta_k=\kappa_0(k+t_1),~ \eta_k=\frac{\kappa_2}{\beta_k},~\delta_{i,k}\le \frac{\kappa_\delta\sqrt{p\eta_k}}{\sqrt{n+p}},~\forall k\in\mathbb{N}_0,
\end{align}
where $\kappa_1>c_1$, $\kappa_2\in(0,c_2(\kappa_1))$, $\kappa_0\in[\frac{3\hat{c}_0\nu\kappa_2}{16},\frac{3\nu\kappa_2}{16})$, $t_1\ge\check{c}_3(\kappa_0,\kappa_1,\kappa_2)$, $\kappa_\delta>0$, and $\hat{c}_0\in(0,1)$ with $c_1$ and $c_2(\kappa_1)$ being given in Appendix~\ref{zerosg:proof-thm-random-pd-sm}, and $\check{c}_3(\kappa_0,\kappa_1,\kappa_2)$ being given in Appendix~\ref{zerosg-a5:proof-thm-sg-diminishingt}.
Then, for any $T\in\mathbb{N}_+$,
\begin{subequations}
\begin{align}
&\mathbf{E}\Big[\frac{1}{n}\sum_{i=1}^{n}\|x_{i,T}-\bar{x}_T\|^2\Big]
=\mathcal{O}\Big(\frac{p}{T^{2}}\Big), \label{zerosg-a5:thm-sg-diminishing-equ2.1bounded}\\
&\mathbf{E}[f(\bar{x}_{T})-f^*]
=\mathcal{O}\Big(\frac{p}{nT}\Big)+\mathcal{O}\Big(\frac{p}{T^{2}}\Big).
\label{zerosg-a5:thm-sg-diminishing-equ2bounded}
\end{align}
\end{subequations}
\end{theorem}
\begin{proof}
The proof is given  in Appendix~\ref{zerosg-a5:proof-thm-sg-diminishingt}. It should be highlighted that the omitted constants in the first term on the right-hand side of \eqref{zerosg-a5:thm-sg-diminishing-equ2bounded} do not depend on any parameters related to the communication network.
\end{proof}
\begin{remark}
To the best of our knowledge, Theorems~\ref{zerosg:thm-sg-diminishing}--\ref{zerosg-a5:thm-sg-diminishingt} are the first performance analysis results for ZO algorithms to solve stochastic optimization problems under the P--{\L} condition or strong convexity assumption.
In \cite{shamir2013complexity}, a centralized ZO algorithm based on one-point sampling with additive sampling noise was proposed and an $\mathcal{O}(p^2/T)$ convergence rate was achieved for deterministic optimization problems with strongly convex quadratic cost functions. In \cite{bach2016highly}, a centralized ZO algorithm based on two-point sampling with additive noise was proposed and an $\mathcal{O}(p/\sqrt{T})$ convergence rate was achieved for deterministic strongly convex and smooth optimization problems. In \cite{sahu2018distributed}, a distributed ZO gradient descent algorithm based on $2p$-point sampling with additive noise was proposed and an $\mathcal{O}(pn^2/\sqrt{T})$ convergence rate was achieved for deterministic strongly convex and smooth optimization problems. In \cite{tang2020distributedzero}, a distributed ZO gradient descent algorithm based on two-point sampling was proposed and an $\mathcal{O}(p/T)$ convergence rate was achieved for deterministic smooth optimization problems under the P--{\L} condition. It is straightforward to see that aforementioned convergence rates achieved in \cite{shamir2013complexity,bach2016highly,sahu2018distributed,tang2020distributedzero} are slower than that achieved by our distributed stochastic ZO primal--dual algorithm as stated in Theorem~\ref{zerosg-a5:thm-sg-diminishingt}. Moreover, we consider stochastic optimization problems and use the P--{\L} condition, which is slightly weaker than the strong convexity condition. The distributed ZO gradient tracking algorithm in \cite{tang2020distributedzero} and the distributed ZO primal--dual algorithms in \cite{yi2021linear}  achieved linear convergence under the P--{\L} condition. However, both algorithms require each agent at each iteration to sample $\mathcal{O}(p)$ points, which results in a heavy sampling burden when $p$ is large.
\end{remark}

As shown in Theorems~\ref{zerosg:thm-sg-diminishing}--\ref{zerosg-a5:thm-sg-diminishingt}, in expectation, the convergence rate of Algorithm~\ref{zerosg:algorithm-random-pd} with diminishing stepsizes is sublinear. The following theorem establishes that, in expectation, the output of Algorithm~\ref{zerosg:algorithm-random-pd} with constant algorithm parameters linearly converges to a neighborhood of a global
optimum.

\begin{theorem}\label{zerosg:thm-random-pd-fixed}
Suppose Assumptions~\ref{zerosg:ass:graph}--\ref{zerosg:ass:fig} hold. Let $\{\bsx_k\}$ be the sequence generated by Algorithm~\ref{zerosg:algorithm-random-pd} with
\begin{align}\label{zerosg:step:eta2-fixed}
&\alpha_k=\alpha=\kappa_1\beta,~\beta_k=\beta,~ \eta_k=\eta=\frac{\kappa_2}{\beta},
~\delta_{i,k}\le\kappa_\delta\tilde{\varepsilon}^{k},~\forall k\in\mathbb{N}_0,
\end{align}
where $\kappa_1>c_1$, $\kappa_2\in(0,c_2(\kappa_1))$, $\beta\ge\tilde{c}_0(\kappa_1,\kappa_2)$, $\tilde{\varepsilon}\in(0,1)$, and $\kappa_\delta>0$ with $\tilde{c}_0(\kappa_1,\kappa_2)$ being given in Appendix~\ref{zerosg:proof-thm-sg-smT}, and $c_1$ and $c_2(\kappa_1)$ being given in Appendix~\ref{zerosg:proof-thm-random-pd-sm}.
Then, for any $T\in\mathbb{N}_+$,
\begin{subequations}
\begin{align}
&\frac{1}{T}\sum_{k=0}^{T-1}\mathbf{E}\Big[\frac{1}{n}\sum_{i=1}^{n}
\|x_{i,k}-\bar{x}_k\|^2\Big]
=\mathcal{O}\Big(\frac{1}{ T}+(\sigma^2_1+2(1+\sigma_0^2)\sigma^2_2)p\eta^2\Big),
\label{zerosg:thm-sg-fixed-equ3.1}\\
&\mathbf{E}\Big[\frac{1}{n}\sum_{i=1}^{n}
\|x_{i,T}-\bar{x}_T\|^2\Big]=
\mathcal{O}\Big(p\eta^2+(\sigma^2_1+2(1+\sigma_0^2)\sigma^2_2)
p^2\eta^4\Big(\frac{1}{n}+\eta\Big)T\Big),
\label{zerosg:thm-sg-fixed-equ3.2}\\
&\frac{1}{T}\sum_{k=0}^{T-1}\mathbf{E}[\|\nabla f(\bar{x}_k)\|^2]
=\mathcal{O}\Big(\frac{1}{\eta T}+(\sigma^2_1+2(1+\sigma_0^2)\sigma^2_2)\Big(\frac{p\eta}{n}+p\eta^2\Big)\Big).
\label{zerosg:thm-sg-fixed-equ3}
\end{align}
\end{subequations}
Moreover, if Assumption~\ref{zerosg:ass:fil} also holds, then
\begin{align}\label{zerosg:thm-sg-fixed-equ1}
&\mathbf{E}\Big[\frac{1}{n}\sum_{i=1}^{n}\|x_{i,k}-\bar{x}_k\|^2+f(\bar{x}_k)-f^*\Big]
=\mathcal{O}(\varepsilon^{k}+(\sigma^2_1+2(1+\sigma_0^2)\sigma^2_2)p\eta),~\forall k\in\mathbb{N}_+,
\end{align}
where $\varepsilon\in(0,1)$ is a positive constant given in Appendix~\ref{zerosg:proof-thm-random-pd-fixed}.
\end{theorem}
\begin{proof}
The proof is given in Appendix~\ref{zerosg:proof-thm-random-pd-fixed}.
\end{proof}

If Assumption~\ref{zerosg:ass:zeroth-variance}$'$--\ref{zerosg:ass:fig}$'$ hold, then $\sigma_1=\sigma_2=0$. In this case, from Theorem~\ref{zerosg:thm-random-pd-fixed}, we have the following results.
\begin{corollary}[Linear convergence]\label{zerosg:thm-random-pd-fixed-coro1}
Under the same setup as Theorem~\ref{zerosg:thm-random-pd-fixed} and suppose Assumption~\ref{zerosg:ass:zeroth-variance}$'$--\ref{zerosg:ass:fig}$'$ hold, then, for any $T\in\mathbb{N}_+$,
\begin{subequations}
\begin{align}
&\frac{1}{T}\sum_{k=0}^{T-1}\mathbf{E}\Big[\frac{1}{n}\sum_{i=1}^{n}
\|x_{i,k}-\bar{x}_k\|^2\Big]
=\mathcal{O}\Big(\frac{1}{ T}\Big),
\label{zerosg:thm-sg-fixed-equ3.1-coro1}\\
&\mathbf{E}\Big[\frac{1}{n}\sum_{i=1}^{n}
\|x_{i,k}-\bar{x}_k\|^2\Big]=
\mathcal{O}(p\eta^2),
\label{zerosg:thm-sg-fixed-equ3.2-coro1}\\
&\frac{1}{T}\sum_{k=0}^{T-1}\mathbf{E}[\|\nabla f(\bar{x}_k)\|^2]=\mathcal{O}\Big(\frac{1}{\eta T}\Big).
\label{zerosg:thm-sg-fixed-equ3-coro1}
\end{align}
\end{subequations}
Moreover, if Assumption~\ref{zerosg:ass:fil} also holds, then
\begin{align}\label{zerosg:thm-sg-fixed-equ1-coro1}
&\mathbf{E}\Big[\frac{1}{n}\sum_{i=1}^{n}\|x_{i,k}-\bar{x}_k\|^2+f(\bar{x}_k)-f^*\Big]
=\mathcal{O}(\varepsilon^{k}),~\forall k\in\mathbb{N}_+.
\end{align}
\end{corollary}

\begin{remark}
The result stated in \eqref{zerosg:thm-sg-fixed-equ3-coro1} shows that a stationary point can be found with a rate $\mathcal{O}(p/T)$.  This rate is the same as that achieved by the ZO algorithms in \cite{nesterov2017random,liu2018zeroth,gu2018faster,huang2019faster,kozak2020stochastic}. Although the ZO variance reduced algorithms in \cite{fang2018spider,pmlr-v97-ji19a} and the stochastic direct-search algorithms in \cite{bergou2019stochastic,bibi2019stochastic,gorbunov2019stochastic} achieved a faster rate $\mathcal{O}(1/T)$, these algorithms require three or more samplings at each iteration, while our proposed algorithm requires only two samplings. Moreover, the result stated in \eqref{zerosg:thm-sg-fixed-equ1-coro1} shows that a global optimum can be found linearly. 
The ZO algorithms in \cite{nesterov2017random,ye2018hessian,pmlr-v97-ji19a,chen2020accelerated,
kozak2020stochastic,cai2020zeroth} and the stochastic direct-search algorithms in \cite{bergou2019stochastic,bibi2019stochastic,gorbunov2019stochastic,golovin2019gradientless} also achieved linear convergence. However, the algorithms in \cite{ye2018hessian,bergou2019stochastic,bibi2019stochastic,gorbunov2019stochastic,
pmlr-v97-ji19a,golovin2019gradientless,chen2020accelerated}  require three or more samplings at each iteration; the P--{\L} constant needs to be known in advance in \cite{kozak2020stochastic,pmlr-v97-ji19a}, which is not needed in Theorem~\ref{zerosg:thm-random-pd-fixed}; and the cost functions in \cite{nesterov2017random,ye2018hessian,bergou2019stochastic,bibi2019stochastic,
gorbunov2019stochastic,golovin2019gradientless,chen2020accelerated,cai2020zeroth}  are
(restricted) strongly convex, which is stronger than the P--{\L} condition used in Theorem~\ref{zerosg:thm-random-pd-fixed}.
\end{remark}

To end this section, we would like to briefly explain the challenges when analyzing the performance of Algorithm~\ref{zerosg:algorithm-random-pd}. Algorithm~\ref{zerosg:algorithm-random-pd} is simple in the sense that it is a combination of the first-order algorithm proposed in \cite{yi2021linear} with zeroth-order gradient estimators. For such a kind of combination, the standard technique to handle the bias in the ZO gradients is using smoothing function, which is also used in our proofs. However, there still is a gap between the smoothing function and the original function. This gap complicates the proof details, especially under the distributed and stochastic setting. As a result, one needs to make an assumption on the relation between the local and global gradients, such as the Lipschitz assumption, i.e., $\|\nabla f_i(x)\|$ is globally bounded, the bounded variance assumption, i.e., $\|\nabla f_i(x)-\nabla f(x)\|^2$ is globally bounded, or the weaker Assumption~\ref{zerosg:ass:fig} used in this paper. Moreover, to the best of our knowledge, how to show linear speedup for distributed ZO algorithms is an open problem in the literature. A key point to show linear speedup is to guarantee that the omitted constants in the dominate term in the convergence rate do not depend on any parameters related to the communication network.
In addition, the proofs are much more complicated due to weaker assumptions.

\section{Distributed ZO Primal Algorithm}\label{zerosg-p:sec-main-random}
In this section, we propose a distributed ZO primal  algorithm and analyze its convergence rate.
Inspired by distributed first-order (sub)gradient descent algorithm proposed in \cite{nedic2009distributed}, we propose the following distributed ZO primal algorithm:
\begin{align}\label{zerosg-p:alg:random}
x_{i,k+1} =x_{i,k}-\gamma\sum_{j=1}^nL_{ij}x_{j,k}-\eta_kg^e_{i,k},
\end{align}
where $\gamma$ is a positive constant, $\{\eta_k\}$ is a positive sequence to be specified later, and $g^e_{i,k}$ is the stochastic gradient estimator defined in \eqref{dbco:gradient:model2-st}.

We write the distributed random ZO algorithm \eqref{zerosg-p:alg:random} in pseudo-code as Algorithm~\ref{zerosg-p:algorithm-random}. Compared with Algorithm~\ref{zerosg:algorithm-random-pd}, in Algorithm~\ref{zerosg-p:algorithm-random} each agent only computes the primal variable. Similar results as stated in Theorems~\ref{zerosg:thm-random-pd-sm}--\ref{zerosg:thm-random-pd-fixed} and Corollary~\ref{zerosg:thm-random-pd-fixed-coro1} also hold for Algorithm~\ref{zerosg-p:algorithm-random}.
\begin{algorithm}[tb]
\caption{Distributed  ZO Primal Algorithm}
\label{zerosg-p:algorithm-random}
\begin{algorithmic}[1]
\STATE \textbf{Input}: positive constant $\gamma$ and positive sequences $\{\eta_k\}$ and $\{\delta_{i,k}\}$.
\STATE \textbf{Initialize}: $ x_{i,0}\in\mathbb{R}^p,~\forall i\in[n]$.
\FOR{$k=0,1,\dots$}
\FOR{$i=1,\dots,n$  in parallel}
\STATE  Broadcast $x_{i,k}$ to $\mathcal{N}_i$ and receive $x_{j,k}$ from $j\in\mathcal{N}_i$;
\STATE Generate vector $u_{i,k}\in\mathbb{S}^{p}$ independently and uniformly at random;
\STATE Generate $\xi_{i,k}$ independently and randomly according to the distribution of $\xi_i$;
\STATE Sample $F_i(x_{i,k},\xi_{i,k})$ and $F_i(x_{i,k}+\delta_{i,k}u_{i,k},\xi_{i,k})$;
\STATE  Update $x_{i,k+1}$ by \eqref{zerosg-p:alg:random}.
\ENDFOR
\ENDFOR
\STATE  \textbf{Output}: $\{\bsx_{k}\}$.
\end{algorithmic}
\end{algorithm}

\subsection{Find stationary points}

\begin{theorem}\label{zerosg-p:thm-random-sm}
Suppose Assumptions~\ref{zerosg:ass:graph}--\ref{zerosg:ass:fig} hold. Let $\{\bsx_k\}$ be the sequence generated by Algorithm~\ref{zerosg-p:algorithm-random} with
\begin{align}\label{zerosg-p:thm-random-sm-akbk}
\gamma\in(0,d_1),~\eta_k=\frac{\kappa_\eta}{(k+t_1)^\theta},~\delta_{i,k}\le \frac{\kappa_\delta\sqrt{p\eta_k}}{\sqrt{n+p}},~\forall k\in\mathbb{N}_0,
\end{align}
where $\kappa_\delta>0$, $\kappa_\eta\in(0,d_2(\gamma)t_1^\theta]$, $\theta\in(0.5,1)$, and $t_1\ge p^{\frac{1}{2\theta}}$ with $d_1$ and $d_2(\gamma)$ being given in Appendix~\ref{zerosg-p:proof-thm-random-sm}.
Then, for any $T\in\mathbb{N}_+$,
\begin{subequations}
\begin{align}
&\frac{1}{T}\sum_{k=0}^{T-1}\mathbf{E}[\|\nabla f(\bar{x}_k)\|^2]
=\mathcal{O}\Big(\frac{\sqrt{p}}{T^{1-\theta}}+\frac{p}{T}\Big),\label{zerosg-p:thm-sg-sm-equ1}\\
&\mathbf{E}[f(\bar{x}_{T})]-f^*=\mathcal{O}(1),\label{zerosg-p:thm-sg-sm-equ2}\\
&\mathbf{E}\Big[\frac{1}{n}\sum_{i=1}^{n}\|x_{i,T}-\bar{x}_T\|^2\Big]
=\mathcal{O}\Big(\frac{1}{T^{2\theta}}\Big), \label{zerosg-p:thm-sg-sm-equ1.1bounded}\\
&\lim_{T\rightarrow+\infty}\mathbf{E}[\|\nabla f(\bar{x}_T)\|^2]=0.\label{zerosg-p:thm-sg-sm-equ1bounded}
\end{align}
\end{subequations}
\end{theorem}
\begin{proof}
The proof is given in Appendix~\ref{zerosg-p:proof-thm-random-sm}.
\end{proof}

\begin{theorem}[Linear speedup]\label{zerosg-p:thm-sg-smT}
Suppose Assumptions~\ref{zerosg:ass:graph}--\ref{zerosg:ass:fig} hold. For any given $T\ge \max\{\frac{n}{pd_2^2(\gamma)},~\frac{n^3}{p}\}$, let $\{\bsx_k,k=0,\dots,T\}$ be the output generated by Algorithm~\ref{zerosg-p:algorithm-random} with
\begin{align}\label{zerosg-p:step:eta2-sm}
&\gamma\in(0,d_1),~\eta_k=\frac{\sqrt{n}}{\sqrt{pT}},
~\delta_{i,k}\le\frac{p^{\frac{1}{4}}n^{\frac{1}{4}}\kappa_\delta}
{\sqrt{n+p}(k+1)^{\frac{1}{4}}},~\forall k\le T,
\end{align}
where $\kappa_\delta>0$, and $d_1$ as well as $d_2(\gamma)$ are given in Appendix~\ref{zerosg-p:proof-thm-random-sm}, then
\begin{subequations}
\begin{align}
&\frac{1}{T}\sum_{k=0}^{T-1}\mathbf{E}[\|\nabla f(\bar{x}_k)\|^2]
=\mathcal{O}\Big(\frac{\sqrt{p}}{\sqrt{nT}}\Big)+\mathcal{O}\Big(\frac{n}{T}\Big),\label{zerosg-p:thm-sg-sm-equ3}\\
&\mathbf{E}[f(\bar{x}_{T})]-f^*=\mathcal{O}(1),\label{zerosg-p:thm-sg-sm-equ4}\\
&\mathbf{E}\Big[\frac{1}{n}\sum_{i=1}^{n}\|x_{i,T}-\bar{x}_T\|^2\Big]
=\mathcal{O}\Big(\frac{n}{T}\Big),\label{zerosg-p:thm-sg-sm-equ3.1}\\
&\lim_{T\rightarrow+\infty}\mathbf{E}[\|\nabla f(\bar{x}_T)\|^2]=0.\label{zerosg-p:coro-sg-sm-equ3bounded}
\end{align}
\end{subequations}
\end{theorem}
\begin{proof}
The proof is given in Appendix~\ref{zerosg-p:proof-thm-sg-smT}. It should be highlighted that the omitted constants in the first term on the right-hand side of \eqref{zerosg-p:thm-sg-sm-equ3} do not depend on any parameters related to the communication network.
\end{proof}

\subsection{Find global optimum}
\begin{theorem}\label{zerosg-p:thm-sg-diminishing}
Suppose Assumptions~\ref{zerosg:ass:graph}--\ref{zerosg:ass:fil} hold. Let $\{\bsx_k\}$ be the sequence generated by Algorithm~\ref{zerosg-p:algorithm-random} with
\begin{align}\label{zerosg-p:step:eta1}
\gamma\in(0,d_1),~\eta_k=\frac{\kappa_\eta}{(k+t_1)^\theta},~\delta_{i,k}\le \frac{\kappa_\delta\sqrt{p\eta_k}}{\sqrt{n+p}},~\forall k\in\mathbb{N}_0,
\end{align}
where $\kappa_\delta>0$, $\kappa_\eta\in(0,d_2(\gamma)t_1^\theta]$, $\theta\in(0,1)$, $t_1\in[ p^{\frac{1}{\theta}},d_3p^{\frac{1}{\theta}}]$, and $d_3\ge1$ with $d_1$ and $d_2(\gamma)$ being given in Appendix~\ref{zerosg-p:proof-thm-random-sm}.
Then, for any $T\in\mathbb{N}_+$,
\begin{subequations}
\begin{align}
&\mathbf{E}\Big[\frac{1}{n}\sum_{i=1}^{n}\|x_{i,T}-\bar{x}_T\|^2\Big]
=\mathcal{O}\Big(\frac{p}{T^{2\theta}}\Big), \label{zerosg-p:thm-sg-diminishing-equ1.1bounded}\\
&\mathbf{E}[f(\bar{x}_{T})-f^*]
=\mathcal{O}\Big(\frac{p}{nT^\theta}\Big)+\mathcal{O}\Big(\frac{p}{T^{2\theta}}\Big).
\label{zerosg-p:thm-sg-diminishing-equ1bounded}
\end{align}
\end{subequations}
\end{theorem}
\begin{proof}
The proof is given in Appendix~\ref{zerosg-p:proof-thm-sg-diminishing}. It should be highlighted that the omitted constants in the first term on the right-hand side of \eqref{zerosg-p:thm-sg-diminishing-equ1bounded} do not depend on any parameters related to the communication network.
\end{proof}

\begin{theorem}[Linear speedup]\label{zerosg-p:thm-sg-diminishingt}
Suppose Assumptions~\ref{zerosg:ass:graph}--\ref{zerosg:ass:fil} hold and the P--{\L} constant $\nu$ is known in advance. Let $\{\bsx_{k}\}$ be the sequence generated by Algorithm~\ref{zerosg-p:algorithm-random} with
\begin{align}\label{zerosg-p:step:eta1t1}
\gamma\in(0,d_1),~\eta_k=\frac{\kappa_\eta}{k+t_1},~\delta_{i,k}\le \frac{\kappa_\delta\sqrt{p\eta_k}}{\sqrt{n+p}},~\forall k\in\mathbb{N}_0,
\end{align}
where $\kappa_\delta>0$, $\kappa_\eta\in(\frac{8}{\nu},\frac{8d_3}{\nu}]$, $t_1> \hat{d}_2(\gamma)$ and $d_3>1$ with $d_1$ and $\hat{d}_2(\gamma)$ being given in Appendices~\ref{zerosg-p:proof-thm-random-sm} and \ref{zerosg-p:proof-thm-sg-diminishingt}, respectively.
Then, for any $T\in\mathbb{N}_+$,
\begin{subequations}
\begin{align}
&\mathbf{E}\Big[\frac{1}{n}\sum_{i=1}^{n}\|x_{i,T}-\bar{x}_T\|^2\Big]
=\mathcal{O}\Big(\frac{p}{T^{2}}\Big), \label{zerosg-p:thm-sg-diminishing-equ2.1bounded}\\
&\mathbf{E}[f(\bar{x}_{T})-f^*]
=\mathcal{O}\Big(\frac{p}{nT}\Big)+\mathcal{O}\Big(\frac{p}{T^{2}}\Big).
\label{zerosg-p:thm-sg-diminishing-equ2bounded}
\end{align}
\end{subequations}
\end{theorem}
\begin{proof}
The proof is given  in Appendix~\ref{zerosg-p:proof-thm-sg-diminishingt}. It should be highlighted that the omitted constants in the first term on the right-hand side of \eqref{zerosg-p:thm-sg-diminishing-equ2bounded} do not depend on any parameters related to the communication network.
\end{proof}

\begin{theorem}[Linear speedup]\label{zerosg-p-a5:thm-sg-diminishingt}
Suppose Assumptions~\ref{zerosg:ass:graph}--\ref{zerosg:ass:zeroth-variance} and \ref{zerosg:ass:fil} hold, and the P--{\L} constant $\nu$ is known in advance, and each $f_i^*>-\infty$.  Let $\{\bsx_{k}\}$ be the sequence generated by Algorithm~\ref{zerosg-p:algorithm-random} with
\begin{align}\label{zerosg-p-a5:step:eta1t1}
\gamma\in(0,d_1),~\eta_k=\frac{\kappa_\eta}{k+t_1},~\delta_{i,k}\le \frac{\kappa_\delta\sqrt{p\eta_k}}{\sqrt{n+p}},~\forall k\in\mathbb{N}_0,
\end{align}
where $\kappa_\delta>0$, $\kappa_\eta\in(\frac{8}{\nu},\frac{8d_3}{\nu}]$, $t_1> \check{d}_2(\gamma)$, and $d_3>1$ with $d_1$ and $\check{d}_2(\gamma)$ being given in Appendices~\ref{zerosg-p:proof-thm-random-sm} and \ref{zerosg-p-a5:proof-thm-sg-diminishingt}, respectively.
Then, for any $T\in\mathbb{N}_+$,
\begin{subequations}
\begin{align}
&\mathbf{E}\Big[\frac{1}{n}\sum_{i=1}^{n}\|x_{i,T}-\bar{x}_T\|^2\Big]
=\mathcal{O}\Big(\frac{p}{T^{2}}\Big), \label{zerosg-p-a5:thm-sg-diminishing-equ2.1bounded}\\
&\mathbf{E}[f(\bar{x}_{T})-f^*]
=\mathcal{O}\Big(\frac{p}{nT}\Big)+\mathcal{O}\Big(\frac{p}{T^{2}}\Big).
\label{zerosg-p-a5:thm-sg-diminishing-equ2bounded}
\end{align}
\end{subequations}
\end{theorem}
\begin{proof}
The proof is given in Appendix~\ref{zerosg-p-a5:proof-thm-sg-diminishingt}. It should be highlighted that the omitted constants in the first term on the right-hand side of \eqref{zerosg-p-a5:thm-sg-diminishing-equ2bounded} do not depend on any parameters related to the communication network.
\end{proof}



\begin{theorem}\label{zerosg-p:thm-random-pd-fixed}
Suppose Assumptions~\ref{zerosg:ass:graph}--\ref{zerosg:ass:fig} hold. Let $\{\bsx_k\}$ be the sequence generated by Algorithm~\ref{zerosg-p:algorithm-random} with
\begin{align}\label{zerosg-p:step:eta2-fixed}
&\gamma\in(0,d_1),~\eta_k=\eta,
~\delta_{i,k}\le\kappa_\delta\tilde{\epsilon}^{k},~\forall k\in\mathbb{N}_0,
\end{align}
where $\eta\in(0,d_2(\gamma)$ and $\tilde{\epsilon}\in(0,1)$ with $d_1$ and $d_2(\gamma)$ being given in Appendix~\ref{zerosg-p:proof-thm-random-sm}.
Then, for any $T\in\mathbb{N}_+$,
\begin{subequations}
\begin{align}
&\frac{1}{T}\sum_{k=0}^{T-1}\mathbf{E}\Big[\frac{1}{n}\sum_{i=1}^{n}
\|x_{i,k}-\bar{x}_k\|^2\Big]
 =\mathcal{O}\Big(\frac{1}{ T}+(\sigma^2_1+2(1+\sigma_0^2)\sigma^2_2)p\eta^2\Big),
\label{zerosg-p:thm-sg-fixed-equ3.1}\\
&\mathbf{E}\Big[\frac{1}{n}\sum_{i=1}^{n}
\|x_{i,T}-\bar{x}_T\|^2\Big]=
\mathcal{O}\Big(p\eta^2+(\sigma^2_1+2(1+\sigma_0^2)\sigma^2_2)
p^2\eta^4\Big(\frac{1}{n}+\eta\Big)T\Big),
\label{zerosg-p:thm-sg-fixed-equ3.2}\\
&\frac{1}{T}\sum_{k=0}^{T-1}\mathbf{E}[\|\nabla f(\bar{x}_k)\|^2]
=\mathcal{O}\Big(\frac{1}{\eta T}+(\sigma^2_1+2(1+\sigma_0^2)\sigma^2_2)\Big(\frac{p\eta}{n}+p\eta^2\Big)\Big).
\label{zerosg-p:thm-sg-fixed-equ3}
\end{align}
\end{subequations}
Moreover, if Assumption~\ref{zerosg:ass:fil} also holds, then
\begin{align}\label{zerosg-p:thm-sg-fixed-equ1}
&\mathbf{E}\Big[\frac{1}{n}\sum_{i=1}^{n}\|x_{i,k}-\bar{x}_k\|^2+f(\bar{x}_k)-f^*\Big]
=\mathcal{O}(\epsilon^{k}+(\sigma^2_1+2(1+\sigma_0^2)\sigma^2_2)p\eta),~\forall k\in\mathbb{N}_+,
\end{align}
where $\epsilon\in(0,1)$ is a positive constant given in Appendix~\ref{zerosg-p:proof-thm-random-pd-fixed}.
\end{theorem}
\begin{proof}
The proof is given in Appendix~\ref{zerosg-p:proof-thm-random-pd-fixed}.
\end{proof}

\begin{corollary}[Linear convergence]\label{zerosg-p:thm-random-pd-fixed-coro1}
Under the same setup as Theorem~\ref{zerosg-p:thm-random-pd-fixed} and suppose Assumption~\ref{zerosg:ass:zeroth-variance}$'$--\ref{zerosg:ass:fig}$'$ hold, then, for any $T\in\mathbb{N}_+$,
\begin{subequations}
\begin{align}
&\frac{1}{T}\sum_{k=0}^{T-1}\mathbf{E}\Big[\frac{1}{n}\sum_{i=1}^{n}
\|x_{i,k}-\bar{x}_k\|^2\Big]
=\mathcal{O}\Big(\frac{1}{ T}\Big),
\label{zerosg-p:thm-sg-fixed-equ3.1-coro1}\\
&\mathbf{E}\Big[\frac{1}{n}\sum_{i=1}^{n}
\|x_{i,k}-\bar{x}_k\|^2\Big]=
\mathcal{O}(p\eta^2),
\label{zerosg-p:thm-sg-fixed-equ3.2-coro1}\\
&\frac{1}{T}\sum_{k=0}^{T-1}\mathbf{E}[\|\nabla f(\bar{x}_k)\|^2]=\mathcal{O}\Big(\frac{1}{\eta T}\Big).
\label{zerosg-p:thm-sg-fixed-equ3-coro1}
\end{align}
\end{subequations}
Moreover, if Assumption~\ref{zerosg:ass:fil} also holds, then
\begin{align}\label{zerosg-p:thm-sg-fixed-equ1-coro1}
&\mathbf{E}\Big[\frac{1}{n}\sum_{i=1}^{n}\|x_{i,k}-\bar{x}_k\|^2+f(\bar{x}_k)-f^*\Big]
=\mathcal{O}(\epsilon^{k}),~\forall k\in\mathbb{N}_+.
\end{align}
\end{corollary}

\section{Simulations}\label{zerosg:sec-simulation}
In this section, we verify the theoretical results through numerical simulations. Specifically, we evaluate the performance of Algorithms~\ref{zerosg:algorithm-random-pd} and  \ref{zerosg-p:algorithm-random} in generating adversarial examples from black-box deep neural networks (DNNs).

In image classification tasks, DNNs are vulnerable to adversarial examples \cite{goodfellow2014explaining} even under small perturbations, which leads misclassifications.
Considering the setting of ZO attacks in \cite{carlini2017towards,liu2018zeroth}, the model is hidden and no gradient information is available. We treat this task of generating adversarial examples as a ZO optimization problem. The black-box attack loss function \cite{carlini2017towards,liu2018zeroth} is given as
\begin{align*}
f_i(x)  &= \max \Big\{ F_{y_i} \Big(\frac{1}{2}\tanh ( \tanh^{-1} 2 a_i + x)\Big)
 - \max_{j \neq y_i}\Big\{ F_j\Big(\frac{1}{2}\tanh ( \tanh^{-1} 2 a_i + x)\Big)\Big\} ,~0 \Big\} \\
&\quad + c\Big\|\frac{1}{2}\tanh ( \tanh^{-1} 2 a_i + x) - a_i \Big\|_2^2,
\end{align*}
where $c$ is a constant, $(a_i,y_i)$ denotes the pair of the $i$th natural image $a_i$ and its original class label $y_i$. The output of function $F(z)=\col(F_1(z),\dots,F_m(z))$ is the well-trained model prediction of the input $z$ in all $m$ image classes.

The well-trained DNN model\footnote{\url{https://github.com/carlini/nn_robust_attacks}} on the MNIST handwritten dataset has $99.4\%$ test accuracy on natural examples \cite{liu2018zeroth}.
We compare the proposed distributed primal--dual ZO algorithm (Algorithm~\ref{zerosg:algorithm-random-pd}) and distributed primal ZO algorithm (Algorithm~\ref{zerosg-p:algorithm-random}) with state-of-the-art centralized and distributed ZO algorithms: RSGF \cite{ghadimi2013stochastic}, SZO-SPIDER \cite{fang2018spider}, ZO-SVRG \cite{liu2018zeroth}, SZVR-G \cite{liu2018stochastic}, and ZO-SPIDER-Coord \cite{pmlr-v97-ji19a}, ZO-GDA \cite{tang2020distributedzero}, and ZONE-M \cite{hajinezhad2019zone}.

We consider $n=10$ agents and assume the communication network is generated randomly following the Erd\H{o}s--R\' enyi model with probability of $0.4$. All the hyper-parameters used in the experiment are given in TABLE~\ref{tab:para}. 

\begin{table*}[ht!]
\caption{Parameters in each algorithm.}
\label{tab:para}
\centering
\small
\vskip 0.05in
\begin{tabular}{M{3.5cm}|M{2.1cm}|M{8.0cm}N}
\hline
Algorithm  & Distributed &Parameters&\\[5pt]

\hline
Algorithm~\ref{zerosg:algorithm-random-pd}       & \ding{52}  & $\eta = 0.5/{k^{10^{-5}}}$, $\alpha = 0.5k^{10^{-5}}$, $\beta = 0.1k^{10^{-5}}$ &\\[5pt]

\hline
Algorithm~\ref{zerosg-p:algorithm-random}       & \ding{52}  & $\gamma = 0.01$, $\eta = 0.08/{k^{10^{-5}}}$ & \\[5pt]

\hline
ZO-GDA       & \ding{52}  & $\eta = 0.08/{k^{10^{-5}}}$ &\\[5pt]

\hline
ZONE-M       & \ding{52}  & $\mu = 1\sqrt{k}$, $\rho = 0.4\sqrt{k}$ &\\[5pt]

\hline
RSGF         & \ding{55}  & $\mu = 0.01$ &\\[5pt]

\hline
SZO-SPIDER      & \ding{55}  & $\mu = 0.01$ &\\[5pt]

\hline
ZO-SVRG      & \ding{55}  & $\mu = 0.01$ &\\[5pt]

\hline
SZVR-G       & \ding{55}  & $\mu = 0.01$ &\\[5pt]

\hline
ZO-SPIDER-Coord       & \ding{55}  & $\mu = 0.01$ &\\[5pt]

\hline
\end{tabular}
\end{table*}

Fig.~\ref{zerosg:fig:dnn_loss} and Fig.~\ref{zerosg:fig:dnn_loss_query} show the evolutions of the  black-box attack loss achieved by each ZO algorithm with respect to the number of iterations and function value queries, respectively. From these two figures, we can see that our proposed distributed ZO algorithms are as efficient as ZO-GDA \cite{tang2020distributedzero} in terms of both convergence rate and sampling complexity, and more efficient than the other algorithms. The least $\ell_2$ distortions of the successful adversarial perturbations are listed in TABLE~\ref{zerosg:tab:distortion}. We can see that the adversarial examples generated by the distributed algorithms in general have slightly larger $\ell_2$ distortions than those generated by the centralized algorithms. TABLE~\ref{zerosg:table:digit4} provides a comparison of generated adversarial examples from the DNN on the MNIST dataset: digit class ``4''.
\begin{figure}
\centering
  \includegraphics[width=1\textwidth]{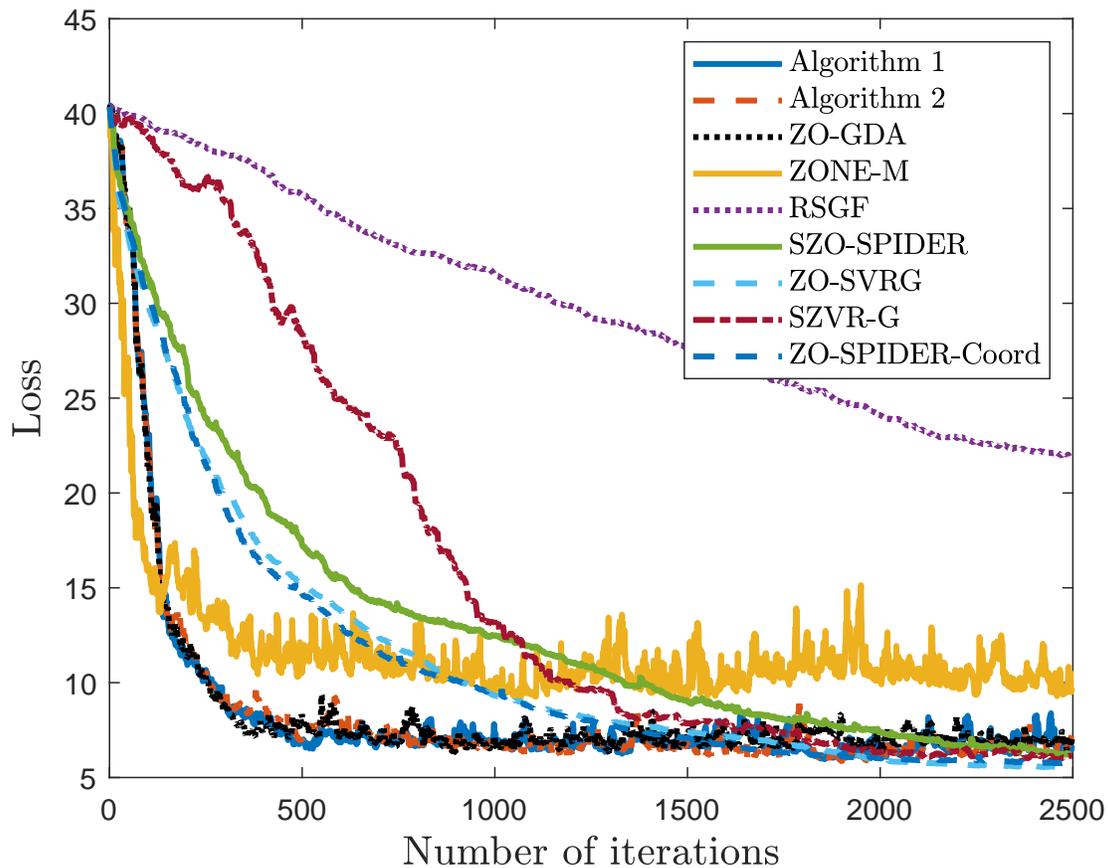}
  \caption{Evolutions of the black-box attack loss with respect to the number of iterations.}
  \label{zerosg:fig:dnn_loss}
\end{figure}

\begin{figure}
\centering
  \includegraphics[width=1\textwidth]{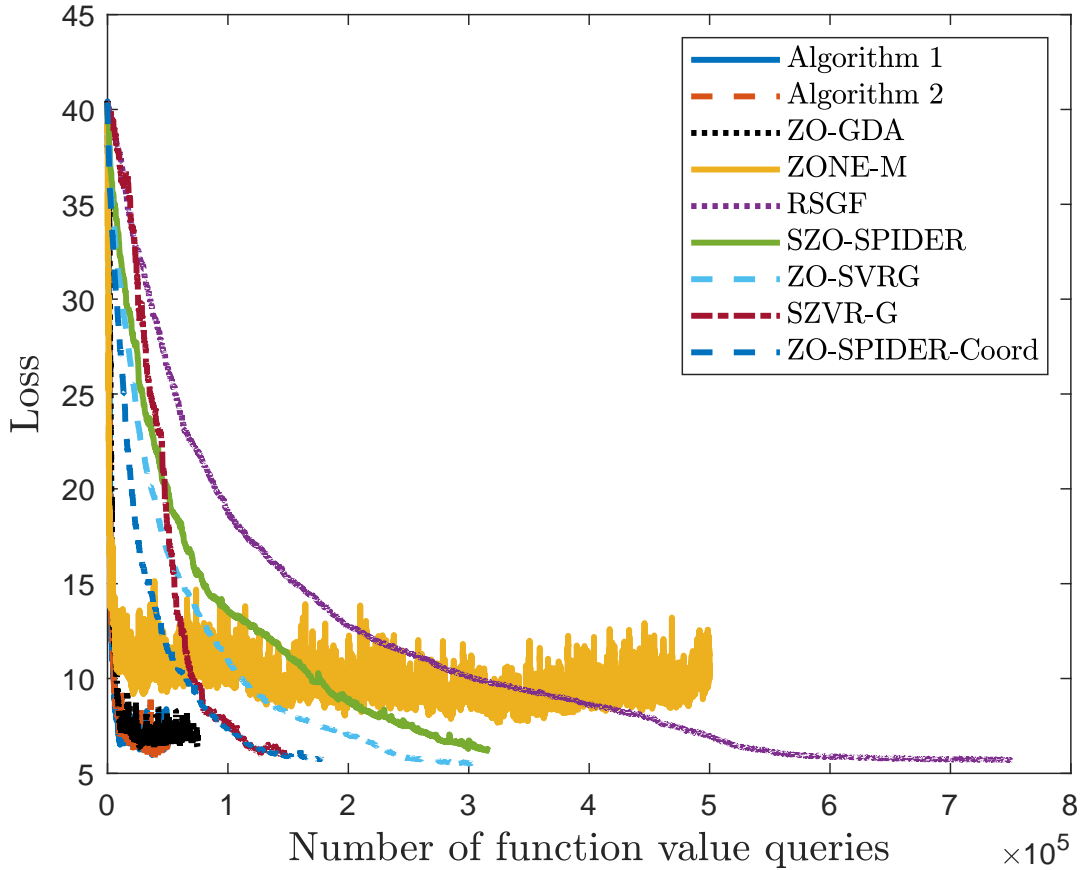}
  \caption{Evolutions of the black-box attack loss with respect to the number of function value queries.}
  \label{zerosg:fig:dnn_loss_query}
\end{figure}

\begin{table}[]
\caption{Distortion}
\label{zerosg:tab:distortion}
\centering
\small
\vskip 0.05in
\begin{tabular}{M{3.5cm}|M{2.1cm}N}
\hline
Algorithm  & $\ell_2$ distortion&\\[5pt]

\hline
Algorithm~\ref{zerosg:algorithm-random-pd}       & 6.44 &\\[5pt]

\hline
Algorithm~\ref{zerosg-p:algorithm-random}       & 5.77  &\\[5pt]

\hline
ZO-GDA       & 7.23  &\\[5pt]

\hline
ZONE-M       & 9.96  &\\[5pt]

\hline
RSGF         & 5.69  &\\[5pt]

\hline
SZO-SPIDER       & 6.19  &\\[5pt]

\hline
ZO-SVRG      & 4.76 &\\[5pt]

\hline
SZVR-G       & 5.16   &\\[5pt]

\hline
ZO-SPIDER-Coord      & 5.76   &\\[5pt]

\hline
\end{tabular}
\end{table}

 \begin{table*}[ht!]
 \caption{Comparison of generated adversarial examples from a black-box DNN on MNIST: digit class ``4''.} \label{zerosg:table:digit4}
  \centering
  \footnotesize
    \vskip 0.05in
  \begin{adjustbox}{max width=\textwidth}
  \begin{tabular}
      {ccccccccccc}
      \hline
      	Image ID & $~~$ 4 & $~~$ 6 & $~~$ 19 & $~~$ 24 & $~~$ 27 & $~~$ 33 & $~~$ 42 & $~~$ 48 & $~~$ 49 & $~~$ 56 \\
      \hline &&&&&&&&&& \vspace{-0.2cm} \\
      	Original &
        \parbox[c]{2.2em}{\includegraphics[width=0.4in]{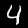}} &
        \parbox[c]{2.2em}{\includegraphics[width=0.4in]{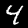}} &
        \parbox[c]{2.2em}{\includegraphics[width=0.4in]{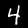}} &
        \parbox[c]{2.2em}{\includegraphics[width=0.4in]{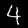}} &
        \parbox[c]{2.2em}{\includegraphics[width=0.4in]{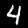}} &
        \parbox[c]{2.2em}{\includegraphics[width=0.4in]{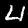}} &
        \parbox[c]{2.2em}{\includegraphics[width=0.4in]{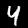}} &
        \parbox[c]{2.2em}{\includegraphics[width=0.4in]{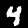}} &
        \parbox[c]{2.2em}{\includegraphics[width=0.4in]{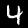}} &
        \parbox[c]{2.2em}{\includegraphics[width=0.4in]{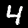}}
      \vspace{0.2cm} \\

      \hline &&&&&&&&&& \vspace{-0.2cm} \\
      	Algorithm~\ref{zerosg:algorithm-random-pd} &
        \parbox[c]{2.2em}{\includegraphics[width=0.4in]{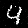}} &
        \parbox[c]{2.2em}{\includegraphics[width=0.4in]{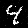}} &
        \parbox[c]{2.2em}{\includegraphics[width=0.4in]{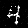}} &
        \parbox[c]{2.2em}{\includegraphics[width=0.4in]{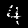}} &
        \parbox[c]{2.2em}{\includegraphics[width=0.4in]{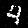}} &
        \parbox[c]{2.2em}{\includegraphics[width=0.4in]{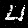}} &
        \parbox[c]{2.2em}{\includegraphics[width=0.4in]{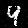}} &
        \parbox[c]{2.2em}{\includegraphics[width=0.4in]{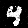}} &
        \parbox[c]{2.2em}{\includegraphics[width=0.4in]{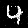}} &
        \parbox[c]{2.2em}{\includegraphics[width=0.4in]{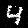}} \\
        Classified as & $~~$ 9 & $~~$ 8 & $~~$ 2 & $~~$ 7 & $~~$ 2 & $~~$ 2 & $~~$ 9 & $~~$ 9 & $~~$ 9 & $~~$ 9 \\

 	\hline &&&&&&&&&& \vspace{-0.2cm} \\
      	Algorithm~\ref{zerosg-p:algorithm-random} &
        \parbox[c]{2.2em}{\includegraphics[width=0.4in]{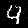}} &
        \parbox[c]{2.2em}{\includegraphics[width=0.4in]{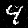}} &
        \parbox[c]{2.2em}{\includegraphics[width=0.4in]{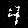}} &
        \parbox[c]{2.2em}{\includegraphics[width=0.4in]{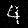}} &
        \parbox[c]{2.2em}{\includegraphics[width=0.4in]{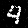}} &
        \parbox[c]{2.2em}{\includegraphics[width=0.4in]{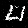}} &
        \parbox[c]{2.2em}{\includegraphics[width=0.4in]{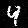}} &
        \parbox[c]{2.2em}{\includegraphics[width=0.4in]{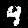}} &
        \parbox[c]{2.2em}{\includegraphics[width=0.4in]{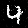}} &
        \parbox[c]{2.2em}{\includegraphics[width=0.4in]{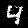}} \\
        Classified as & $~~$ 9 & $~~$ 9 & $~~$ 7 & $~~$ 9 & $~~$ 9 & $~~$ 2 & $~~$ 9 & $~~$ 9 & $~~$ 9 & $~~$ 9 \\

	\hline &&&&&&&&&& \vspace{-0.2cm} \\
      	ZO-GDA &
        \parbox[c]{2.2em}{\includegraphics[width=0.4in]{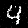}} &
        \parbox[c]{2.2em}{\includegraphics[width=0.4in]{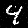}} &
        \parbox[c]{2.2em}{\includegraphics[width=0.4in]{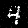}} &
        \parbox[c]{2.2em}{\includegraphics[width=0.4in]{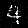}} &
        \parbox[c]{2.2em}{\includegraphics[width=0.4in]{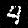}} &
        \parbox[c]{2.2em}{\includegraphics[width=0.4in]{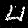}} &
        \parbox[c]{2.2em}{\includegraphics[width=0.4in]{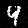}} &
        \parbox[c]{2.2em}{\includegraphics[width=0.4in]{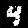}} &
        \parbox[c]{2.2em}{\includegraphics[width=0.4in]{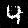}} &
        \parbox[c]{2.2em}{\includegraphics[width=0.4in]{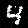}} \\
        Classified as & $~~$ 9 & $~~$ 9 & $~~$ 2 & $~~$ 2 & $~~$ 2 & $~~$ 2 & $~~$ 9 & $~~$ 9 & $~~$ 9 & $~~$ 3 \\

     \hline &&&&&&&&&& \vspace{-0.2cm} \\
      	ZONE-M &
        \parbox[c]{2.2em}{\includegraphics[width=0.4in]{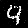}} &
        \parbox[c]{2.2em}{\includegraphics[width=0.4in]{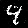}} &
        \parbox[c]{2.2em}{\includegraphics[width=0.4in]{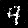}} &
        \parbox[c]{2.2em}{\includegraphics[width=0.4in]{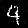}} &
        \parbox[c]{2.2em}{\includegraphics[width=0.4in]{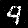}} &
        \parbox[c]{2.2em}{\includegraphics[width=0.4in]{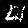}} &
        \parbox[c]{2.2em}{\includegraphics[width=0.4in]{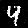}} &
        \parbox[c]{2.2em}{\includegraphics[width=0.4in]{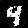}} &
        \parbox[c]{2.2em}{\includegraphics[width=0.4in]{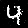}} &
        \parbox[c]{2.2em}{\includegraphics[width=0.4in]{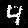}} \\
        Classified as & $~~$ 9 & $~~$ 9 & $~~$ 7 & $~~$ 9 & $~~$ 9 & $~~$ 2 & $~~$ 9 & $~~$ 9 & $~~$ 9 & $~~$ 9 \\
      \hline &&&&&&&&&& \vspace{-0.2cm} \\
      	RSGF &
        \parbox[c]{2.2em}{\includegraphics[width=0.4in]{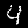}} &
        \parbox[c]{2.2em}{\includegraphics[width=0.4in]{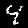}} &
        \parbox[c]{2.2em}{\includegraphics[width=0.4in]{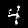}} &
        \parbox[c]{2.2em}{\includegraphics[width=0.4in]{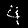}} &
        \parbox[c]{2.2em}{\includegraphics[width=0.4in]{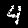}} &
        \parbox[c]{2.2em}{\includegraphics[width=0.4in]{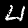}} &
        \parbox[c]{2.2em}{\includegraphics[width=0.4in]{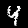}} &
        \parbox[c]{2.2em}{\includegraphics[width=0.4in]{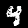}} &
        \parbox[c]{2.2em}{\includegraphics[width=0.4in]{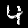}} &
        \parbox[c]{2.2em}{\includegraphics[width=0.4in]{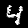}} \\
        Classified as & $~~$ 9 & $~~$ 9 & $~~$ 2 & $~~$ 9 & $~~$ 9 & $~~$ 2 & $~~$ 9 & $~~$ 9 & $~~$ 9 & $~~$ 9 \\

        \hline &&&&&&&&&& \vspace{-0.2cm} \\
      	SZO-SPIDER &
        \parbox[c]{2.2em}{\includegraphics[width=0.4in]{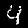}} &
        \parbox[c]{2.2em}{\includegraphics[width=0.4in]{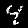}} &
        \parbox[c]{2.2em}{\includegraphics[width=0.4in]{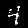}} &
        \parbox[c]{2.2em}{\includegraphics[width=0.4in]{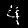}} &
        \parbox[c]{2.2em}{\includegraphics[width=0.4in]{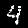}} &
        \parbox[c]{2.2em}{\includegraphics[width=0.4in]{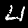}} &
        \parbox[c]{2.2em}{\includegraphics[width=0.4in]{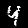}} &
        \parbox[c]{2.2em}{\includegraphics[width=0.4in]{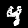}} &
        \parbox[c]{2.2em}{\includegraphics[width=0.4in]{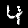}} &
        \parbox[c]{2.2em}{\includegraphics[width=0.4in]{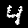}} \\
        Classified as & $~~$ 9 & $~~$ 9 & $~~$ 7 & $~~$ 9 & $~~$ 9 & $~~$ 2 & $~~$ 9 & $~~$ 9 & $~~$ 9 & $~~$ 9 \\

        \hline &&&&&&&&&& \vspace{-0.2cm} \\
      	ZO-SVRG &
        \parbox[c]{2.2em}{\includegraphics[width=0.4in]{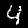}} &
        \parbox[c]{2.2em}{\includegraphics[width=0.4in]{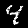}} &
        \parbox[c]{2.2em}{\includegraphics[width=0.4in]{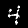}} &
        \parbox[c]{2.2em}{\includegraphics[width=0.4in]{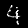}} &
        \parbox[c]{2.2em}{\includegraphics[width=0.4in]{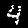}} &
        \parbox[c]{2.2em}{\includegraphics[width=0.4in]{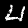}} &
        \parbox[c]{2.2em}{\includegraphics[width=0.4in]{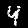}} &
        \parbox[c]{2.2em}{\includegraphics[width=0.4in]{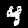}} &
        \parbox[c]{2.2em}{\includegraphics[width=0.4in]{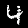}} &
        \parbox[c]{2.2em}{\includegraphics[width=0.4in]{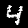}} \\
        Classified as & $~~$ 9 & $~~$ 8 & $~~$ 2 & $~~$ 9 & $~~$ 9 & $~~$ 2 & $~~$ 9 & $~~$ 9 & $~~$ 9 & $~~$ 9 \\

        \hline &&&&&&&&&& \vspace{-0.2cm} \\
      	SZVR-G &
        \parbox[c]{2.2em}{\includegraphics[width=0.4in]{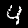}} &
        \parbox[c]{2.2em}{\includegraphics[width=0.4in]{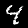}} &
        \parbox[c]{2.2em}{\includegraphics[width=0.4in]{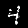}} &
        \parbox[c]{2.2em}{\includegraphics[width=0.4in]{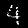}} &
        \parbox[c]{2.2em}{\includegraphics[width=0.4in]{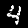}} &
        \parbox[c]{2.2em}{\includegraphics[width=0.4in]{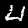}} &
        \parbox[c]{2.2em}{\includegraphics[width=0.4in]{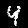}} &
        \parbox[c]{2.2em}{\includegraphics[width=0.4in]{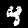}} &
        \parbox[c]{2.2em}{\includegraphics[width=0.4in]{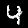}} &
        \parbox[c]{2.2em}{\includegraphics[width=0.4in]{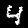}} \\
        Classified as & $~~$ 9 & $~~$ 8 & $~~$ 2 & $~~$ 2 & $~~$ 2 & $~~$ 2 & $~~$ 9 & $~~$ 9 & $~~$ 9 & $~~$ 9 \\

         \hline &&&&&&&&&& \vspace{-0.2cm} \\
      	ZO-SPIDER-Coord &
        \parbox[c]{2.2em}{\includegraphics[width=0.4in]{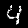}} &
        \parbox[c]{2.2em}{\includegraphics[width=0.4in]{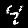}} &
        \parbox[c]{2.2em}{\includegraphics[width=0.4in]{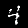}} &
        \parbox[c]{2.2em}{\includegraphics[width=0.4in]{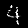}} &
        \parbox[c]{2.2em}{\includegraphics[width=0.4in]{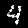}} &
        \parbox[c]{2.2em}{\includegraphics[width=0.4in]{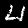}} &
        \parbox[c]{2.2em}{\includegraphics[width=0.4in]{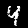}} &
        \parbox[c]{2.2em}{\includegraphics[width=0.4in]{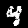}} &
        \parbox[c]{2.2em}{\includegraphics[width=0.4in]{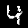}} &
        \parbox[c]{2.2em}{\includegraphics[width=0.4in]{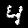}} \\
        Classified as & $~~$ 9 & $~~$ 9 & $~~$ 2 & $~~$ 9 & $~~$ 9 & $~~$ 2 & $~~$ 9 & $~~$ 9 & $~~$ 9 & $~~$ 9 \\

      \hline
  \end{tabular}
  \end{adjustbox}
\end{table*}


In order to verify the result that linear speedup convergence is achieved with respect to the number of agents, we also consider $n=100$ agents. To illustrate the linear speedup results in a more clear manner, we plot the loss in log scale and draw the extensive lines along the convergence lines in Fig.~\ref{zerosg:fig:dnn_loss_node}. The slopes of 10-node lines (blue and red lines) are approximately $-0.025$ and the slopes of 100-node lines (blue and red dash lines) are approximately $-0.079$, which implies the linear speedup results since $-0.079/\sqrt{10} \approx -0.025$. This simulation shows that linear speedup is achieved by our proposed two algorithms even though the optimization problem is nonsmooth. 
\begin{figure}
\centering
  \includegraphics[width=1\textwidth]{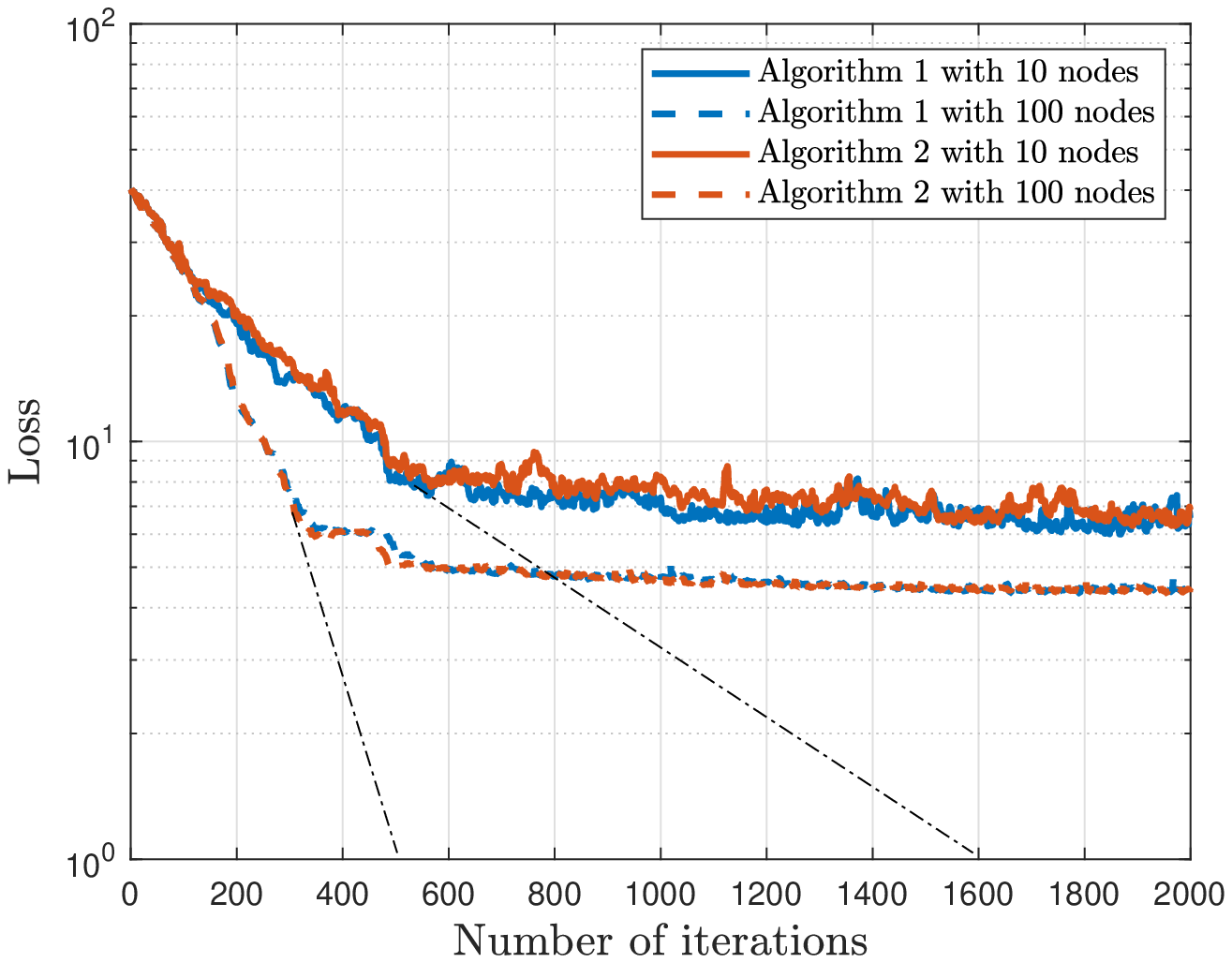}
  \caption{Evolutions of the black-box attack loss with respect to the number of iterations when using different numbers of agents.}
  \label{zerosg:fig:dnn_loss_node}
\end{figure}

\section{Conclusions}\label{zerosg:sec-conclusion}
In this paper, we studied stochastic distributed nonconvex optimization with ZO information feedback. We proposed two distributed ZO algorithms and analyzed their convergence properties. More specifically, linear speedup convergence rate $\mathcal{O}(\sqrt{p/(nT)})$ was established for smooth nonconvex cost functions under  arbitrarily connected communication networks. The convergence rate was improved to  $\mathcal{O}(p/(nT))$ when the global cost function satisfies the P--{\L} condition. It was also shown that the output of the proposed algorithms linearly converges to a neighborhood of a global optimum. Interesting directions for future work include establishing faster convergence with reduced sampling complexity by using variance reduction techniques, and considering communication reduction with asynchronous, periodic, or compressed communication.

\bibliographystyle{IEEEtran}
\bibliography{icml_zeroth,refextra}

\begin{thebibliography}{10}
\providecommand{\url}[1]{#1}
\csname url@samestyle\endcsname
\providecommand{\newblock}{\relax}
\providecommand{\bibinfo}[2]{#2}
\providecommand{\BIBentrySTDinterwordspacing}{\spaceskip=0pt\relax}
\providecommand{\BIBentryALTinterwordstretchfactor}{4}
\providecommand{\BIBentryALTinterwordspacing}{\spaceskip=\fontdimen2\font plus
\BIBentryALTinterwordstretchfactor\fontdimen3\font minus
  \fontdimen4\font\relax}
\providecommand{\BIBforeignlanguage}[2]{{%
\expandafter\ifx\csname l@#1\endcsname\relax
\typeout{** WARNING: IEEEtran.bst: No hyphenation pattern has been}%
\typeout{** loaded for the language `#1'. Using the pattern for}%
\typeout{** the default language instead.}%
\else
\language=\csname l@#1\endcsname
\fi
#2}}
\providecommand{\BIBdecl}{\relax}
\BIBdecl

\bibitem{conn2009introduction}
A.~R. Conn, K.~Scheinberg, and L.~N. Vicente, \emph{Introduction to
  Derivative-Free Optimization}.\hskip 1em plus 0.5em minus 0.4em\relax
  MPS-SIAM Series on Optimization. SIAM Philadelphia, 2009.

\bibitem{audet2017derivative}
C.~Audet and W.~Hare, \emph{Derivative-Free and Blackbox Optimization}.\hskip
  1em plus 0.5em minus 0.4em\relax Springer, 2017.

\bibitem{larson2019derivative}
J.~Larson, M.~Menickelly, and S.~M. Wild, ``Derivative-free optimization
  methods,'' \emph{Acta Numerica}, vol.~28, pp. 287--404, 2019.

\bibitem{chen2017zoo}
P.-Y. Chen, H.~Zhang, Y.~Sharma, J.~Yi, and C.-J. Hsieh, ``{ZOO}: Zeroth order
  optimization based black-box attacks to deep neural networks without training
  substitute models,'' in \emph{ACM Workshop on Artificial Intelligence and
  Security}, 2017, pp. 15--26.

\bibitem{nedic2018distributed}
A.~Nedi{\'c} and J.~Liu, ``Distributed optimization for control,'' \emph{Annual
  Review of Control, Robotics, and Autonomous Systems}, vol.~1, pp. 77--103,
  2018.

\bibitem{Koloskova2019decentralized}
A.~Koloskova, S.~Stich, and M.~Jaggi, ``Decentralized stochastic optimization
  and gossip algorithms with compressed communication,'' in \emph{International
  Conference on Machine Learning}, 2019, pp. 3478--3487.

\bibitem{yang2019survey}
T.~Yang, X.~Yi, J.~Wu, Y.~Yuan, D.~Wu, Z.~Meng, Y.~Hong, H.~Wang, Z.~Lin, and
  K.~H. Johansson, ``A survey of distributed optimization,'' \emph{Annual
  Reviews in Control}, vol.~47, pp. 278--305, 2019.

\bibitem{hooke1961direct}
R.~Hooke and T.~A. Jeeves, ````{D}irect search'' solution of numerical and
  statistical problems,'' \emph{Journal of the ACM}, vol.~8, no.~2, pp.
  212--229, 1961.

\bibitem{matyas1965random}
J.~Matyas, ``Random optimization,'' \emph{Automation and Remote Control},
  vol.~26, no.~2, pp. 246--253, 1965.

\bibitem{nelder1965simplex}
J.~A. Nelder and R.~Mead, ``A simplex method for function minimization,''
  \emph{The Computer Journal}, vol.~7, no.~4, pp. 308--313, 1965.

\bibitem{bergou2019stochastic}
E.~H. Bergou, E.~Gorbunov, and P.~Richtarik, ``Stochastic three points method
  for unconstrained smooth minimization,'' \emph{arXiv:1902.03591}, 2019.

\bibitem{bibi2019stochastic}
A.~Bibi, E.~H. Bergou, O.~Sener, B.~Ghanem, and P.~Richtarik, ``A stochastic
  derivative-free optimization method with importance sampling: Theory and
  learning to control,'' \emph{arXiv:1902.01272}, 2019.

\bibitem{gorbunov2019stochastic}
E.~Gorbunov, A.~Bibi, O.~Sener, E.~H. Bergou, and P.~Richt{\'a}rik, ``A
  stochastic derivative free optimization method with momentum,'' in
  \emph{International Conference on Learning Representations}, 2020.

\bibitem{golovin2019gradientless}
D.~Golovin, J.~Karro, G.~Kochanski, C.~Lee, X.~Song \emph{et~al.},
  ``Gradientless descent: High-dimensional zeroth-order optimization,''
  \emph{arXiv:1911.06317}, 2019.

\bibitem{marazzi2002wedge}
M.~Marazzi and J.~Nocedal, ``Wedge trust region methods for derivative free
  optimization,'' \emph{Mathematical Programming}, vol.~91, no.~2, pp.
  289--305, 2002.

\bibitem{conn2009global}
A.~R. Conn, K.~Scheinberg, and L.~N. Vicente, ``Global convergence of general
  derivative-free trust-region algorithms to first- and second-order critical
  points,'' \emph{SIAM Journal on Optimization}, vol.~20, no.~1, pp. 387--415,
  2009.

\bibitem{scheinberg2010self}
K.~Scheinberg and P.~L. Toint, ``Self-correcting geometry in model-based
  algorithms for derivative-free unconstrained optimization,'' \emph{SIAM
  Journal on Optimization}, vol.~20, no.~6, pp. 3512--3532, 2010.

\bibitem{duchi2015optimal}
J.~C. Duchi, M.~I. Jordan, M.~J. Wainwright, and A.~Wibisono, ``Optimal rates
  for zero-order convex optimization: The power of two function evaluations,''
  \emph{IEEE Transactions on Information Theory}, vol.~61, no.~5, pp.
  2788--2806, 2015.

\bibitem{shamir2017optimal}
O.~Shamir, ``An optimal algorithm for bandit and zero-order convex optimization
  with two-point feedback,'' \emph{Journal of Machine Learning Research},
  vol.~18, no.~52, pp. 1--11, 2017.

\bibitem{nesterov2017random}
Y.~Nesterov and V.~Spokoiny, ``Random gradient-free minimization of convex
  functions,'' \emph{Foundations of Computational Mathematics}, vol.~17, no.~2,
  pp. 527--566, 2017.

\bibitem{shamir2013complexity}
O.~Shamir, ``On the complexity of bandit and derivative-free stochastic convex
  optimization,'' in \emph{Conference on Learning Theory}, 2013, pp. 3--24.

\bibitem{ghadimi2013stochastic}
S.~Ghadimi and G.~Lan, ``Stochastic first- and zeroth-order methods for
  nonconvex stochastic programming,'' \emph{SIAM Journal on Optimization},
  vol.~23, no.~4, pp. 2341--2368, 2013.

\bibitem{bach2016highly}
F.~Bach and V.~Perchet, ``Highly-smooth zero-th order online optimization,'' in
  \emph{Conference on Learning Theory}, 2016, pp. 257--283.

\bibitem{balasubramanian2018zeroth}
K.~Balasubramanian and S.~Ghadimi, ``Zeroth-order (non)-convex stochastic
  optimization via conditional gradient and gradient updates,'' in
  \emph{Advances in Neural Information Processing Systems}, 2018, pp.
  3455--3464.

\bibitem{jin2018local}
C.~Jin, L.~T. Liu, R.~Ge, and M.~I. Jordan, ``On the local minima of the
  empirical risk,'' in \emph{Advances in Neural Information Processing
  Systems}, 2018, pp. 4896--4905.

\bibitem{ye2018hessian}
H.~Ye, Z.~Huang, C.~Fang, C.~J. Li, and T.~Zhang, ``Hessian-aware zeroth-order
  optimization for black-box adversarial attack,'' \emph{arXiv:1812.11377},
  2018.

\bibitem{vlatakis2019efficiently}
E.-V. Vlatakis-Gkaragkounis, L.~Flokas, and G.~Piliouras, ``Efficiently
  avoiding saddle points with zero order methods: No gradients required,'' in
  \emph{Advances in Neural Information Processing Systems}, 2019, pp.
  10\,066--10\,077.

\bibitem{kozak2020stochastic}
D.~Kozak, S.~Becker, A.~Doostan, and L.~Tenorio, ``A stochastic subspace
  approach to gradient-free optimization in high dimensions,''
  \emph{arXiv:2003.02684}, 2020.

\bibitem{liu2018zerothieee}
S.~Liu, X.~Li, P.-Y. Chen, J.~Haupt, and L.~Amini, ``Zeroth-order stochastic
  projected gradient descent for nonconvex optimization,'' in \emph{IEEE Global
  Conference on Signal and Information Processing}, 2018, pp. 1179--1183.

\bibitem{liu2019signsgd}
S.~Liu, P.-Y. Chen, X.~Chen, and M.~Hong, ``sign{SGD} via zeroth-order
  oracle,'' in \emph{International Conference on Learning Representations},
  2019.

\bibitem{zhang2020improving}
Y.~Zhang, Y.~Zhou, K.~Ji, and M.~M. Zavlanos, ``Improving the convergence rate
  of one-point zeroth-order optimization using residual feedback,''
  \emph{arXiv:2006.10820}, 2020.

\bibitem{lian2016Comprehensive}
X.~Lian, H.~Zhang, C.-J. Hsieh, Y.~Huang, and J.~Liu, ``A comprehensive linear
  speedup analysis for asynchronous stochastic parallel optimization from
  zeroth-order to first-order,'' in \emph{Advances in Neural Information
  Processing Systems}, 2016, pp. 3054--3062.

\bibitem{ghadimi2016mini}
S.~Ghadimi, G.~Lan, and H.~Zhang, ``Mini-batch stochastic approximation methods
  for nonconvex stochastic composite optimization,'' \emph{Mathematical
  Programming}, vol. 155, no. 1-2, pp. 267--305, 2016.

\bibitem{gao2018information}
X.~Gao, B.~Jiang, and S.~Zhang, ``On the information-adaptive variants of the
  {ADMM}: An iteration complexity perspective,'' \emph{Journal of Scientific
  Computing}, vol.~76, no.~1, pp. 327--363, 2018.

\bibitem{kazemi2018proximal}
E.~Kazemi and L.~Wang, ``A proximal zeroth-order algorithm for nonconvex
  nonsmooth problems,'' in \emph{Annual Allerton Conference on Communication,
  Control, and Computing}, 2018, pp. 64--71.

\bibitem{gu2018faster}
B.~Gu, Z.~Huo, C.~Deng, and H.~Huang, ``Faster derivative-free stochastic
  algorithm for shared memory machines,'' in \emph{International Conference on
  Machine Learning}, 2018, pp. 1812--1821.

\bibitem{fang2018spider}
C.~Fang, C.~J. Li, Z.~Lin, and T.~Zhang, ``Spider: Near-optimal non-convex
  optimization via stochastic path-integrated differential estimator,'' in
  \emph{Advances in Neural Information Processing Systems}, 2018, pp. 689--699.

\bibitem{liu2018zeroth}
S.~Liu, B.~Kailkhura, P.-Y. Chen, P.~Ting, S.~Chang, and L.~Amini,
  ``Zeroth-order stochastic variance reduction for nonconvex optimization,'' in
  \emph{Advances in Neural Information Processing Systems}, 2018, pp.
  3727--3737.

\bibitem{gorbunov2018accelerated}
E.~Gorbunov, P.~Dvurechensky, and A.~Gasnikov, ``An accelerated method for
  derivative-free smooth stochastic convex optimization,''
  \emph{arXiv:1802.09022}, 2018.

\bibitem{liu2018stochastic}
L.~Liu, M.~Cheng, C.-J. Hsieh, and D.~Tao, ``Stochastic zeroth-order
  optimization via variance reduction method,'' \emph{arXiv:1805.11811}, 2018.

\bibitem{huang2019faster}
F.~Huang, B.~Gu, Z.~Huo, S.~Chen, and H.~Huang, ``Faster gradient-free proximal
  stochastic methods for nonconvex nonsmooth optimization,'' in \emph{AAAI
  Conference on Artificial Intelligence}, vol.~33, 2019, pp. 1503--1510.

\bibitem{pmlr-v97-ji19a}
K.~Ji, Z.~Wang, Y.~Zhou, and Y.~Liang, ``Improved zeroth-order variance reduced
  algorithms and analysis for nonconvex optimization,'' in \emph{International
  Conference on Machine Learning}, 2019, pp. 3100--3109.

\bibitem{Huang2020Accelerated}
F.~Huang, L.~Tao, and S.~Chen, ``Accelerated stochastic gradient-free and
  projection-free methods,'' in \emph{International Conference on Machine
  Learning}, 2020.

\bibitem{chen2020accelerated}
Y.~Chen, A.~Orvieto, and A.~Lucchi, ``An accelerated {DFO} algorithm for
  finite-sum convex functions,'' in \emph{International Conference on Machine
  Learning}, 2020.

\bibitem{gao2020can}
H.~Gao and H.~Huang, ``Can stochastic zeroth-order {F}rank--{W}olfe method
  converge faster for non-convex problems?'' in \emph{International Conference
  on Machine Learning}, 2020.

\bibitem{cai2020zeroth}
H.~Cai, D.~Mckenzie, W.~Yin, and Z.~Zhang, ``Zeroth-order regularized
  optimization ({ZORO}): Approximately sparse gradients and adaptive
  sampling,'' \emph{arXiv:2003.13001}, 2020.

\bibitem{nazari2020adaptive}
P.~Nazari, D.~A. Tarzanagh, and G.~Michailidis, ``Adaptive first- and
  zeroth-order methods for weakly convex stochastic optimization problems,''
  \emph{arXiv:2005.09261}, 2020.

\bibitem{sahu2019towards}
A.~K. Sahu, M.~Zaheer, and S.~Kar, ``Towards gradient free and projection free
  stochastic optimization,'' in \emph{International Conference on Artificial
  Intelligence and Statistics}, 2019, pp. 3468--3477.

\bibitem{wang2018stochastic}
Y.~Wang, S.~Du, S.~Balakrishnan, and A.~Singh, ``Stochastic zeroth-order
  optimization in high dimensions,'' in \emph{International Conference on
  Artificial Intelligence and Statistics}, 2018, pp. 1356--1365.

\bibitem{chen2019zo}
X.~Chen, S.~Liu, K.~Xu, X.~Li, X.~Lin, M.~Hong, and D.~Cox, ``{ZO}-{A}da{MM}:
  Zeroth-order adaptive momentum method for black-box optimization,'' in
  \emph{Advances in Neural Information Processing Systems}, 2019, pp.
  7204--7215.

\bibitem{huang2019zeroth}
F.~Huang, S.~Gao, S.~Chen, and H.~Huang, ``Zeroth-order stochastic alternating
  direction method of multipliers for nonconvex nonsmooth optimization,'' in
  \emph{International Conference on Artificial Intelligence and Statistics},
  2019, pp. 2549--2555.

\bibitem{huang2019nonconvex}
F.~Huang, S.~Gao, J.~Pei, and H.~Huang, ``Nonconvex zeroth-order stochastic
  {ADMM} methods with lower function query complexity,''
  \emph{arXiv:1907.13463}, 2019.

\bibitem{yuan2014randomized}
D.~Yuan and D.~W. Ho, ``Randomized gradient-free method for multiagent
  optimization over time-varying networks,'' \emph{IEEE Transactions on Neural
  Networks and Learning Systems}, vol.~26, no.~6, pp. 1342--1347, 2014.

\bibitem{sahu2018distributed}
A.~K. Sahu, D.~Jakovetic, D.~Bajovic, and S.~Kar, ``Distributed zeroth order
  optimization over random networks: A {K}iefer--{W}olfowitz stochastic
  approximation approach,'' in \emph{IEEE Conference on Decision and Control},
  2018, pp. 4951--4958.

\bibitem{sahu2018communication}
A.~K. Sahu, D.~Jakoveti{\'c}, D.~Bajovi{\'c}, and S.~Kar,
  ``Communication-efficient distributed strongly convex stochastic
  optimization: Non-asymptotic rates,'' \emph{arXiv preprint arXiv:1809.02920},
  2018.

\bibitem{wang2019distributed}
Y.~Wang, W.~Zhao, Y.~Hong, and M.~Zamani, ``Distributed subgradient-free
  stochastic optimization algorithm for nonsmooth convex functions over
  time-varying networks,'' \emph{SIAM Journal on Control and Optimization},
  vol.~57, no.~4, pp. 2821--2842, 2019.

\bibitem{pang2019randomized}
Y.~Pang and G.~Hu, ``Randomized gradient-free distributed optimization methods
  for a multi-agent system with unknown cost function,'' \emph{IEEE
  Transactions on Automatic Control}, vol.~65, no.~1, pp. 333--340, 2020.

\bibitem{tang2020distributedzero}
Y.~Tang, J.~Zhang, and N.~Li, ``Distributed zero-order algorithms for nonconvex
  multi-agent optimization,'' \emph{arXiv:1908.11444v3}, 2020.

\bibitem{yuan2015gradient}
D.~Yuan, S.~Xu, and J.~Lu, ``Gradient-free method for distributed multi-agent
  optimization via push-sum algorithms,'' \emph{International Journal of Robust
  and Nonlinear Control}, vol.~25, no.~10, pp. 1569--1580, 2015.

\bibitem{yu2019distributed}
Z.~Yu, D.~W. Ho, and D.~Yuan, ``Distributed randomized gradient-free mirror
  descent algorithm for constrained optimization,'' \emph{arXiv:1903.04157},
  2019.

\bibitem{hajinezhad2018gradient}
D.~Hajinezhad and M.~M. Zavlanos, ``Gradient-free multi-agent nonconvex
  nonsmooth optimization,'' in \emph{IEEE Conference on Decision and Control},
  2018, pp. 4939--4944.

\bibitem{hajinezhad2019zone}
D.~Hajinezhad, M.~Hong, and A.~Garcia, ``{ZONE}: Zeroth-order nonconvex
  multiagent optimization over networks,'' \emph{IEEE Transactions on Automatic
  Control}, vol.~64, no.~10, pp. 3995--4010, 2019.

\bibitem{yi2021linear}
X.~Yi, S.~Zhang, T.~Yang, T.~Chai, and K.~H. Johansson, ``Linear convergence of
  first-and zeroth-order primal-dual algorithms for distributed nonconvex
  optimization,'' \emph{IEEE Transactions on Automatic Control}, 2021.

\bibitem{beznosikov2019derivative}
A.~Beznosikov, E.~Gorbunov, and A.~Gasnikov, ``Derivative-free method for
  composite optimization with applications to decentralized distributed
  optimization,'' \emph{arXiv:1911.10645v4}, 2020.

\bibitem{gratton2020privacy}
C.~Gratton, N.~K. Venkategowda, R.~Arablouei, and S.~Werner,
  ``Privacy-preserving distributed zeroth-order optimization,'' \emph{arXiv
  preprint arXiv:2008.13468}, 2020.

\bibitem{sahu2020decentralized}
A.~K. Sahu and S.~Kar, ``Decentralized zeroth-order constrained stochastic
  optimization algorithms: Frank--{W}olfe and variants with applications to
  black-box adversarial attacks,'' \emph{Proceedings of the IEEE}, vol. 108,
  no.~11, pp. 1890--1905, 2020.

\bibitem{lian2017can}
X.~Lian, C.~Zhang, H.~Zhang, C.-J. Hsieh, W.~Zhang, and J.~Liu, ``Can
  decentralized algorithms outperform centralized algorithms? {A} case study
  for decentralized parallel stochastic gradient descent,'' in \emph{Advances
  in Neural Information Processing Systems}, 2017, pp. 5330--5340.

\bibitem{karimi2016linear}
H.~Karimi, J.~Nutini, and M.~Schmidt, ``Linear convergence of gradient and
  proximal-gradient methods under the {P}olyak--{{\L}}ojasiewicz condition,''
  in \emph{Joint European Conference on Machine Learning and Knowledge
  Discovery in Databases}, 2016, pp. 795--811.

\bibitem{zhang2015restricted}
H.~Zhang and L.~Cheng, ``Restricted strong convexity and its applications to
  convergence analysis of gradient-type methods in convex optimization,''
  \emph{Optimization Letters}, vol.~9, no.~5, pp. 961--979, 2015.

\bibitem{shi2015extra}
W.~Shi, Q.~Ling, G.~Wu, and W.~Yin, ``{EXTRA}: An exact first-order algorithm
  for decentralized consensus optimization,'' \emph{SIAM Journal on
  Optimization}, vol.~25, no.~2, pp. 944--966, 2015.

\bibitem{nedic2017geometrically}
A.~Nedi{\'c}, A.~Olshevsky, W.~Shi, and C.~A. Uribe, ``Geometrically convergent
  distributed optimization with uncoordinated step-sizes,'' in \emph{American
  Control Conference}, 2017, pp. 3950--3955.

\bibitem{qu2018harnessing}
G.~Qu and N.~Li, ``Harnessing smoothness to accelerate distributed
  optimization,'' \emph{IEEE Transactions on Control of Network Systems},
  vol.~5, no.~3, pp. 1245--1260, 2018.

\bibitem{qu2017accelerated}
------, ``Accelerated distributed {N}esterov gradient descent,'' \emph{IEEE
  Transactions on Automatic Control}, vol.~65, no.~6, pp. 2566--2581, 2020.

\bibitem{jakovetic2020primal}
D.~Jakoveti{\'c}, D.~Bajovi{\'c}, J.~Xavier, and J.~M. Moura, ``Primal--dual
  methods for large-scale and distributed convex optimization and data
  analytics,'' \emph{Proceedings of the IEEE}, vol. 108, no.~11, pp.
  1923--1938, 2020.

\bibitem{Yu2019on}
H.~Yu, R.~Jin, and S.~Yang, ``On the linear speedup analysis of communication
  efficient momentum {SGD} for distributed non-convex optimization,'' in
  \emph{International Conference on Machine Learning}, 2019, pp. 7184--7193.

\bibitem{nedic2009distributed}
A.~Nedic and A.~Ozdaglar, ``Distributed subgradient methods for multi-agent
  optimization,'' \emph{IEEE Transactions on Automatic Control}, vol.~54,
  no.~1, pp. 48--61, 2009.

\bibitem{goodfellow2014explaining}
I.~J. Goodfellow, J.~Shlens, and C.~Szegedy, ``Explaining and harnessing
  adversarial examples,'' in \emph{International Conference on Learning
  Representations}, 2015.

\bibitem{carlini2017towards}
N.~Carlini and D.~Wagner, ``Towards evaluating the robustness of neural
  networks,'' in \emph{IEEE symposium on security and privacy}, 2017, pp.
  39--57.

\bibitem{Yi2018distributed}
X.~Yi, L.~Yao, T.~Yang, J.~George, and K.~H. Johansson, ``Distributed
  optimization for second-order multi-agent systems with dynamic
  event-triggered communication,'' in \emph{IEEE Conference on Decision and
  Control}, 2018, pp. 3397--3402.

\bibitem{nesterov2018lectures}
Y.~Nesterov, \emph{Lectures on Convex Optimization}, 2nd~ed.\hskip 1em plus
  0.5em minus 0.4em\relax Springer International Publishing, 2018.

\bibitem{yi2019distributed}
X.~Yi, X.~Li, T.~Yang, L.~Xie, T.~Chai, and K.~H. Johansson, ``Distributed
  bandit online convex optimization with time-varying coupled inequality
  constraints,'' \emph{arXiv:1912.03719}, 2019.

\bibitem{kar2012distributed}
S.~Kar, J.~M. Moura, and K.~Ramanan, ``Distributed parameter estimation in
  sensor networks: Nonlinear observation models and imperfect communication,''
  \emph{IEEE Transactions on Information Theory}, vol.~58, no.~6, pp.
  3575--3605, 2012.

\end{thebibliography}










\appendix

\subsection{Notations, Definitions, and Useful Lemmas}\label{zerosg:app-lemmas}

\subsubsection{Notations}
${\bm 1}_n$ (${\bm 0}_n$) denotes the column one (zero) vector of dimension $n$. $\col(z_1,\dots,z_k)$ is the concatenated column vector of vectors $z_i\in\mathbb{R}^{p_i},~i\in[k]$. ${\bm I}_n$ is the $n$-dimensional identity matrix. Given a vector $[x_1,\dots,x_n]^\top\in\mathbb{R}^n$, $\diag([x_1,\dots,x_n])$ is a diagonal matrix with the $i$-th diagonal element being $x_i$.  The notation $A\otimes B$ denotes the Kronecker product
of matrices $A$ and $B$. $\nullrank(A)$ is the null space of matrix $A$.
Given two symmetric matrices $M,N$, $M\ge N$ means that $M-N$ is positive semi-definite. $\rho(\cdot)$ stands for the spectral radius for matrices and $\rho_2(\cdot)$ indicates the minimum
positive eigenvalue for matrices having positive eigenvalues. For any square matrix $A$, $\|x\|_A^2$ denotes $x^\top Ax$. $\lceil \cdot\rceil$ and $\lfloor\cdot\rfloor$ denote the ceiling and floor functions, respectively. For any $x\in\mathbb{R}$, $[x]_+$ is the positive part of $x$. ${\bm 1}_{(\cdot)}$ is the indicator function.

\subsubsection{Graph Theory}
For an undirected graph $\mathcal G=(\mathcal V,\mathcal E)$, let $\mathcal{A}=(a_{ij})$ be the associated weighted adjacency matrix with $a_{ij}>0$ if $(i,j)\in \mathcal E$ if $a_{ij}>0$ and zero otherwise. It is assumed that $a_{ii}=0$ for all $i\in [n]$. Let $\deg_i=\sum\limits_{j=1}^{n}a_{ij}$ denotes the weighted degree of vertex $i$. The degree matrix of graph $\mathcal G$ is $\Deg=\diag([\deg_1, \cdots, \deg_n])$. The Laplacian matrix is $L=(L_{ij})=\Deg-\mathcal{A}$. A  path of length $k$ between vertices $i$ and $j$ is a subgraph with distinct vertices $i_0=i,\dots,i_k=j\in [n]$ and edges $(i_j,i_{j+1})\in\mathcal E,~j=0,\dots,k-1$.
An undirected graph is  connected if there exists at least one path between any two distinct vertices.

For a connected undirected graph, we have the following results.
\begin{lemma}\label{nonconvex:lemma-Xinlei} (Lemmas~1 and 2 in \cite{Yi2018distributed})
Let $L$ be the Laplacian matrix of the connected graph $\mathcal{G}$ and $K_n={\bm I}_n-\frac{1}{n}{\bm 1}_n{\bm 1}^{\top}_n$.
Then $L$ and $K_n$ are positive semi-definite, $\nullrank(L)=\nullrank(K_n)=\{{\bm 1}_n\}$, $L\le\rho(L){\bm I}_n$, $\rho(K_n)=1$,
\begin{subequations}
\begin{align}
&K_nL=LK_n=L,\label{nonconvex:KL-L-eq}\\
&0\le\rho_2(L)K_n\le L\le\rho(L)K_n.\label{nonconvex:KL-L-eq2}
\end{align}
\end{subequations}
Moreover, there exists an orthogonal matrix $[r \ R]\in \mathbb{R}^{n \times n}$ with $r=\frac{1}{\sqrt{n}}\mathbf{1}_n$ and $R \in \mathbb{R}^{n\times (n-1)}$ such that
\begin{subequations}
\begin{align}
&R\Lambda_1^{-1}R^{\top}L=LR\Lambda_1^{-1}R^{\top}=K_n,\label{nonconvex:lemma-eq}\\
&\frac{1}{\rho(L)}K_n\leq R\Lambda_1^{-1}R^{\top}\le\frac{1}{\rho_2(L)}K_n,\label{nonconvex:lemma-eq2}
\end{align}
\end{subequations}
where $\Lambda_1=\diag([\lambda_2,\dots,\lambda_n])$ with $0<\lambda_2\leq\dots\leq\lambda_n$ being the eigenvalues of the Laplacian matrix $L$.
\end{lemma}

\subsubsection{Smooth Functions}
\begin{definition}\cite{nesterov2018lectures}
A function $f(x):~\mathbb{R}^p\mapsto\mathbb{R}$ is smooth with constant $L_f>0$ if it is differentiable and
\begin{align}\label{nonconvex:smooth}
\|\nabla f(x)-\nabla f(y)\|\le L_{f}\|x-y\|,~\forall x,y\in \mathbb{R}^p.
\end{align}
\end{definition}
From Lemma~1.2.3 in \cite{nesterov2018lectures}, we know that \eqref{nonconvex:smooth} implies
\begin{align}
&|f(y)-f(x)-(y-x)^\top\nabla f(x)|
\le\frac{L_f}{2}\|y-x\|^2,~\forall x,y\in\mathbb{R}^{p}, \label{nonconvex:lemma:lipschitz}
\end{align}
which further implies
\begin{align}
\|\nabla f(x)\|^2\le2L_f(f(x)-f^*),~\forall x,y\in\mathbb{R}^{p},\label{nonconvex:lemma:lipschitz2}
\end{align}
where $f^*=\min_{x\in\mathbb{R}^p}f(x)$.

\subsubsection{Properties of Gradient Approximation}
The random gradient estimator $\hat{\nabla}_2f$ defined in \eqref{dbco:gradient:model2} is an unbiased gradient estimator of $f^s$, where $f^s$ is the uniformly smoothed version of $f$ defined as
\begin{align}\label{dbco:equ:uniformsmoothing}
f^s(x,\delta)=\mathbf{E}_{u\in\mathbb{B}^p}[f(x+\delta u)],
\end{align}
with the expectation is taken with respect to uniform distribution.

From Lemma~2 in \cite{yi2019distributed}, Lemma~5 in \cite{tang2020distributedzero}, and Proposition~7.6 in \cite{gao2018information}, we have the following properties of $f^s$ and $\hat{\nabla}_2f$.
\begin{lemma}\label{dbco:lemma:uniformsmoothing}
\begin{enumerate}[label=(\roman*)]
\item The uniform smoothing $f^s(x,\delta)$ is differentiable with respect to $x$, and
\begin{align}
\nabla f^s(x,\delta)=\mathbf{E}_{u\in\mathbb{S}^p}[\hat{\nabla}_2f(x,\delta,u)].
\label{zerosg:lemma:uniformsmoothing-equ1}
\end{align}
\item If $f$ is smooth with constant $L_f>0$, then
\begin{subequations}
\begin{align}
&\|\nabla f^s(x,\delta)-\nabla f(x)\|\le\delta L_f,\label{zerosg:lemma:uniformsmoothing-equ8}\\
&\mathbf{E}_{u\in\mathbb{S}^p}[\|\hat{\nabla}_2f(x,\delta,u)\|^2]\le 2p\|\nabla f(x)\|^2+\frac{1}{2}p^2\delta^2 L_f^2.\label{zerosg:lemma:uniformsmoothing-equ6-1}
\end{align}
\end{subequations}
\end{enumerate}
\end{lemma}

\subsubsection{Useful Lemmas on Series}
\begin{lemma}\label{zerosg:lemma:sumgeo}
Let $a,b\in(0,1)$ be two constants, then
\begin{align}\label{zerosg:lemma:sumgeo-equ}
\sum_{\tau=0}^{k}a^\tau b^{k-\tau}\le
\begin{cases}
  \frac{a^{k+1}}{a-b}, & \mbox{if } a>b \\
  \frac{b^{k+1}}{b-a}, & \mbox{if } a<b \\
  \frac{c^{k+1}}{c-b}, & \mbox{if } a=b ,
\end{cases}
\end{align}
where $c$ is any constant in $(a,1)$.
\end{lemma}
\begin{proof}
If $a>b$, then
\begin{align*}
\sum_{\tau=0}^{k}a^\tau b^{k-\tau}=a^k\sum_{\tau=0}^{k}\Big(\frac{b}{a}\Big)^{k-\tau}
\le \frac{a^{k+1}}{a-b}.
\end{align*}
Similarly, when $a<b$, we have
\begin{align*}
\sum_{\tau=0}^{k}a^\tau b^{k-\tau}=b^k\sum_{\tau=0}^{k}\Big(\frac{a}{b}\Big)^\tau
\le \frac{b^{k+1}}{b-a}.
\end{align*}
If $a=b$, then for any $c\in(a,1)$, we have
\begin{align*}
\sum_{\tau=0}^{k}a^\tau b^{k-\tau}\le\sum_{\tau=0}^{k}c^\tau b^{k-\tau}=c^k\sum_{\tau=0}^{k}\Big(\frac{b}{c}\Big)^{k-\tau}
\le \frac{c^{k+1}}{c-b}.
\end{align*}
Hence, this lemma holds.
\end{proof}

\begin{lemma}\label{zerosg:serise:lemma:sum}
Let $k$ and $\tau$ be two integers and $\delta$ be a constant. Suppose $k\ge\tau\ge1$, then
\begin{align}\label{zerosg:serise:lemma:sum-equ}
\sum_{l=\tau}^{k}l^\delta\le
\begin{cases}
  \frac{(k+1)^{\delta+1}}{\delta+1}, & \mbox{if } \delta>-1 \\
  \ln(k), & \mbox{if } \delta=-1 \\
  \frac{-(\tau-1)^{\delta+1}}{\delta+1}, & \mbox{if } \delta<-1~\text{and}~\tau\ge2.
\end{cases}
\end{align}
\end{lemma}
\begin{proof}
If $\delta\ge0$, then $h(t)=t^\delta$ is an increasing function in the interval $[1,+\infty)$. Hence,
\begin{align}\label{zerosg:serise:lemma:sum-proof1}
\sum_{l=\tau}^{k}l^\delta\le\int_{\tau}^{k+1}t^\delta dt
=\frac{(k+1)^{\delta+1}-\tau^{\delta+1}}{\delta+1}
\le\frac{(k+1)^{\delta+1}}{\delta+1}.
\end{align}

If $\delta<0$, then $h(t)=t^\delta$ is a decreasing function in the interval $[1,+\infty)$. Hence,
\begin{align}\label{zerosg:serise:lemma:sum-proof2}
\sum_{l=\tau}^{k}l^\delta&\le\int_{\tau-1}^{k}t^\delta dt
=\begin{cases}
  \ln(\frac{k}{\tau-1}), & \mbox{if } \delta=-1, \\
  \frac{k^{\delta+1}-(\tau-1)^{\delta+1}}{\delta+1}, & \mbox{if }-1<\delta<0,\\
  \frac{k^{\delta+1}-(\tau-1)^{\delta+1}}{\delta+1}, & \mbox{if }\delta<-1~\text{and}~\tau\ge2,
\end{cases}
\nonumber\\
&\le\begin{cases}
  \ln(k), & \mbox{if } \delta=-1, \\
  \frac{(k+1)^{\delta+1}}{\delta+1}, & \mbox{if }-1<\delta<0,\\
  \frac{-(\tau-1)^{\delta+1}}{\delta+1}, & \mbox{if }\delta<-1~\text{and}~\tau\ge2.
\end{cases}
\end{align}

Finally, \eqref{zerosg:serise:lemma:sum-proof1} and \eqref{zerosg:serise:lemma:sum-proof2} yield \eqref{zerosg:serise:lemma:sum-equ}.
\end{proof}

\begin{lemma}\label{zerosg:serise:lemma:sequence} 
Let $\{z_k\}$, $\{r_{1,k}\}$, and $\{r_{2,k}\}$ be sequences. Suppose there exists $t_1\in\mathbb{N}_+$ such that
\begin{subequations}
\begin{align}
&z_k\ge0,\label{zerosg:serise:lemma:sequence-equ0}\\
&z_{k+1}\le(1-r_{1,k})z_k+r_{2,k},\label{zerosg:serise:lemma:sequence-equ1}\\
&1> r_{1,k}\ge\frac{a_1}{(k+t_1)^{\delta_1}},\label{zerosg:serise:lemma:sequence-equ2}\\
&r_{2,k}\le\frac{a_2}{(k+t_1)^{\delta_2}},~\forall k\in\mathbb{N}_0, \label{zerosg:serise:lemma:sequence-equ3}
\end{align}
\end{subequations}
where $a_1>0$, $a_2>0$, $\delta_1\in[0,1]$, and $\delta_2>\delta_1$ are constants.

(i) If $\delta_1\in(0,1)$, then
\begin{align}\label{zerosg:serise:lemma:sequence-equ4}
z_{k}\le\phi_1(k,t_1,a_1,a_2,\delta_1,\delta_2,z_{0}),~\forall k\in\mathbb{N}_+,
\end{align}
where
\begin{align}\label{zerosg:serise:lemma:sequence-equ4-phi2}
\phi_1(k,t_1,a_1,a_2,\delta_1,\delta_2,z_{0})
&=\frac{1}{s_1(k+t_1)}\Big(s_1(t_1)z_{0}
+\frac{[t_2-1-t_1]_+s_1(t_1+1)a_2}{t_1^{\delta_2}}\Big)\nonumber\\
&\quad+\frac{a_2}{(k+t_1-1)^{\delta_2}}+\frac{{\bm 1}_{(k+t_1-1\ge t_2)}(\frac{t_1+1}{t_1})^{\delta_2}a_2\delta_2}{a_1\delta_1(k+t_1)^{\delta_2-\delta_1}},
\end{align}
$s_1(k)=e^{\frac{a_1}{1-\delta_1}k^{1-\delta_1}}$ and $t_2=\lceil(\frac{\delta_2}{a_1})^{\frac{1}{1-\delta_1}}\rceil$.

(ii) If $\delta_1=1$, then
\begin{align}\label{zerosg:serise:lemma:sequence-equ5}
z_{k}&\le \phi_2(k,t_1,a_1,a_2,\delta_2,z_{0}),~\forall k\in\mathbb{N}_+,
\end{align}
where
\begin{align}\label{zerosg:serise:lemma:sequence-equ5-phi3}
\phi_2(k,t_1,a_1,a_2,\delta_2,z_{0})&=\frac{t_1^{a_1}z_{0}}{(k+t_1)^{a_1}}
+\frac{a_2}{(k+t_1-1)^{\delta_2}}
+\Big(\frac{t_1+1}{t_1}\Big)^{\delta_2}a_2s_2(k+t_1),
\end{align}
and
\begin{align*}
s_2(k)=
\begin{cases}
  \frac{1}{(a_1-\delta_2+1)k^{\delta_2-1}}, & \mbox{if } a_1-\delta_2>-1, \\
  \frac{\ln(k-1)}{k^{a_1}}, & \mbox{if } a_1-\delta_2=-1, \\
  \frac{-t_1^{a_1-\delta_2+1}}{(a_1-\delta_2+1)k^{a_1}}, & \mbox{if } a_1-\delta_2<-1.
\end{cases}
\end{align*}

(iii) If $\delta_1=0$, then
\begin{align}\label{zerosg:serise:lemma:sequence-equ6}
z_{k}&\le \phi_3(k,t_1,a_1,a_2,\delta_2,z_{0}),~\forall k\in\mathbb{N}_+,
\end{align}
where
\begin{align}\label{zerosg:serise:lemma:sequence-equ6-phi4}
\phi_3(k,t_1,a_1,a_2,\delta_2,z_{0})
&=(1-a_1)^kz_{0}+a_2(1-a_1)^{k+t_1-1}\Big([t_3-t_1]_+s_3(t_1)\nonumber\\
&\quad+([t_4-t_1]_+-[t_3-t_1]_+)s_3(t_4)\Big)\nonumber\\
&\quad+\frac{{\bm 1}_{(k+t_1-1\ge t_4)}2a_2}{-\ln(1-a_1)(k+t_1)^{\delta_2}(1-a_1)},
\end{align}
$s_3(k)=\frac{1}{k^{\delta_2}(1-a_1)^{k}}$, $t_3=\lceil \frac{-\delta_2}{\ln(1-a_1)}\rceil$, and $t_4=\lceil \frac{-2\delta_2}{\ln(1-a_1)}\rceil$.
\end{lemma}
\begin{proof}
This proof is inspired by the proof of Lemma~25 in \cite{kar2012distributed}.

From \eqref{zerosg:serise:lemma:sequence-equ0}--\eqref{zerosg:serise:lemma:sequence-equ2}, for any $k\in\mathbb{N}_+$, it holds that
\begin{align}\label{zerosg:serise:lemma:seqproof1}
z_{k}&\le \prod_{\tau=0}^{k-1}(1-r_{1,\tau})z_{0}+r_{2,k-1}
+\sum_{l=0}^{k-2}\prod_{\tau=l+1}^{k-1}(1-r_{1,\tau})r_{2,l}.
\end{align}


For any $t\in[0, 1]$, it holds that $1-t\le e^{-t}$ since $s_4(t)=1-t-e^{-t}$ is a non-increasing function in the interval $[0, 1]$ and $s_4(0)=0$. Thus, for any $k>l\ge 0$, it holds that
\begin{align}\label{zerosg:serise:lemma:seqproof2}
\prod_{\tau=l}^{k-1}(1-r_{1,\tau})
\le e^{-\sum_{\tau=l}^{k-1}r_{1,\tau}}.
\end{align}

We also have
\begin{align}\label{zerosg:serise:lemma:seqproof3}
\sum_{\tau=l}^{k-1}r_{1,\tau}&\ge\sum_{\tau=l}^{k-1}\frac{a_1}{(\tau+t_1)^{\delta_1}}
=\sum_{\tau=l+t_1}^{k-1+t_1}\frac{a_1}{\tau^{\delta_1}}\ge\int_{t=l+t_1}^{k+t_1}\frac{a_1}{t^{\delta_1}}dt\nonumber\\
&=\begin{cases}
   \frac{a_1}{1-\delta_1}((k+t_1)^{1-\delta_1}-(l+t_1)^{1-\delta_1}), & \mbox{if } \delta_1\in(0,1), \\
   a_1\ln(\frac{k+t_1}{l+t_1}), & \mbox{if } \delta_1=1,
 \end{cases}
\end{align}
where the first inequality holds due to \eqref{zerosg:serise:lemma:sequence-equ2} and the second inequality holds since $s_5(t)=\frac{a_1}{t^{\delta_1}}$ is a decreasing function in the interval $[1, +\infty)$.

Hence, \eqref{zerosg:serise:lemma:seqproof2} and \eqref{zerosg:serise:lemma:seqproof3} yield
\begin{align}\label{zerosg:serise:lemma:seqproof4}
&\prod_{\tau=l}^{k-1}(1-r_{1,\tau})
\le e^{-\sum_{\tau=l}^{k-1}r_{1,\tau}}\le\begin{cases}
  \frac{s_1(l+t_1)}{s_1(k+t_1)}, & \mbox{if } \delta_1\in(0,1), \\
   \frac{(l+t_1)^{a_1}}{(k+t_1)^{a_1}}, & \mbox{if } \delta_1=1.
    \end{cases}
\end{align}

(i) When $\delta_1\in(0,1)$, from \eqref{zerosg:serise:lemma:seqproof4} and \eqref{zerosg:serise:lemma:sequence-equ3}, we have
\begin{align}\label{zerosg:serise:lemma:seqproof5}
\sum_{l=0}^{k-2}\prod_{\tau=l+1}^{k-1}(1-r_{1,\tau})r_{2,l}
&\le\sum_{l=0}^{k-2}\frac{s_1(l+t_1+1)}{s_1(k+t_1)}\frac{a_2}{(l+t_1)^{\delta_2}}\nonumber\\
&=\frac{a_2}{s_1(k+t_1)}
\sum_{l=0}^{k-2}\frac{s_1(l+t_1+1)}{(l+t_1)^{\delta_2}}\nonumber\\
&\le \frac{a_2}{s_1(k+t_1)}
\sum_{l=0}^{k-2}\frac{s_1(l+t_1+1)}{(\frac{t_1}{t_1+1}l+t_1)^{\delta_2}}\nonumber\\
&=\frac{(\frac{t_1+1}{t_1})^{\delta_2}a_2}{s_1(k+t_1)}
\sum_{l=0}^{k-2}\frac{s_1(l+t_1+1)}{(l+t_1+1)^{\delta_2}}\nonumber\\
&=\frac{(\frac{t_1+1}{t_1})^{\delta_2}a_2}{s_1(k+t_1)}
\sum_{l=t_1+1}^{k+t_1-1}\frac{s_1(l)}{l^{\delta_2}}\nonumber\\
&=\frac{(\frac{t_1+1}{t_1})^{\delta_2}a_2}{s_1(k+t_1)}
\Big(\sum_{l=t_1+1}^{t_2-1}\frac{s_1(l)}{l^{\delta_2}}
+\sum_{l=t_2}^{k+t_1-1}\frac{s_1(l)}{l^{\delta_2}}\Big).
\end{align}

We know that $s_6(t)=\frac{s_1(t)}{t^{\delta_2}}$ is a decreasing function in the interval $[1,t_2-1]$ since
\begin{align*}
\frac{ds_6(t)}{dt}=\Big(a_1-\frac{\delta_2}{t^{1-\delta_1}}\Big)\frac{s_6(t)}{t^{\delta_1}}
\le0,~\forall t\in\Big(0,\Big(\frac{\delta_2}{a_1}\Big)^{\frac{1}{1-\delta_1}}\Big].
\end{align*} Thus, for any $k\in[1,t_2-1]$, we have
\begin{align}\label{zerosg:serise:lemma:seqproof8.1}
\sum_{l=k}^{t_2-1}\frac{s_1(l)}{l^{\delta_2}}
\le(t_2-k)\frac{s_1(k)}{k^{\delta_2}}.
\end{align}

Noting that $s_6(t)=\frac{s_1(t)}{t^{\delta_2}}$ is an increasing function in the interval $[t_2,+\infty)$, for any $k\ge t_2$, we have
\begin{align}\label{zerosg:serise:lemma:seqproof6}
\sum_{l=t_2}^{k}\frac{s_1(l)}{l^{\delta_2}}
\le\int_{t_2}^{k+1}\frac{s_1(t)}{t^{\delta_2}}dt.
\end{align}

We have
\begin{align}\label{zerosg:serise:lemma:seqproof7}
\int_{t_2}^{k+1}\frac{s_1(t)}{t^{\delta_2}}dt
&=\int_{t_2}^{k+1}\frac{1}{a_1t^{\delta_2-\delta_1}}ds_1(t)\nonumber\\
&=\frac{s_1(k+1)}{a_1(k+1)^{\delta_2-\delta_1}}
-\frac{s_1(t_2)}{a_1t_2^{\delta_2-\delta_1}}
+\int_{t_2}^{k+1}\frac{(\delta_2-\delta_1)s_1(t)}{a_1t^{\delta_2-\delta_1+1}}dt\nonumber\\
&\le\frac{s_1(k+1)}{a_1(k+1)^{\delta_2-\delta_1}}
+\int_{t_2}^{k+1}\frac{(\delta_2-\delta_1)}{a_1t^{1-\delta_1}}
\frac{s_1(t)}{t^{\delta_2}}dt\nonumber\\
&\le\frac{s_1(k+1)}{a_1(k+1)^{\delta_2-\delta_1}}
+\frac{\delta_2-\delta_1}{a_1t_2^{1-\delta_1}}
\int_{t_2}^{k+1}\frac{s_1(t)}{t^{\delta_2}}dt\nonumber\\
&\le\frac{s_1(k+1)}{a_1(k+1)^{\delta_2-\delta_1}}
+\frac{\delta_2-\delta_1}{\delta_2}
\int_{t_2}^{k+1}\frac{s_1(t)}{t^{\delta_2}}dt,
\end{align}
where the second inequality holds since $s_7(t)=\frac{1}{t^{1-\delta_1}}$ is a decreasing function in the interval $[1, +\infty)$; and the last inequality holds due to $t_2^{1-\delta_1}\ge\frac{\delta_2}{a_1}$.

From \eqref{zerosg:serise:lemma:seqproof6} and \eqref{zerosg:serise:lemma:seqproof7}, for any $k\ge t_2$, we have
\begin{align}\label{zerosg:serise:lemma:seqproof8}
\sum_{l=t_2}^{k}\frac{s_1(l)}{l^{\delta_2}}
\le\int_{t_2}^{k+1}\frac{s_1(t)}{t^{\delta_2}}dt
\le\frac{\delta_2s_1(k+1)}{a_1\delta_1(k+1)^{\delta_2-\delta_1}}.
\end{align}

From \eqref{zerosg:serise:lemma:seqproof5}, \eqref{zerosg:serise:lemma:seqproof8.1}, and \eqref{zerosg:serise:lemma:seqproof8}, we have
\begin{align}\label{zerosg:serise:lemma:seqproof5.1}
&\sum_{l=0}^{k-2}\prod_{\tau=l+1}^{k-1}(1-r_{1,\tau})r_{2,l}\nonumber\\
&\le\frac{(\frac{t_1+1}{t_1})^{\delta_2}a_2}{s_1(k+t_1)}
\Big(\frac{[t_2-1-t_1]_+s_1(t_1+1)}{(t_1+1)^{\delta_2}}
+{\bm 1}_{(k+t_1-1\ge t_2)}\frac{\delta_2s_1(k+t_1)}{a_1\delta_1(k+t_1)^{\delta_2-\delta_1}}\Big).
\end{align}

Then, \eqref{zerosg:serise:lemma:seqproof1},  \eqref{zerosg:serise:lemma:seqproof4}, and \eqref{zerosg:serise:lemma:seqproof5.1} yield
\eqref{zerosg:serise:lemma:sequence-equ4}.

(ii) When $\delta_1=1$, from \eqref{zerosg:serise:lemma:seqproof4} and \eqref{zerosg:serise:lemma:sequence-equ3}, we have
\begin{align}\label{zerosg:serise:lemma:seqproof9}
\sum_{l=0}^{k-2}\prod_{\tau=l+1}^{k-1}(1-r_{1,\tau})r_{2,l}
&\le\sum_{l=0}^{k-2}\frac{(l+t_1+1)^{a_1}}{(k+t_1)^{a_1}}
\frac{a_2}{(l+t_1)^{\delta_2}}\nonumber\\
&\le\sum_{l=0}^{k-2}\frac{(l+t_1+1)^{a_1}}{(k+t_1)^{a_1}}
\frac{a_2}{(\frac{t_1}{t_1+1}l+t_1)^{\delta_2}}\nonumber\\
&=\frac{(\frac{t_1+1}{t_1})^{\delta_2}a_2}{(k+t_1)^{a_1}}
\sum_{l=0}^{k-2}\frac{(l+t_1+1)^{a_1}}{(l+t_1+1)^{\delta_2}}\nonumber\\
&=\frac{(\frac{t_1+1}{t_1})^{\delta_2}a_2}{(k+t_1)^{a_1}}
\sum_{l=t_1+1}^{k+t_1-1}l^{a_1-\delta_2},
\end{align}
where the first inequality holds due to \eqref{zerosg:serise:lemma:seqproof4} and \eqref{zerosg:serise:lemma:sequence-equ3}.

From \eqref{zerosg:serise:lemma:seqproof1}, \eqref{zerosg:serise:lemma:seqproof4}, \eqref{zerosg:serise:lemma:seqproof9}, and \eqref{zerosg:serise:lemma:sum-equ}, we have \eqref{zerosg:serise:lemma:sequence-equ5}.

(iii) Denote $a=1-a_1$.  From \eqref{zerosg:serise:lemma:sequence-equ2} and $\delta_1=0$, we know that $a_1\in(0,1)$. Thus, $a\in(0,1)$.

From \eqref{zerosg:serise:lemma:sequence-equ0}--\eqref{zerosg:serise:lemma:sequence-equ3} and $\delta_1=0$, for any $k\in\mathbb{N}_+$, it holds that
\begin{align}\label{zerosg:serise:lemma:seqproof1-0}
z_{k}&\le (1-a_1)^kz_{0}+\sum_{\tau=0}^{k-1}(1-a_1)^{k-1-\tau}r_{2,\tau}\nonumber\\
&\le a^kz_{0}+a_2a^{k+t_1-1}\sum_{\tau=0}^{k-1}\frac{1}{(\tau+t_1)^{\delta_2}a^{\tau+t_1}}.
\end{align}

We have
\begin{align}\label{zerosg:serise:lemma:seqproof1.1-0}
&\sum_{\tau=0}^{k-1}\frac{1}{(\tau+t_1)^{\delta_2}a^{\tau+t_1}}
=\sum_{\tau=t_1}^{k+t_1-1}\frac{1}{\tau^{\delta_2}a^{\tau}}
=\sum_{\tau=t_1}^{t_3-1}s_3(\tau)
+\sum_{\tau=t_3}^{t_4-1}s_3(\tau)
+\sum_{\tau=t_4}^{k+t_1-1}s_3(\tau).
\end{align}

We know that $s_3(t)=\frac{1}{t^{\delta_2}a^{t}}$ is decreasing and increasing in the intervals $[1,t_3-1]$ and $[t_3,+\infty)$, respectively, since
\begin{align*}
\frac{ds_3(t)}{dt}&=-s_3(t)\Big(\frac{\delta_2}{t}+\ln(a)\Big)\le0,~\forall t\in\Big(0,\frac{-\delta_2}{\ln(a)}\Big],\\
\frac{ds_3(t)}{dt}&=-s_3(t)\Big(\frac{\delta_2}{t}+\ln(a)\Big)\ge0,~\forall t\in\Big[\frac{-\delta_2}{\ln(a)},+\infty\Big).
\end{align*}
Thus, we have
\begin{subequations}
\begin{align}
&\sum_{\tau=k_1}^{t_3-1}s_3(\tau)
\le(t_3-k_1)s_3(k_1),~\forall k_1\in[1,t_3-1],\label{zerosg:serise:lemma:seqproof1.2-0}\\
&\sum_{\tau=k_2}^{t_4-1}s_3(\tau)
\le(t_4-k_2)s_3(t_4),~\forall k_2\in[t_3,t_4-1],\label{zerosg:serise:lemma:seqproof1.3-0}\\
&\sum_{\tau=t_4}^{k_3}s_3(\tau)
\le\int_{t_4}^{k_3+1}s_3(t)dt,~\forall k_3\ge t_4.\label{zerosg:serise:lemma:seqproof1.4-0}
\end{align}
\end{subequations}

Denote $b=\frac{1}{a}$. We have
\begin{align}\label{zerosg:serise:lemma:seqproof1.5-0}
\int_{t_4}^{k_3+1}s_3(t)dt&=\int_{t_4}^{k_3+1}\frac{b^t}{t^{\delta_2}}dt
=\int_{t_4}^{k_3+1}\frac{1}{\ln(b) t^{\delta_2}}db^{t}\nonumber\\
&=\frac{b^{k_3+1}}{\ln(b) (k_3+1)^{\delta_2}}-\frac{b^{t_4}}{\ln(b) t_4^{\delta_2}}
+\int_{t_4}^{k_3+1}\frac{\delta_2b^{t}}{\ln(b) t^{\delta_2+1}}dt\nonumber\\
&\le\frac{b^{k_3+1}}{\ln(b) (k_3+1)^{\delta_2}}
+\int_{t_4}^{k_3+1}\frac{\delta_2}{\ln(b) t}s_3(t)dt\nonumber\\
&\le\frac{b^{k_3+1}}{\ln(b) (k_3+1)^{\delta_2}}
+\frac{\delta_2}{\ln(b) t_4}\int_{t_4}^{k_3+1}s_3(t)dt\nonumber\\
&\le\frac{b^{k_3+1}}{\ln(b) (k_3+1)^{\delta_2}}
+\frac{1}{2}\int_{t_4}^{k_3+1}s_3(t)dt,
\end{align}
where the last inequality holds due to $t_4=\lceil \frac{-2\delta_2}{\ln(1-a_1)}\rceil\ge\frac{-2\delta_2}{\ln(1-a_1)}
=\frac{2\delta_2}{\ln(b)}$.

From \eqref{zerosg:serise:lemma:seqproof1.4-0} and \eqref{zerosg:serise:lemma:seqproof1.5-0}, we have
\begin{align}
&\sum_{\tau=t_4}^{k_3}s_3(\tau)
\le\frac{2}{-\ln(a) (k_3+1)^{\delta_2}a^{k_3+1}},~\forall k_3\ge t_4.\label{zerosg:serise:lemma:seqproof1.6-0}
\end{align}


From \eqref{zerosg:serise:lemma:seqproof1-0}, \eqref{zerosg:serise:lemma:seqproof1.1-0}, \eqref{zerosg:serise:lemma:seqproof1.2-0}, \eqref{zerosg:serise:lemma:seqproof1.3-0}, and \eqref{zerosg:serise:lemma:seqproof1.6-0}, we get \eqref{zerosg:serise:lemma:sequence-equ6}.
\end{proof}

\subsection{Proof of Theorem~\ref{zerosg:thm-random-pd-sm}}\label{zerosg:proof-thm-random-pd-sm}


Denote $\bsL=L\otimes {\bm I}_p$, $\bsK=K_n\otimes {\bm I}_p$, $\bsH=\frac{1}{n}({\bm 1}_n{\bm 1}_n^\top\otimes{\bm I}_p)$, $\bsQ=R\Lambda^{-1}_1R^{\top}\otimes {\bm I}_p$, $\delta_{k}=\max_{i\in[n]}\{\delta_{i,k}\}$, $\bsx=\col(x_1,\dots,x_n)$, $\tilde{f}(\bsx)=\sum_{i=1}^{n}f_i(x_i)$,  $\bar{x}_k=\frac{1}{n}({\bm 1}_n^\top\otimes{\bm I}_p)\bsx_k$, $\bar{\bsx}_k={\bm 1}_n\otimes\bar{x}_k$, $\bsg_k=\nabla\tilde{f}(\bsx_k)$, $\bar{\bsg}_k=\bsH\bsg_{k}$, $\bsg^0_k=\nabla\tilde{f}(\bar{\bsx}_k)$, $\bar{\bsg}_k^0=\bsH\bsg^0_{k}={\bm 1}_n\otimes\nabla f(\bar{x}_k)$, $\bsg^e_k=\col(g^e_{1,k},\dots,g^e_{n,k})$, $\bar{g}^e_k=\frac{1}{n}({\bm 1}_n^\top\otimes{\bm I}_p)\bsg^e_k$,  $\bar{\bsg}^e_k={\bm 1}_n\otimes\bar{g}^e_k=\bsH\bsg^e_k$, $f^s_{i}(x,\delta_{i,k})=\mathbf{E}_{u\in\mathbb{B}^p}[f_i(x+\delta_{i,k} u)]$, $g^s_{i,k}=\nabla f^s_{i}(x_{i,k},\delta_{i,k})$, $\bsg^s_k=\col(g^s_{1,k},\dots,g^s_{n,k})$, and $\bar{\bsg}^s_k=\bsH\bsg^s_k$.

We also denote the following notations.
\begin{align*}
&c_0(\kappa_1,\kappa_2)=\max\Big\{\varepsilon_{1},~\frac{2\varepsilon_5}{\varepsilon_4},
~\Big(\frac{2p(1+\sigma_0^2)(1+\tilde{\sigma}_0^2)\varepsilon_7}{\varepsilon_4}\Big)^{\frac{1}{2}},
~\frac{\varepsilon_8}{2\varepsilon_6},
~\frac{24(1+\tilde{\sigma}_0^2)\kappa_4}{\kappa_2},
~128p(1+\sigma_0^2)(1+\tilde{\sigma}_0^2)\kappa_2\varepsilon_{10}\Big\},\\
&c_1=\frac{1}{\rho_2(L)}+1,\\
&c_2(\kappa_1)=\min\Big\{\frac{\varepsilon_2}{\varepsilon_3},~\frac{1}{5}\Big\},\\
&c_3(\kappa_1,\kappa_2)=\frac{24(1+\tilde{\sigma}_0^2)\kappa_3}{\kappa_2},\\
&\kappa_3=\frac{1}{\rho_2(L)}+\kappa_1+1,\\
&\kappa_4=\frac{1}{\rho_2(L)}+\kappa_1,\\
&\kappa_5=\frac{1}{\rho_2(L)}+\kappa_1+\frac{3}{2},\\
&\kappa_6=\frac{\kappa_1+1}{2}+\frac{1}{2\rho_2(L)},\\
&\kappa_7=\min\Big\{\frac{1}{2\rho(L)},~\frac{\kappa_1-1}{2\kappa_1}\Big\},\\
&\varepsilon_{1}=\max\{1+3L_f^2,
~(8+8p(1+\sigma_0^2)(1+\tilde{\sigma}_0^2)(6+L_f))^{\frac{1}{2}}L_f,~p\kappa_3\},\\
&\varepsilon_2=(\kappa_1-1)\rho_2(L)-1,\\
&\varepsilon_3=\rho(L)+(2\kappa_1^2+1)\rho(L^2)+1,\\
&\varepsilon_4=\frac{1}{2}(\varepsilon_2\kappa_2-\varepsilon_3\kappa_2^2),\\
&\varepsilon_5=\frac{1}{2}-\kappa_1\kappa_2\rho_2(L)+\kappa_2^2\rho(L)
+\frac{1}{2}(1+3\kappa_1\kappa_2+2\kappa_2)\kappa_1\kappa_2\rho(L^2),\\
&\varepsilon_6=\frac{1}{4}(\kappa_2-5\kappa_2^2),\\
&\varepsilon_7
=8(7+6\kappa_2+2\kappa_4+10\kappa_2\kappa_4)\kappa_2L_f^4
+\frac{(1+2L_f^2)\kappa_2}{2p(1+\sigma_0^2)(1+\tilde{\sigma}_0^2)}
+\Big(\frac{5}{p(1+\sigma_0^2)(1+\tilde{\sigma}_0^2)}+24\Big)L_f^2\kappa_2^2,\\
&\varepsilon_8=\kappa_4+\kappa_1\kappa_2+3\kappa_2^2+\kappa_2\kappa_4,\\
&\varepsilon_9=\frac{3\kappa_0}{2\kappa_2^2}(2\kappa_4+1),\\
&\varepsilon_{10}=10+L_f+\frac{1}{\kappa_2}(2\kappa_4+1)L_f^2
+(10\kappa_4+6)L_f^2,\\
&\varepsilon_{11}=L_f^2\Big(\frac{1}{512(1+\sigma_0^2)(1+\tilde{\sigma}_0^2)}
+\frac{13\kappa_2+4}{p}\Big),\\
&\varepsilon_{12}=2\varepsilon_{10}\sigma^2_1+\frac{\varepsilon_9\sigma^2_2}{p}
+4(1+\sigma_0^2)\varepsilon_{10}\sigma^2_2,\\
&\varepsilon_{13}=\frac{(1+\tilde{\sigma}_0^2)\varepsilon_9}{p}
+8(1+\sigma_0^2)(1+\tilde{\sigma}_0^2)\varepsilon_{10},\\
&\varepsilon_{14}=\frac{W_{0}}{n}
+\frac{p(\varepsilon_{11}\kappa_\delta^2
+\varepsilon_{12})\kappa_2^2}{(2\theta-1)\kappa_0^2},\\
&\varepsilon_{15}=2(\sigma^2_1+2(1+\sigma_0^2)\sigma^2_2)L_f,\\
&a_1=\frac{1}{\kappa_6}\min\{\varepsilon_{4},~\varepsilon_{6}\},\\
&a_2=pn(\varepsilon_{11}\kappa_\delta^2+\varepsilon_{12}
+2L_f\varepsilon_{13}\varepsilon_{14})\frac{\kappa_2^2}{\kappa_0^2}.
\end{align*}

To prove Theorem~\ref{zerosg:thm-random-pd-sm}, the following three lemmas are used. 
\begin{lemma}\label{zerosg:lemma:grad-st}
Suppose Assumption~\ref{zerosg:ass:zeroth-smooth} holds. Let $\{\bsx_k\}$ be the sequence generated by Algorithm~\ref{zerosg:algorithm-random-pd}, then
\begin{subequations}
\begin{align}
\bsg^s_k&=\mathbf{E}_{\mathfrak{L}_k}[\bsg^e_k],\label{zerosg:rand-grad-esti1}\\
\|\bsg_k^0-\bsg^s_k\|^2&\le 2L_f^2\|\bsx_{k}\|^2_{\bsK}+2nL_f^2\delta_k^2,\label{zerosg:rand-grad-esti8}\\
\|\bar{\bsg}_k^0-\bar{\bsg}^s_k\|^2
&\le 2L_f^2\|\bsx_{k}\|^2_{\bsK}+2nL_f^2\delta_k^2,\label{zerosg:rand-grad-esti9}\\
\mathbf{E}_{\mathfrak{L}_k}[\|\bar{\bsg}^e_{k}\|^2]
&\le  \frac{1}{n}\mathbf{E}_{\mathfrak{L}_k}[\|\bsg^e_k\|^2]+\|\bar{\bsg}^s_{k}\|^2,\label{zerosg:rand-grad-esti5}\\
\mathbf{E}_{\mathfrak{L}_k}[\|\bsg_k^0-\bsg^e_k\|^2]&\le 4L_f^2\|\bsx_{k}\|^2_{\bsK}+4nL_f^2\delta_k^2
+2\mathbf{E}_{\mathfrak{L}_k}[\|\bsg^e_k\|^2],\label{zerosg:rand-grad-esti6}\\
\|\bsg^0_{k+1}-\bsg^0_{k}\|^2&\le \eta^2_kL_f^2\|\bar{\bsg}^e_{k}\|^2
\le\eta^2_kL_f^2\|\bsg^e_{k}\|^2,\label{zerosg:gg-rand-pd}\\
\|\bar{\bsg}^0_k\|^2&\le 2nL_f(f(\bar{x}_k)-f^*).\label{zerosg:rand-grad-smooth}
\end{align}
\end{subequations}
If Assumptions~\ref{zerosg:ass:zeroth-variance} and \ref{zerosg:ass:fig} also hold, then
\begin{subequations}
\begin{align}
\mathbf{E}_{\mathfrak{L}_k}[\|\bsg^e_k\|^2]
&\le  16p(1+\sigma_0^2)(1+\tilde{\sigma}_0^2)(\|\bar{\bsg}_{k}^0\|^2
+L_f^2\|\bsx_{k}\|^2_{\bsK})+4np\sigma^2_1
+8np(1+\sigma_0^2)\sigma^2_2+\frac{1}{2}np^2L_f^2\delta_k^2,\label{zerosg:rand-grad-esti2}\\
\|\bsg^0_{k+1}\|^2&\le 3(\eta^2_kL_f^2\|\bsg^e_{k}\|^2+n\sigma^2_2
+(1+\tilde{\sigma}_0^2)\|\bar{\bsg}_{k}^0\|^2).\label{zerosg:rand-grad-esti4}
\end{align}
\end{subequations}
\end{lemma}
\begin{proof}
(i) From $u_{i,k}$ and $\xi_{i,k}$ are independent, $x_{i,k}$ is independent of $u_{i,k}$ and $\xi_{i,k}$, and \eqref{zerosg:lemma:uniformsmoothing-equ1}, we have
\begin{align*}
\mathbf{E}_{\mathfrak{L}_k}[g^e_{i,k}]
&=\mathbf{E}_{u_{i,k}}\big[\mathbf{E}_{\xi_{i,k}}\big[\frac{p}{\delta_{i,k}}(F_i(x_{i,k}+\delta_{i,k} u_{i,k},\xi_{i,k})-F_i(x_{i,k},\xi_{i,k}))u_{i,k}\big]\big]\nonumber\\
&=\mathbf{E}_{u_{i,k}}\big[\frac{p}{\delta_{i,k}}(f_i(x_{i,k}+\delta_{i,k} u_{i,k})-f_i(x_{i,k}))u_{i,k}\big]\nonumber\\
&=\mathbf{E}_{u_{i,k}}[\hat{\nabla}_2f_i(x_{i,k},\delta_{i,k},u_{i,k})]=\nabla f^s_i(x_{i,k},\delta_{i,k})=g^s_{i,k},
\end{align*}
which gives \eqref{zerosg:rand-grad-esti1}.

(ii) From Assumption~\ref{zerosg:ass:zeroth-smooth}, we know that each $f_i(x)=\mathbf{E}_{\xi_i}[F_i(x,\xi_i)]$ is smooth with constant $L_{f}$ since
\begin{align}\label{zerosg:ass:fiu:equ}
\|\nabla f_i(x)-\nabla f_i(y)\|&=\|\mathbf{E}_{\xi_i}[\nabla_xF_i(x,\xi_i)-\nabla_xF_i(y,\xi_i)]\|\nonumber\\
&\le\mathbf{E}_{\xi_i}[\|\nabla_xF_i(x,\xi_i)-\nabla_xF_i(y,\xi_i)\|]\nonumber\\
&\le  \mathbf{E}_{\xi_i}[L_f\|x-y\|]= L_f\|x-y\|,~\forall x,y\in\mathbb{R}^p.
\end{align}

From \eqref{zerosg:ass:fiu:equ}, we have
\begin{align}
\|\bsg^0_{k}-\bsg_{k}\|^2
&=\|\nabla\tilde{f}(\bar{\bsx}_{k})-\nabla\tilde{f}(\bsx_k)\|^2\nonumber\\
&=\sum_{i=1}^n\|\nabla f_i(\bar{x}_k)-\nabla f_i(x_{i,k})\|^2
\le\sum_{i=1}^nL_f^2\|\bar{x}_k-x_{i,k}\|^2\nonumber\\
&= L_f^2\|\bar{\bsx}_{k}-\bsx_{k}\|^2
=L_f^2\|\bsx_{k}\|^2_{\bsK}.\label{zerosg:gg1}
\end{align}

From \eqref{zerosg:ass:fiu:equ} and \eqref{zerosg:lemma:uniformsmoothing-equ8}, we have
\begin{align*}
\|g^s_{i,k}-g_{i,k}\|\le L_f\delta_{i,k}.
\end{align*}
Thus,
\begin{align}
\|\bsg^s_k-\bsg_k\|^2=\sum_{i=1}^n\|g^s_{i,k}-g_{i,k}\|^2\le nL_f^2\delta_{k}^2.\label{zerosg:rand-grad-esti3}
\end{align}

Noting $\|\bsg_k^0-\bsg^s_k\|^2\le2\|\bsg_k^0-\bsg_k\|^2+2\|\bsg_k-\bsg^s_k\|^2$, from \eqref{zerosg:gg1} and \eqref{zerosg:rand-grad-esti3}, we know \eqref{zerosg:rand-grad-esti8} holds.

(iii) Noting $\|\bar{\bsg}_k^0-\bar{\bsg}^s_k\|^2=\|\bsH(\bsg_k^0-\bsg^s_k)\|^2$, from $\rho(\bsH)=1$ and \eqref{zerosg:rand-grad-esti8}, we know \eqref{zerosg:rand-grad-esti9} hold.

(iv) We have
\begin{align}\label{zerosg:rand-grad-esti5.1}
\mathbf{E}_{\mathfrak{L}_k}[\|\bar{g}^e_{k}\|^2]
&=\mathbf{E}_{\mathfrak{L}_k}\Big[\Big\|\sum_{i=1}^{n}\frac{1}{n} g^e_{i,k}\Big\|^2\Big]\nonumber\\
&=\frac{1}{n^2}\mathbf{E}_{\mathfrak{L}_k}\Big[\sum_{i=1}^{n}\|g^e_{i,k}\|^2
+\sum_{i=1}^{n}\sum_{j=1,j\neq i}^{n}\langle g^e_{i,k},g^e_{j,k}\rangle\Big]\nonumber\\
&=\frac{1}{n^2}\mathbf{E}_{\mathfrak{L}_k}[\|\bsg^e_{k}\|^2]
+\frac{1}{n^2}\sum_{i=1}^{n}\sum_{j=1,j\neq i}^{n}\langle \mathbf{E}_{\mathfrak{L}_k}[g^e_{i,k}],\mathbf{E}_{\mathfrak{L}_k}[g^e_{j,k}]\rangle\nonumber\\
&=\frac{1}{n^2}\mathbf{E}_{\mathfrak{L}_k}[\|\bsg^e_{k}\|^2]
+\frac{1}{n^2}\sum_{i=1}^{n}\sum_{j=1,j\neq i}^{n}\langle g^s_{i,k},g^s_{j,k}\rangle\nonumber\\
&=\frac{1}{n^2}\mathbf{E}_{\mathfrak{L}_k}[\|\bsg^e_{k}\|^2]
+\|\bar{g}^s_{k}\|^2-\frac{1}{n^2}\|\bsg^s_{k}\|^2,
\end{align}
where the third equality holds since $u_{i,k}$ and $\xi_{i,k},~\forall i\in[n],~k\in\mathbb{N}_+$ are mutually independent; and the fourth equality holds due to \eqref{zerosg:rand-grad-esti1}.

From \eqref{zerosg:rand-grad-esti5.1}, $\mathbf{E}_{\mathfrak{L}_k}[\|\bar{\bsg}^e_{k}\|^2]=n\mathbf{E}_{\mathfrak{L}_k}[\|\bar{g}^e_{k}\|^2]$ and $\|\bar{\bsg}^s_{k}\|^2=n\|\bar{g}^s_{k}\|^2$, we know that \eqref{zerosg:rand-grad-esti5} holds.

(v) We have
\begin{align}\label{zerosg:rand-grad-esti6.1}
\mathbf{E}_{\mathfrak{L}_k}[\|\bsg_k^0-\bsg^e_k\|^2]
&\le2\|\bsg_k^0-\bsg_k^s\|^2
+2\mathbf{E}_{\mathfrak{L}_k}[\|\bsg_k^s-\bsg^e_k\|^2]\nonumber\\
&=2\|\bsg_k^0-\bsg_k^s\|^2
+2\mathbf{E}_{\mathfrak{L}_k}[\|\bsg^e_k\|^2]-2\|\bsg_k^s\|^2,
\end{align}
where the inequality holds due to the Cauchy--Schwarz inequality; and the equality holds since \eqref{zerosg:rand-grad-esti1} and $\bsx_{k}$ is independent of $\mathfrak{L}_k$.

From \eqref{zerosg:rand-grad-esti6.1} and \eqref{zerosg:rand-grad-esti8}, we know \eqref{zerosg:rand-grad-esti6} holds.

(vi) The distributed ZO algorithm \eqref{zerosg:alg:random-pd} can be rewritten as
\begin{subequations}\label{zerosg:alg:random-pd-compact}
\begin{align}
\bm{x}_{k+1}&=\bm{x}_k-\eta_k(\alpha_k\bsL\bm{x}_k+\beta_k\bm{v}_k+\bsg^e_k),\label{zerosg:alg:random-pd-compact-x}\\
\bm{v}_{k+1}&=\bm{v}_k+\eta_k\beta_k\bsL\bm{x}_k,~\forall \bsx_0\in\mathbb{R}^{np},~\sum_{i=1}^{n}v_{i,0}={\bm 0}_p.\label{zerosg:alg:random-pd-compact-v}
\end{align}
\end{subequations}

From \eqref{zerosg:alg:random-pd-compact-v}, we know that
\begin{align}
\bar{v}_{k+1}=\bar{v}_k.\label{zerosg:vbardynamic-rand-pd}
\end{align}
Then, from \eqref{zerosg:vbardynamic-rand-pd}, $\sum_{i=1}^{n}v_{i,0}={\bm 0}_p$, and \eqref{zerosg:alg:random-pd-compact-x}, we know that $\bar{v}_k={\bm 0}_p$ and
\begin{align}
\bar{\bsx}_{k+1}=\bar{\bsx}_{k}-\eta_k\bar{\bsg}^e_k.\label{zerosg:xbardynamic-rand-pd}
\end{align}
Then, we have
\begin{align*}
\|\bsg^0_{k+1}-\bsg^0_{k}\|^2
&=\|\nabla\tilde{f}(\bar{\bsx}_{k+1})-\nabla\tilde{f}(\bar{\bsx}_k)\|^2\nonumber\\
&\le L_f^2\|\bar{\bsx}_{k+1}-\bar{\bsx}_{k}\|^2
=\eta^2_kL_f^2\|\bar{\bsg}^e_{k}\|^2\le\eta^2_kL_f^2\|\bsg^e_{k}\|^2,
\end{align*}
where the first inequality holds due to \eqref{zerosg:ass:fiu:equ}; the last equality holds due to \eqref{zerosg:xbardynamic-rand-pd}; and the last equality holds due to $\bar{\bsg}^e_{k}=\bsH\bsg^e_{k}$ and $\rho(\bsH)=1$. Thus, \eqref{zerosg:gg-rand-pd} holds.

(vii) From \eqref{nonconvex:lemma:lipschitz2}, we have
\begin{align}
\|\bar{\bsg}^0_k\|^2=n\|\nabla f(\bar{x}_k)\|^2\le2nL_f(f(\bar{x}_k)-f^*),
\end{align}
which yields \eqref{zerosg:rand-grad-smooth}.

(viii) From Assumption~\ref{zerosg:ass:zeroth-smooth}, $x_{i,k}$ and $\xi_{i,k}$ are independent of $u_{i,k}$, and \eqref{zerosg:lemma:uniformsmoothing-equ6-1},
we know that for almost every $\xi_{i,k}$ it holds that
\begin{align}\label{zerosg:rand-grad2-esti1.1}
\mathbf{E}_{u_{i,k}}[\|g^e_{i,k}\|^2]
\le2p\|\nabla_{x}F_{i}(x_{i,k},\xi_{i,k})\|^2+\frac{1}{2}p^2L_f^2\delta_{i,k}^2.
\end{align}
Then,
\begin{align}\label{zerosg:rand-grad2-esti1}
\mathbf{E}_{\mathfrak{L}_k}[\|g^e_{i,k}\|^2]
&\le2p\mathbf{E}_{\xi_{i,k}}[\|\nabla_{x}F_{i}(x_{i,k},\xi_{i,k})\|^2]+\frac{1}{2}p^2L_f^2\delta_{i,k}^2\nonumber\\
&=2p\mathbf{E}_{\xi_{i,k}}[\|\nabla_{x}F_{i}(x_{i,k},\xi_{i,k})-\nabla f_i(x_{i,k})+\nabla f_i(x_{i,k})\|^2]+\frac{1}{2}p^2L_f^2\delta_{i,k}^2\nonumber\\
&\le4p\mathbf{E}_{\xi_{i,k}}[\|\nabla_{x}F_{i}(x_{i,k},\xi_{i,k})-\nabla f_i(x_{i,k})\|^2+\|\nabla f_i(x_{i,k})\|^2]+\frac{1}{2}p^2L_f^2\delta_{i,k}^2\nonumber\\
&\le4p(1+\sigma_0^2)\|\nabla f_i(x_{i,k})\|^2+4p\sigma^2_1+\frac{1}{2}p^2L_f^2\delta_{i,k}^2,
\end{align}
where the first inequality holds due to \eqref{zerosg:rand-grad2-esti1.1}; the second inequality holds due to the Cauchy--Schwarz inequality; and the last inequality holds since Assumption~\ref{zerosg:ass:zeroth-variance} and $x_{i,k}$ is independent of $\xi_{i,k}$.

From \eqref{zerosg:ass:fiu:equ}, we have
\begin{align}\label{zerosg:rand-grad2-esti1.3}
\|\nabla f(x)-\nabla f(y)\|^2&=\Big\|\frac{1}{n}\sum_{i=1}^{n}(\nabla f_i(x)-\nabla f_i(y))\Big\|^2\nonumber\\
&\le\frac{1}{n}\sum_{i=1}^{n}\|\nabla f_i(x)-\nabla f_i(y)\|^2\le L_f^2\|x-y\|^2,~\forall x,y\in\mathbb{R}^p.
\end{align}
Then, we have 
\begin{align}\label{zerosg:rand-grad2-esti1.2}
\|\nabla f_i(x_{i,k})\|^2&=\|\nabla f_i(x_{i,k})-\nabla f(x_{i,k})
+\nabla f(x_{i,k})\|^2\nonumber\\
&\le2(\|\nabla f_i(x_{i,k})-\nabla f(x_{i,k})\|^2
+\|\nabla f(x_{i,k})\|^2)\nonumber\\
&\le2(\sigma^2_2
+(1+\tilde{\sigma}_0^2)\|\nabla f(x_{i,k})\|^2)\nonumber\\
&=2(\sigma^2_2
+(1+\tilde{\sigma}_0^2)\|\nabla f(x_{i,k})-\nabla f(\bar{x}_k)+\nabla f(\bar{x}_k)\|^2)\nonumber\\
&\le 2\sigma^2_2
+4(1+\tilde{\sigma}_0^2)(\|\nabla f(x_{i,k})-\nabla f(\bar{x}_k)\|^2+\|\nabla f(\bar{x}_k)\|^2)\nonumber\\
&\le 2\sigma^2_2
+4(1+\tilde{\sigma}_0^2)(L_f^2\|x_{i,k}-\bar{x}_k\|^2+\|\nabla f(\bar{x}_k)\|^2),
\end{align}
where the first and third inequalities hold due to the Cauchy--Schwarz inequality; the second inequality holds due to Assumption~\ref{zerosg:ass:fig}; and the last inequality holds due to \eqref{zerosg:rand-grad2-esti1.3}.

From \eqref{zerosg:rand-grad2-esti1} and \eqref{zerosg:rand-grad2-esti1.2}, we know \eqref{zerosg:rand-grad-esti2} holds.

(ix) From  the Cauchy--Schwarz inequality, we have
\begin{align}
\|\bsg^0_{k+1}\|^2
=\|\bsg^0_{k+1}-\bsg^0_{k}+\bsg^0_{k}-\bar{\bsg}^0_k+\bar{\bsg}^0_k\|^2
\le3(\|\bsg^0_{k+1}-\bsg^0_{k}\|^2+\|\bsg^0_{k}-\bar{\bsg}^0_k\|^2+\|\bar{\bsg}^0_k\|^2).
\label{zerosg:rand-grad-esti4.1}
\end{align}
From Assumption~\ref{zerosg:ass:fig}, we have
\begin{align}\label{zerosg:rand-grad-esti4.2}
\|\bsg^0_{k}-\bar{\bsg}^0_k\|^2=\sum_{i=1}^{n}\|f_i(\bar{x}_k)-f(\bar{x}_k)\|^2\le n\sigma^2_2+n\tilde{\sigma}^2_0\|f(\bar{x}_k)\|^2.
\end{align}

From \eqref{zerosg:rand-grad-esti4.1}, \eqref{zerosg:rand-grad-esti4.2}, and \eqref{zerosg:gg-rand-pd}, we know  \eqref{zerosg:rand-grad-esti4} holds.
\end{proof}

\begin{lemma}\label{zerosg:lemma:sg}
Suppose Assumptions~\ref{zerosg:ass:graph}--\ref{zerosg:ass:fig} hold. Suppose $\{\beta_k\}$ is non-decreasing, $\alpha_k=\kappa_1\beta_k$, and $\eta_k=\frac{\kappa_2}{\beta_k}$, where $\kappa_1>1$ and $\kappa_2>0$ are constants. Moreover, suppose $\beta_k\ge\varepsilon_{1}$.  Let $\{\bsx_k\}$ be the sequence generated by Algorithm~\ref{zerosg:algorithm-random-pd}, then
\begin{subequations}
\begin{align}
\mathbf{E}_{\mathfrak{L}_k}[W_{k+1}]
&\le  W_{k}-\|\bsx_k\|^2_{(2\varepsilon_4-\varepsilon_5\omega_k-b_{1,k})\bsK}
-\Big\|\bm{v}_k+\frac{1}{\beta_k}\bsg_{k}^0\Big\|^2_{b_{2,k}\bsK}\nonumber\\
&\quad-\eta_k\Big(\frac{1}{4}-(1+\tilde{\sigma}_0^2)(b_{3,k}+8p(1+\sigma_0^2)b_{4,k})\eta_k\Big)\|\bar{\bsg}^0_{k}\|^2
+2pn\sigma^2_1b_{4,k}\eta_k^2\nonumber\\
&\quad+n\sigma^2_2(b_{3,k}+4p(1+\sigma_0^2)b_{4,k})\eta_k^2+b_{5,k}\eta_k\delta_k^2,
\label{zerosg:sgproof-vkLya}\\
\mathbf{E}_{\mathfrak{L}_k}[\breve{W}_{k+1}]
&\le  \breve{W}_{k}-\|\bsx_k\|^2_{(2\varepsilon_4-\varepsilon_5\omega_k-b_{1,k})\bsK}
-\Big\|\bm{v}_k+\frac{1}{\beta_k}\bsg_{k}^0\Big\|^2_{b_{2,k}\bsK}\nonumber\\
&\quad+(1+\tilde{\sigma}_0^2)(b_{3,k}+8p(1+\sigma_0^2)b_{4,k})\eta_k^2\|\bar{\bsg}^0_{k}\|^2
+2pn\sigma^2_1b_{4,k}\eta_k^2\nonumber\\
&\quad+n\sigma^2_2(b_{3,k}+4p(1+\sigma_0^2)b_{4,k})\eta_k^2+b_{5,k}\eta_k\delta_k^2,
\label{zerosg:sgproof-vkLya-bounded}
\end{align}
\end{subequations}
where $W_{k}=\sum_{i=1}^{4}W_{i,k}$, $\breve{W}_k=\sum_{i=1}^{3}W_{i,k}$, $\omega_k=\frac{1}{\beta_{k}}-\frac{1}{\beta_{k+1}}$ and
\begin{align*}
W_{1,k}&=\frac{1}{2}\|\bm{x}_k \|^2_{\bsK},\\
W_{2,k}&=\frac{1}{2}\Big\|\bsv_k+\frac{1}{\beta_k}\bsg_k^0\Big\|^2_{\bsQ+\kappa_1\bsK},\\
W_{3,k}&=\bsx_k^\top\bsK\Big(\bm{v}_k+\frac{1}{\beta_k}\bsg_k^0\Big),\\
W_{4,k}&=n(f(\bar{x}_k)-f^*)=\tilde{f}(\bar{\bsx}_k)-f^*,\\
b_{1,k}
&=8p(1+\sigma_0^2)(1+\tilde{\sigma}_0^2)\kappa_3L_f^4\frac{\eta_k}{\beta_k^2}
+16p(1+\sigma_0^2)(1+\tilde{\sigma}_0^2)\kappa_5L_f^4\frac{\eta_k^2}{\beta_k^2}\\
&\quad+\Big(\frac{1}{2}+L_f^2\Big)\eta_k\omega_k
+8p(1+\sigma_0^2)(1+\tilde{\sigma}_0^2)\kappa_4L_f^4
\frac{\eta_k\omega_k}{\beta_k^2}\\
&\quad+(5+32p(1+\sigma_0^2)(1+\tilde{\sigma}_0^2)
+24p(1+\sigma_0^2)(1+\tilde{\sigma}_0^2)\kappa_3L_f^2)L_f^2\eta_k^2\omega_k\\
&\quad+24p(1+\sigma_0^2)(1+\tilde{\sigma}_0^2)\kappa_4L_f^4\eta_k^2\omega_k^2
+16p(1+\sigma_0^2)(1+\tilde{\sigma}_0^2)\kappa_4L_f^4\frac{\eta_k^2\omega_k}{\beta_k^2},\\
b_{2,k}&=2\varepsilon_6
-\frac{1}{2}\omega_k(\kappa_1\kappa_2+3\kappa_2^2+\kappa_4)
-\frac{1}{2}\omega_k\eta_k\kappa_4,\\
b_{3,k}&=\frac{3}{2}\kappa_3\frac{\omega_k}{\eta_k^2}
+\frac{3}{2}\kappa_4\frac{\omega_k^2}{\eta_k^2},\\
b_{4,k}&=6+L_f+\frac{\kappa_3L_f^2}{\kappa_2}\frac{1}{\beta_k}+2\kappa_5L_f^2\frac{1}{\beta_k^2}
+\frac{\kappa_4L_f^2}{\kappa_2}\frac{\omega_k}{\beta_k}\\
&\quad+(4+3\kappa_3L_f^2)\omega_k+3\kappa_4L_f^2\omega_k^2+2\kappa_4L_f^2\frac{\omega_k}{\beta_k^2},\\
b_{5,k}
&=nL_f^2\Big(\frac{1}{4}p^2b_{4,k}\eta_k+3+\omega_k+8\eta_k+5\eta_k\omega_k\Big).
\end{align*}
\end{lemma}
\begin{proof}
Note that $W_{4,k}$ is well defined due to $f^*>-\infty$ as assumed in Assumption~\ref{zerosg:ass:optset}. Thus, $W_{k}$ is well defined.

(i) We have
\begin{align}
\mathbf{E}_{\mathfrak{L}_k}[W_{1,k+1}]
&=\mathbf{E}_{\mathfrak{L}_k}\Big[\frac{1}{2}\|\bm{x}_{k+1} \|^2_{\bsK}\Big]\nonumber\\
&=\mathbf{E}_{\mathfrak{L}_k}\Big[\frac{1}{2}\|\bm{x}_k-\eta_k(\alpha_k\bsL\bm{x}_k+\beta_k\bm{v}_k+\bsg^e_k) \|^2_{\bsK}\Big]\nonumber\\
&=\mathbf{E}_{\mathfrak{L}_k}\Big[\frac{1}{2}\|\bm{x}_k\|^2_{\bsK}-\eta_k\alpha_k\|\bsx_k\|^2_{\bsL}
+\frac{1}{2}\eta^2_k\alpha^2_k\|\bsx_k\|^2_{\bsL^2}
\nonumber\\
&\quad-\eta_k\beta_k\bsx^\top_k({\bm I}_{np}-\eta_k\alpha_k\bsL)\bsK\Big(\bm{v}_k+\frac{1}{\beta_k}\bsg^e_k\Big)
+\frac{1}{2}\eta^2_k\beta^2_k\Big\|\bm{v}_k+\frac{1}{\beta_k}\bsg^e_k\Big\|^2_{\bsK}\Big]\nonumber\\
&=W_{1,k}-\|\bsx_k\|^2_{\eta_k\alpha_k\bsL
-\frac{1}{2}\eta^2_k\alpha^2_k\bsL^2}
-\eta_k\beta_k\bsx^\top_k({\bm I}_{np}-\eta_k\alpha_k\bsL)\bsK\Big(\bm{v}_k
+\frac{1}{\beta_k}\bsg^s_k\Big)\nonumber\\
&\quad+\frac{1}{2}\eta^2_k\beta^2_k\mathbf{E}_{\mathfrak{L}_k}\Big[\Big\|\bm{v}_k+\frac{1}{\beta_k}\bsg_k^0
+\frac{1}{\beta_k}\bsg^e_k-\frac{1}{\beta_k}\bsg_k^0\Big\|^2_{\bsK}\Big]\nonumber\\
&=W_{1,k}-\|\bsx_k\|^2_{\eta_k\alpha_k\bsL
-\frac{1}{2}\eta^2_k\alpha^2_k\bsL^2}-\eta_k\beta_k\bsx^\top_k({\bm I}_{np}\nonumber\\
&\quad-\eta_k\alpha_k\bsL)\bsK\Big(\bm{v}_k
+\frac{1}{\beta_k}\bsg_k^0
+\frac{1}{\beta_k}\bsg^s_k-\frac{1}{\beta_k}\bsg_k^0\Big)\nonumber\\
&\quad+\frac{1}{2}\eta^2_k\beta^2_k\mathbf{E}_{\mathfrak{L}_k}\Big[\Big\|\bm{v}_k+\frac{1}{\beta_k}\bsg_k^0
+\frac{1}{\beta_k}\bsg^e_k-\frac{1}{\beta_k}\bsg_k^0\Big\|^2_{\bsK}\Big]\nonumber\\
&\le W_{1,k}-\|\bsx_k\|^2_{\eta_k\alpha_k\bsL
-\frac{1}{2}\eta^2_k\alpha^2_k\bsL^2}-\eta_k\beta_k\bsx^\top_k\bsK\Big(\bm{v}_k+\frac{1}{\beta_k}\bsg_k^0\Big)\nonumber\\
&~~~+\frac{1}{2}\eta_k\|\bm{x}_k\|^2_{\bsK}
+\frac{1}{2}\eta_k\|\bsg^s_k-\bsg_k^0\|^2\nonumber\\
&~~~+\frac{1}{2}\eta^2_k\alpha^2_k\|\bm{x}_k\|^2_{\bsL^2}
+\frac{1}{2}\eta^2_k\beta^2_k\Big\|\bm{v}_k+\frac{1}{\beta_k}\bsg_k^0\Big\|^2_{\bsK}\nonumber\\
&~~~+\frac{1}{2}\eta^2_k\alpha^2_k\|\bm{x}_k\|^2_{\bsL^2}
+\frac{1}{2}\eta^2_k\|\bsg^s_k-\bsg_k^0\|^2\nonumber\\
&~~~+\eta^2_k\beta^2_k\Big\|\bm{v}_k+\frac{1}{\beta_k}\bsg_k^0\Big\|^2_{\bsK}
+\eta^2_k\mathbf{E}_{\mathfrak{L}_k}[\|\bsg^e_k-\bsg_k^0\|^2]\nonumber\\
&= W_{1,k}-\|\bsx_k\|^2_{\eta_k\alpha_k\bsL-\frac{1}{2}\eta_k\bsK
-\frac{3}{2}\eta^2_k\alpha^2_k\bsL^2}\nonumber\\
&~~~+\frac{1}{2}\eta_k(1+\eta_k)\|\bsg^s_k-\bsg_k^0\|^2
+\eta^2_k\mathbf{E}_{\mathfrak{L}_k}[\|\bsg^e_k-\bsg_k^0\|^2]\nonumber\\
&~~~-\eta_k\beta_k\bsx^\top_k\bsK\Big(\bm{v}_k+\frac{1}{\beta_k}\bsg_k^0\Big)
+\Big\|\bm{v}_k+\frac{1}{\beta_k}\bsg_k^0\Big\|^2_{\frac{3}{2}\eta^2_k\beta^2_k\bsK}\nonumber\\
&\le W_{1,k}-\|\bsx_k\|^2_{\eta_k\alpha_k\bsL-\frac{1}{2}\eta_k\bsK
-\frac{3}{2}\eta^2_k\alpha^2_k\bsL^2-\eta_k(1+5\eta_k)L_f^2\bsK}
\nonumber\\
&~~~-\eta_k\beta_k\bsx^\top_k\bsK\Big(\bm{v}_k+\frac{1}{\beta_k}\bsg_k^0\Big)
+\Big\|\bm{v}_k+\frac{1}{\beta_k}\bsg_k^0\Big\|^2_{\frac{3}{2}\eta^2_k\beta^2_k\bsK}\nonumber\\
&~~~+nL_f^2\eta_k(1+5\eta_k)\delta^2_k+2\eta^2_k\mathbf{E}_{\mathfrak{L}_k}[\|\bsg^e_k\|^2],\label{zerosg:v1k}
\end{align}
where the second equality holds due to \eqref{zerosg:alg:random-pd-compact-x}; the third equality holds due to \eqref{nonconvex:KL-L-eq} in Lemma~\ref{nonconvex:lemma-Xinlei}; the fourth equality holds since \eqref{zerosg:rand-grad-esti1} and that $x_{i,k}$ and $v_{i,k}$ are independent of $\mathfrak{L}_k$; the first inequality holds due to the Cauchy--Schwarz inequality and $\rho(\bsK)=1$; and the last  inequality holds due to \eqref{zerosg:rand-grad-esti8} and \eqref{zerosg:rand-grad-esti6}.

(ii) We know that $\omega_k\ge0$ since $\{\beta_k\}$ is non-decreasing. We have
\begin{align}
W_{2,k+1}
&=\frac{1}{2}\Big\|\bsv_{k+1}+\frac{1}{\beta_{k+1}}\bsg_{k+1}^0\Big\|^2_{\bsQ+\kappa_1\bsK}\nonumber\\
&=\frac{1}{2}\Big\|\bsv_{k+1}+\frac{1}{\beta_{k}}\bsg_{k+1}^0
+\Big(\frac{1}{\beta_{k+1}}-\frac{1}{\beta_{k}}\Big)\bsg_{k+1}^0\Big\|^2_{\bsQ+\kappa_1\bsK}\nonumber\\
&\le\frac{1}{2}(1+\omega_k)\Big\|\bsv_{k+1}+\frac{1}{\beta_{k}}\bsg_{k+1}^0\Big\|^2_{\bsQ+\kappa_1\bsK}
+\frac{1}{2}(\omega_k+\omega_k^2)\|\bsg_{k+1}^0\|^2_{\bsQ+\kappa_1\bsK},\label{zerosg:v2k-1}
\end{align}
where the inequality holds due to the Cauchy--Schwarz inequality.

For the first term on the right-hand side of \eqref{zerosg:v2k-1}, we have
\begin{align}
\frac{1}{2}\Big\|\bsv_{k+1}+\frac{1}{\beta_{k}}\bsg_{k+1}^0\Big\|^2_{\bsQ+\kappa_1\bsK}
&=\frac{1}{2}\Big\|\bm{v}_k+\frac{1}{\beta_k}\bsg_{k}^0+\eta_k\beta_k\bsL\bm{x}_k
+\frac{1}{\beta_k}(\bsg_{k+1}^0-\bsg_{k}^0) \Big\|^2_{\bsQ+\kappa_1\bsK}\nonumber\\
&= W_{2,k}+\eta_k\beta_k\bsx^\top_k(\bsK+\kappa_1\bsL)\Big(\bm{v}_k+\frac{1}{\beta_k}\bsg_k^0\Big)\nonumber\\
&\quad+\|\bsx_k\|^2_{\frac{1}{2}\eta_k^2\beta_k^2(\bsL+\kappa_1\bsL^2)}
+\frac{1}{2\beta^2_k}\Big\|\bsg_{k+1}^0-\bsg_{k}^0\Big\|^2_{\bsQ+\kappa_1\bsK}\nonumber\\
&\quad+\frac{1}{\beta_k}\Big(\bm{v}_k+\frac{1}{\beta_k}\bsg_{k}^0
+\eta_k\beta_k\bsL\bm{x}_k\Big)^\top(\bsQ
+\kappa_1\bsK)(\bsg_{k+1}^0-\bsg_{k}^0)\nonumber\\
&\le W_{2,k}+\eta_k\beta_k\bsx^\top_k(\bsK+\kappa_1\bsL)\Big(\bm{v}_k+\frac{1}{\beta_k}\bsg_k^0\Big)\nonumber\\
&\quad+\|\bsx_k\|^2_{\frac{1}{2}\eta_k^2\beta_k^2(\bsL+\kappa_1\bsL^2)}
+\frac{1}{2\beta^2_k}\|\bsg_{k+1}^0-\bsg_{k}^0\|^2_{\bsQ+\kappa_1\bsK}\nonumber\\
&\quad+\frac{\eta_k}{2}\Big\|\bm{v}_k+\frac{1}{\beta_k}\bsg_{k}^0\Big\|^2_{\bsQ+\kappa_1\bsK}
+\frac{1}{2\eta_k\beta_k^2}\|\bsg_{k+1}^0-\bsg_{k}^0\|^2_{\bsQ+\kappa_1\bsK}\nonumber\\
&\quad+\frac{1}{2}\eta^2_k\beta^2_k\|\bsL\bm{x}_k\|^2_{\bsQ+\kappa_1\bsK}
+\frac{1}{2\beta^2_k}\|\bsg_{k+1}^0-\bsg_{k}^0\|^2_{\bsQ+\kappa_1\bsK}\nonumber\\
&= W_{2,k}+\eta_k\beta_k\bsx^\top_k(\bsK+\kappa_1\bsL)\Big(\bm{v}_k+\frac{1}{\beta_k}\bsg_k^0\Big)\nonumber\\
&\quad+\|\bsx_k\|^2_{\eta^2_k\beta^2_k(\bsL+\kappa_1\bsL^2)}
+\Big\|\bm{v}_k+\frac{1}{\beta_k}\bsg_{k}^0
\Big\|^2_{\frac{1}{2}\eta_k(\bsQ+\kappa_1\bsK)}\nonumber\\
&\quad+\frac{1}{\beta^2_k}\Big(1+\frac{1}{2\eta_k}\Big)
\|\bsg_{k+1}^0-\bsg_{k}^0\|^2_{\bsQ+\kappa_1\bsK}\nonumber\\
&\le W_{2,k}+\eta_k\beta_k\bsx^\top_k(\bsK+\kappa_1\bsL)\Big(\bm{v}_k+\frac{1}{\beta_k}\bsg_k^0\Big)\nonumber\\
&\quad+\|\bsx_k\|^2_{\eta^2_k\beta^2_k(\bsL+\kappa_1\bsL^2)}
+\Big\|\bm{v}_k+\frac{1}{\beta_k}\bsg_{k}^0
\Big\|^2_{\frac{1}{2}\eta_k(\bsQ+\kappa_1\bsK)}\nonumber\\
&\quad+\frac{1}{\beta^2_k}\Big(1+\frac{1}{2\eta_k}\Big)\kappa_4
\|\bsg_{k+1}^0-\bsg_{k}^0\|^2\nonumber\\
&\le W_{2,k}+\eta_k\beta_k\bsx^\top_k(\bsK+\kappa_1\bsL)\Big(\bm{v}_k+\frac{1}{\beta_k}\bsg_k^0\Big)\nonumber\\
&\quad+\|\bsx_k\|^2_{\eta^2_k\beta^2_k(\bsL+\kappa_1\bsL^2)}
+\Big\|\bm{v}_k+\frac{1}{\beta_k}\bsg_{k}^0
\Big\|^2_{\frac{1}{2}\eta_k(\bsQ+\kappa_1\bsK)}\nonumber\\
&\quad+\frac{\eta_k}{\beta^2_k}\Big(\eta_k+\frac{1}{2}\Big)
\kappa_4L_f^2\|\bar{\bsg}^e_{k}\|^2,\label{zerosg:v2k-2}
\end{align}
where the first equality holds due to \eqref{zerosg:alg:random-pd-compact-v}; the second equality holds due to \eqref{nonconvex:KL-L-eq} and \eqref{nonconvex:lemma-eq} in Lemma~\ref{nonconvex:lemma-Xinlei};  the first inequality holds due to the Cauchy--Schwarz inequality; the last equality holds due to \eqref{nonconvex:KL-L-eq} and \eqref{nonconvex:lemma-eq}; the second inequality holds due to $\rho(\bsQ+\kappa_1\bsK)\le\rho(\bsQ)+\kappa_1\rho(\bsK)$, \eqref{nonconvex:lemma-eq2}, $\rho(\bsK)=1$; and the last  inequality holds due to \eqref{zerosg:gg-rand-pd}.

For the second term on the right-hand side of \eqref{zerosg:v2k-1}, we have
\begin{align}\label{zerosg:v2k-3}
\|\bsg_{k+1}^0\|^2_{\bsQ+\kappa_1\bsK}
\le\kappa_4\|\bsg_{k+1}^0\|^2.
\end{align}

Also note that
\begin{align}\label{zerosg:v2k-4}
\Big\|\bm{v}_k+\frac{1}{\beta_k}\bsg_{k}^0\Big\|^2_{\bsQ+\kappa_1\bsK}
\le\kappa_4\Big\|\bm{v}_k+\frac{1}{\beta_k}\bsg_{k}^0
\Big\|^2_{\bsK}.
\end{align}

Then, from \eqref{zerosg:v2k-1}--\eqref{zerosg:v2k-4}, we have
\begin{align}
W_{2,k+1}
&\le W_{2,k}
+(1+\omega_k)\eta_k\beta_k\bsx^\top_k(\bsK+\kappa_1\bsL)\Big(\bm{v}_k+\frac{1}{\beta_k}\bsg_k^0\Big)\nonumber\\
&\quad+\frac{1}{2}(\eta_k+\omega_k+\eta_k\omega_k)\kappa_4
\Big\|\bm{v}_k+\frac{1}{\beta_k}\bsg_{k}^0\Big\|^2_{\bsK}\nonumber\\
&\quad+\|\bsx_k\|^2_{(1+\omega_k)\eta^2_k\beta^2_k(\bsL+\kappa_1\bsL^2)}
+\frac{\eta_k}{\beta^2_k}\Big(\eta_k+\frac{1}{2}\Big)(1+\omega_k)
\kappa_4L_f^2\|\bar{\bsg}^e_{k}\|^2\nonumber\\
&\quad+\frac{1}{2}\kappa_4(\omega_k+\omega_k^2)\|\bsg_{k+1}^0\|^2.
\label{zerosg:v2k}
\end{align}

(iii) We have
\begin{align}
W_{3,k+1}&=\bsx_{k+1}^\top\bsK\Big(\bm{v}_{k+1}+\frac{1}{\beta_{k+1}}\bsg_{k+1}^0\Big)\nonumber\\
&=\bsx_{k+1}^\top\bsK\Big(\bm{v}_{k+1}+\frac{1}{\beta_{k}}\bsg_{k+1}^0
+\Big(\frac{1}{\beta_{k+1}}-\frac{1}{\beta_{k}}\Big)\bsg_{k+1}^0\Big)\nonumber\\
&=\bsx_{k+1}^\top\bsK\Big(\bm{v}_{k+1}+\frac{1}{\beta_{k}}\bsg_{k+1}^0\Big)
-\omega_k\bsx_{k+1}^\top\bsK\bsg_{k+1}^0\nonumber\\
&\le\bsx_{k+1}^\top\bsK\Big(\bm{v}_{k+1}+\frac{1}{\beta_{k}}\bsg_{k+1}^0\Big)
+\frac{1}{2}\omega_k(\|\bsx_{k+1}\|^2_{\bsK}+\|\bsg_{k+1}^0\|^2).\label{zerosg:v3k-1}
\end{align}
For the first term on the right-hand side of \eqref{zerosg:v3k-1}, we have
\begin{align}
\mathbf{E}_{\mathfrak{L}_k}\Big[\bsx_{k+1}^\top\bsK\Big(\bm{v}_{k+1}
+\frac{1}{\beta_k}\bsg_{k+1}^0\Big)\Big]
&=\mathbf{E}_{\mathfrak{L}_k}\Big[(\bm{x}_k-\eta_k(\alpha_k\bsL\bm{x}_k+\beta_k\bm{v}_k
+\bsg_k^0+\bsg^e_k-\bsg_k^0))^\top
\nonumber\\
&~~~\times\bsK\Big(\bm{v}_k+\frac{1}{\beta_k}\bsg_{k}^0+\eta_k\beta_k\bsL\bm{x}_k
+\frac{1}{\beta_k}(\bsg_{k+1}^0-\bsg_{k}^0)\Big)\Big]\nonumber\\
&=\bm{x}_k^\top(\bsK-\eta_k(\alpha_k+\eta_k\beta^2_k)\bsL)\Big(\bm{v}_k+\frac{1}{\beta_k}\bsg_{k}^0\Big)
\nonumber\\
&~~~+\|\bm{x}_k\|^2_{\eta_k\beta_k(\bsL-\eta_k\alpha_k\bsL^2)}\nonumber\\
&~~~+\frac{1}{\beta_k}\bm{x}_k^\top(\bsK-\eta_k\alpha_k\bsL)
\mathbf{E}_{\mathfrak{L}_k}[\bsg_{k+1}^0-\bsg_{k}^0]\nonumber\\
&~~~-\eta_k\beta_k\Big\|\bm{v}_k+\frac{1}{\beta_k}\bsg_{k}^0\Big\|^2_{\bsK}\nonumber\\
&~~~-\eta_k\Big(\bm{v}_k+\frac{1}{\beta_k}\bsg_{k}^0\Big)^\top\bsK
\mathbf{E}_{\mathfrak{L}_k}[\bsg_{k+1}^0-\bsg_{k}^0]\nonumber\\
&~~~-\eta_k(\bsg_k^s-\bsg_k^0)^\top
\bsK\Big(\bm{v}_k+\frac{1}{\beta_k}\bsg_{k}^0+\eta_k\beta_k\bsL\bm{x}_k\Big)\nonumber\\
&~~~-\frac{1}{\beta_k}\mathbf{E}_{\mathfrak{L}_k}[\eta_k(\bsg^e_k-\bsg_k^0)^\top
\bsK(\bsg_{k+1}^0-\bsg_{k}^0)]\nonumber\\
&\le\bm{x}_k^\top(\bsK-\eta_k\alpha_k\bsL)\Big(\bm{v}_k+\frac{1}{\beta_k}\bsg_{k}^0\Big)
+\frac{1}{2}\eta^2_k\beta^2_k\|\bsL\bsx_k\|^2\nonumber\\
&~~~+\frac{1}{2}\eta^2_k\beta^2_k\Big\|\bm{v}_k+\frac{1}{\beta_k}\bsg_{k}^0\Big\|^2_{\bsK}
+\|\bm{x}_k\|^2_{\eta_k\beta_k(\bsL-\eta_k\alpha_k\bsL^2)}\nonumber\\
&~~~+\frac{1}{2}\eta_k\|\bm{x}_k\|^2_\bsK
+\frac{1}{2\eta_k\beta^2_k}\mathbf{E}_{\mathfrak{L}_k}[\|\bsg_{k+1}^0-\bsg_{k}^0\|^2\nonumber\\
&~~~+\frac{1}{2}\eta^2_k\alpha^2_k\|\bsL\bm{x}_k\|^2
+\frac{1}{2\beta^2_k}\mathbf{E}_{\mathfrak{L}_k}[\|\bsg_{k+1}^0-\bsg_{k}^0\|^2]\nonumber\\
&~~~
-\eta_k\beta_k\Big\|\bm{v}_k+\frac{1}{\beta_k}\bsg_{k}^0\Big\|^2_{\bsK}\nonumber\\
&~~~+\frac{1}{2}\eta^2_k\beta_k^2\Big\|\bm{v}_k+\frac{1}{\beta_k}\bsg_{k}^0\Big\|^2_{\bsK}
+\frac{1}{2\beta_k^2}\mathbf{E}_{\mathfrak{L}_k}[\|\bsg_{k+1}^0-\bsg_{k}^0\|^2]\nonumber\\
&~~~+\frac{1}{2}\eta_k\|\bsg_k^s-\bsg_k^0\|^2
+\frac{1}{2}\eta_k\Big\|\bm{v}_k+\frac{1}{\beta_k}\bsg_{k}^0\Big\|^2_{\bsK}\nonumber\\
&~~~
+\frac{1}{2}\eta^2_k\|\bsg_k^s-\bsg_k^0\|^2
+\frac{1}{2}\eta^2_k\beta^2_k\|\bsL\bm{x}_k\|^2\nonumber\\
&~~~
+\frac{1}{2}\eta^2_k\mathbf{E}_{\mathfrak{L}_k}[\|\bsg^e_k-\bsg_k^0\|^2]
+\frac{1}{2\beta^2_k}\mathbf{E}_{\mathfrak{L}_k}[\|\bsg_{k+1}^0-\bsg_{k}^0\|^2]\nonumber\\
&=\bm{x}_k^\top(\bsK-\eta_k\alpha_k\bsL)\Big(\bm{v}_k+\frac{1}{\beta_k}\bsg_{k}^0\Big)\nonumber\\
&~~~+\frac{1}{2}(\eta_k+\eta^2_k)\|\bsg_k^s-\bsg_k^0\|^2
+\frac{1}{2}\eta^2_k\mathbf{E}_{\mathfrak{L}_k}[\|\bsg^e_k-\bsg_k^0\|^2]\nonumber\\
&~~~+\|\bm{x}_k\|^2_{\eta_k(\beta_k\bsL+\frac{1}{2}\bsK)
+\eta^2_k(\frac{1}{2}\alpha^2_k-\alpha_k\beta_k+\beta^2_k)\bsL^2}\nonumber\\
&~~~+\Big(\frac{1}{2\eta_k\beta^2_k}+\frac{3}{2\beta^2_k}\Big)
\mathbf{E}_{\mathfrak{L}_k}[\|\bsg_{k+1}^0-\bsg_{k}^0\|^2]\nonumber\\
&~~~-\Big\|\bm{v}_k+\frac{1}{\beta_k}\bsg_{k}^0\Big\|^2_{\eta_k(\beta_k-\frac{1}{2}
-\eta_k\beta^2_k)\bsK}\nonumber\\
&\le\bm{x}_k^\top\bsK\Big(\bm{v}_k+\frac{1}{\beta_k}\bsg_{k}^0\Big)
-(1+\omega_k)\eta_k\alpha_k\bm{x}_k^\top\bsL\Big(\bm{v}_k+\frac{1}{\beta_k}\bsg_{k}^0\Big)\nonumber\\
&~~~+\omega_k\eta_k\alpha_k\bm{x}_k^\top\bsL\Big(\bm{v}_k+\frac{1}{\beta_k}\bsg_{k}^0\Big)\nonumber\\
&~~~+\|\bm{x}_k\|^2_{\eta_k(\beta_k\bsL+\frac{1}{2}\bsK)
+\eta^2_k(\frac{1}{2}\alpha^2_k-\alpha_k\beta_k+\beta^2_k)\bsL^2
+\eta_k(1+3\eta_k)L_f^2\bsK}\nonumber\\
&~~~+\frac{\eta_k}{2\beta^2_k}(1+3\eta_k)L_f^2\mathbf{E}_{\mathfrak{L}_k}[\|\bar{\bsg}^e_{k}\|^2]
+nL_f^2\eta_k(1+3\eta_k)\delta^2_k\nonumber\\
&~~~+\eta^2_k\mathbf{E}_{\mathfrak{L}_k}[\|\bsg^e_k\|^2]-\Big\|\bm{v}_k+\frac{1}{\beta_k}\bsg_{k}^0\Big\|^2_{\eta_k(\beta_k-\frac{1}{2}
-\eta_k\beta^2_k)\bsK},
\label{zerosg:v3k-2}
\end{align}
where the first equality holds due to \eqref{zerosg:alg:random-pd-compact}; the second equality holds since \eqref{nonconvex:KL-L-eq}, \eqref{zerosg:rand-grad-esti1}, and that $x_{i,k}$ and $v_{i,k}$ are independent of $\mathfrak{L}_k$; the first inequality holds due to the Cauchy--Schwarz inequality, the Jensen's inequality, \eqref{nonconvex:KL-L-eq}, and $\rho(\bsK)=1$; and the last  inequality holds due to \eqref{zerosg:rand-grad-esti8}, \eqref{zerosg:rand-grad-esti6}, and \eqref{zerosg:gg-rand-pd}.

For the third term on the right-hand side of \eqref{zerosg:v3k-2}, we have
\begin{align}\label{zerosg:v3k-3}
\omega_k\eta_k\alpha_k\bm{x}_k^\top\bsL\Big(\bm{v}_k+\frac{1}{\beta_k}\bsg_{k}^0\Big)
&=\omega_k\eta_k\alpha_k\bsx^\top_k\bsL\bsK\Big(\bm{v}_k+\frac{1}{\beta_k}\bsg_k^0\Big)\nonumber\\
&\le\|\bsx_k\|^2_{\frac{1}{2}\omega_k\eta_k\alpha_k\bsL^2}
+\Big\|\bm{v}_k+\frac{1}{\beta_k}\bsg_k^0\Big\|^2_{\frac{1}{2}\omega_k\eta_k\alpha_k\bsK}.
\end{align}

Then, from \eqref{zerosg:v3k-1}--\eqref{zerosg:v3k-3}, we have
\begin{align}
\mathbf{E}_{\mathfrak{L}_k}[W_{3,k+1}]
&\le W_{3,k}
-(1+\omega_k)\eta_k\alpha_k\bm{x}_k^\top\bsL\Big(\bm{v}_k+\frac{1}{\beta_k}\bsg_{k}^0\Big)\nonumber\\
&\quad+\|\bm{x}_k\|^2_{\eta_k(\beta_k\bsL+\frac{1}{2}\bsK)
+\eta^2_k(\frac{1}{2}\alpha^2_k-\alpha_k\beta_k+\beta^2_k)\bsL^2}
+\|\bsx_k\|^2_{\frac{1}{2}\omega_k\eta_k\alpha_k\bsL^2+\eta_k(1+3\eta_k)L_f^2\bsK}\nonumber\\
&\quad+\frac{\eta_k}{2\beta^2_k}(1+3\eta_k)L_f^2\mathbf{E}_{\mathfrak{L}_k}[\|\bar{\bsg}^e_{k}\|^2]
+nL_f^2\eta_k(1+3\eta_k)\delta^2_k+\eta^2_k\mathbf{E}_{\mathfrak{L}_k}[\|\bsg^e_k\|^2]\nonumber\\
&\quad-\Big\|\bm{v}_k+\frac{1}{\beta_k}\bsg_{k}^0\Big\|^2_{\eta_k(\beta_k-\frac{1}{2}
-\eta_k\beta^2_k-\frac{1}{2}\omega_k\alpha_k)\bsK}
+\frac{1}{2}\omega_k\mathbf{E}_{\mathfrak{L}_k}[2W_{1,k+1}+\|\bsg_{k+1}^0\|^2].
\label{zerosg:v3k}
\end{align}

(iv) We have
\begin{align}
\mathbf{E}_{\mathfrak{L}_k}[W_{4,k+1}]
&=\mathbf{E}_{\mathfrak{L}_k}[\tilde{f}(\bar{\bsx}_{k+1})-nf^*]\nonumber\\
&=\mathbf{E}_{\mathfrak{L}_k}[\tilde{f}(\bar{\bsx}_k)-nf^*+\tilde{f}(\bar{\bsx}_{k+1})
-\tilde{f}(\bar{\bsx}_k)]\nonumber\\
&\le\mathbf{E}_{\mathfrak{L}_k}\Big[\tilde{f}(\bar{\bsx}_k)-nf^*
-\eta_k(\bar{\bsg}_{k}^e)^\top\bsg^0_k
+\frac{1}{2}\eta^2_kL_f\|\bar{\bsg}_{k}^e\|^2\Big]\nonumber\\
&=W_{4,k}
-\eta_k(\bar{\bsg}_{k}^s)^\top\bsg^0_k
+\frac{1}{2}\eta^2_kL_f\mathbf{E}_{\mathfrak{L}_k}[\|\bar{\bsg}_{k}^e\|^2]\nonumber\\
&=W_{4,k}
-\eta_k(\bar{\bsg}_{k}^s)^\top\bar{\bsg}^0_k
+\frac{1}{2}\eta^2_kL_f\mathbf{E}_{\mathfrak{L}_k}[\|\bar{\bsg}_{k}^e\|^2]\nonumber\\
&=W_{4,k}
-\frac{1}{2}\eta_k(\bar{\bsg}_{k}^s)^\top(\bar{\bsg}^s_k+\bar{\bsg}^0_k-\bar{\bsg}^s_k)\nonumber\\
&\quad-\frac{1}{2}\eta_k(\bar{\bsg}^s_{k}-\bar{\bsg}^0_k+\bar{\bsg}^0_k)^\top\bar{\bsg}^0_k
+\frac{1}{2}\eta^2_kL_f\mathbf{E}_{\mathfrak{L}_k}[\|\bar{\bsg}_{k}^e\|^2]\nonumber\\
&\le W_{4,k}-\frac{1}{4}\eta_k(\|\bar{\bsg}^s_{k}\|^2
-\|\bar{\bsg}^0_k-\bar{\bsg}^s_k\|^2+\|\bar{\bsg}_{k}^0\|^2\nonumber\\
&\quad-\|\bar{\bsg}^0_k-\bar{\bsg}^s_k\|^2)
+\frac{1}{2}\eta^2_kL_f\mathbf{E}_{\mathfrak{L}_k}[\|\bar{\bsg}_{k}^e\|^2]\nonumber\\
&= W_{4,k}-\frac{1}{4}\eta_k\|\bar{\bsg}^s_{k}\|^2
+\frac{1}{2}\eta_k\|\bar{\bsg}^0_k-\bar{\bsg}^s_k\|^2\nonumber\\
&\quad-\frac{1}{4}\eta_k\|\bar{\bsg}_{k}^0\|^2
+\frac{1}{2}\eta^2_kL_f\mathbf{E}_{\mathfrak{L}_k}[\|\bar{\bsg}_{k}^e\|^2]\nonumber\\
&\le W_{4,k}-\frac{1}{4}\eta_k\|\bar{\bsg}^s_{k}\|^2
+\|\bsx_k\|^2_{\eta_kL_f^2\bsK}\nonumber\\
&\quad+nL_f^2\eta_k\delta^2_k-\frac{1}{4}\eta_k\|\bar{\bsg}_{k}^0\|^2
+\frac{1}{2}\eta^2_kL_f\mathbf{E}_{\mathfrak{L}_k}[\|\bar{\bsg}^e_{k}\|^2],\label{zerosg:v4k}
\end{align}
where the first inequality holds since that $\tilde{f}$ is smooth as shown in \eqref{zerosg:ass:fiu:equ}, \eqref{nonconvex:lemma:lipschitz} and \eqref{zerosg:xbardynamic-rand-pd}; the third equality holds since \eqref{zerosg:rand-grad-esti1} and that $x_{i,k}$ and $v_{i,k}$ are independent of $\mathfrak{L}_k$; the fourth equality holds due to $(\bar{\bsg}_{k}^s)^\top\bsg^0_k=(\bsg_{k}^s)^\top\bsH\bsg^0_k=(\bsg_{k}^s)^\top\bsH\bsH\bsg^0_k
=(\bar{\bsg}_{k}^s)^\top\bar{\bsg}^0_k$; the second inequality holds due to the Cauchy--Schwarz inequality; and the last inequality holds due to \eqref{zerosg:rand-grad-esti9}.

(v) Denote
\begin{align*}
\bsM_{1,k}&=(\alpha_k-\beta_k)\bsL-(1+3L_f^2)\bsK,\\
\bsM^0_{2,k}&=\beta^2_k\bsL+(2\alpha^2_k+\beta^2_k)\bsL^2+8L_f^2\bsK,\\
\bsM_{2,k}&=\bsM^0_{2,k}
+8p(1+\sigma_0^2)(1+\tilde{\sigma}_0^2)(6+L_f)L_f^2\bsK,\\
\bsM_{3}&=\frac{1}{2}\bsK-\kappa_1\kappa_2\bsL+\frac{1}{2}\kappa_1\kappa_2\bsL^2
+\frac{3}{2}\kappa_1^2\kappa_2^2\bsL^2
+\kappa_2^2(\bsL+\kappa_1\bsL^2),\\
b^0_{2,k}&=\frac{1}{2}\eta_k(2\beta_k-\kappa_3)-\frac{5}{2}\kappa_2^2
-\frac{1}{2}\omega_k(\kappa_1\kappa_2+3\kappa_2^2+\kappa_4)-\frac{1}{2}\omega_k\eta_k\kappa_4,\\
b^0_{4,k}&=\frac{1}{2}(6+L_f)+\frac{\kappa_3L_f^2}{2\kappa_2}\frac{1}{\beta_k}
+2\omega_k+\kappa_5L_f^2\frac{1}{\beta_k^2}
+\frac{\kappa_4L_f^2}{2\kappa_2}\frac{\omega_k}{\beta_k}
+\kappa_4L_f^2\frac{\omega_k}{\beta_k^2}.
\end{align*}

We have
\begin{align}
\mathbf{E}_{\mathfrak{L}_k}[W_{k+1}]
&\le W_{k}+\frac{1}{2}\omega_k\|\bm{x}_k\|^2_{\bsK}
-(1+\omega_k)\|\bsx_k\|^2_{\eta_k\alpha_k\bsL-\frac{1}{2}\eta_k\bsK
-\frac{3}{2}\eta^2_k\alpha^2_k\bsL^2-\eta_k(1+5\eta_k)L_f^2\bsK}\nonumber\\
&\quad+(1+\omega_k)\Big\|\bm{v}_k+\frac{1}{\beta_k}\bsg_k^0\Big\|^2_{\frac{3}{2}\eta^2_k\beta^2_k\bsK}
+(1+\omega_k)nL_f^2\eta_k(1+5\eta_k)\delta^2_k\nonumber\\
&\quad+2(1+\omega_k)\eta^2_k
\mathbf{E}_{\mathfrak{L}_k}[\|\bsg^e_k\|^2]
+\frac{1}{2}(\eta_k+\omega_k+\eta_k\omega_k)\kappa_4
\Big\|\bm{v}_k+\frac{1}{\beta_k}\bsg_{k}^0\Big\|^2_{\bsK}\nonumber\\
&\quad+\|\bsx_k\|^2_{(1+\omega_k)\eta^2_k\beta^2_k(\bsL+\kappa_1\bsL^2)}
+\frac{\eta_k}{\beta^2_k}\Big(\eta_k+\frac{1}{2}\Big)(1+\omega_k)\kappa_4L_f^2
\mathbf{E}_{\mathfrak{L}_k}[\|\bar{\bsg}^e_{k}\|^2]\nonumber\\
&\quad+\frac{1}{2}\kappa_4(\omega_k+\omega_k^2)
\mathbf{E}_{\mathfrak{L}_k}[\|\bsg_{k+1}^0\|^2]
+\|\bm{x}_k\|^2_{\eta_k(\beta_k\bsL+\frac{1}{2}\bsK)
+\eta^2_k(\frac{1}{2}\alpha^2_k-\alpha_k\beta_k+\beta^2_k)\bsL^2}\nonumber\\
&\quad+\|\bsx_k\|^2_{\frac{1}{2}\omega_k\eta_k\alpha_k\bsL^2+\eta_k(1+3\eta_k)L_f^2\bsK}
+\frac{\eta_k}{2\beta^2_k}(1+3\eta_k)L_f^2
\mathbf{E}_{\mathfrak{L}_k}[\|\bar{\bsg}^e_{k}\|^2]\nonumber\\
&\quad+nL_f^2\eta_k(1+3\eta_k)\delta^2_k
+\eta^2_k\mathbf{E}_{\mathfrak{L}_k}[\|\bsg^e_k\|^2]
-\Big\|\bm{v}_k+\frac{1}{\beta_k}\bsg_{k}^0\Big\|^2_{\eta_k(\beta_k-\frac{1}{2}
-\eta_k\beta^2_k-\frac{1}{2}\omega_k\alpha_k)\bsK}\nonumber\\
&\quad+\frac{1}{2}\omega_k\mathbf{E}_{\mathfrak{L}_k}[\|\bsg_{k+1}^0\|^2]
-\frac{1}{4}\eta_k\|\bar{\bsg}^s_{k}\|^2
+\|\bsx_k\|^2_{\eta_kL_f^2\bsK}\nonumber\\
&\quad+nL_f^2\eta_k\delta^2_k-\frac{1}{4}\eta_k\|\bar{\bsg}_{k}^0\|^2
+\frac{1}{2}\eta^2_kL_f\mathbf{E}_{\mathfrak{L}_k}[\|\bar{\bsg}^e_{k}\|^2]\nonumber\\
&\le W_{k}-\|\bsx_k\|^2_{\eta_k\bsM_{1,k}-\eta_k^2\bsM^0_{2,k}-\omega_k\bsM_{3}
-((\frac{1}{2}+L_f^2)\eta_k\omega_k+5L_f^2\eta_k^2\omega_k)\bsK}
-\Big\|\bm{v}_k+\frac{1}{\beta_k}\bsg_{k}^0\Big\|^2_{b^0_{2,k}\bsK}\nonumber\\
&\quad-\frac{1}{4}\eta_k\|\bar{\bsg}_{k}^0\|^2+b^0_{4,k}\eta^2_k
\mathbf{E}_{\mathfrak{L}_k}[\|\bsg^e_k\|^2]
+\frac{1}{2}(\kappa_3\omega_k+\kappa_4\omega_k^2)
\mathbf{E}_{\mathfrak{L}_k}[\|\bsg_{k+1}^0\|^2]\nonumber\\
&\quad+nL_f^2(3+\omega_k+8\eta_k+5\eta_k\omega_k)\eta_k\delta_k^2\label{zerosg:vkLya-1}\\
&\le W_{k}-\|\bsx_k\|^2_{\eta_k\bsM_{1,k}-\eta_k^2\bsM_{2,k}-\omega_k\bsM_{3}-b_{1,k}\bsK}
-\Big\|\bm{v}_k+\frac{1}{\beta_k}\bsg_{k}^0\Big\|^2_{b^0_{2,k}\bsK}\nonumber\\
&\quad-\eta_k\Big(\frac{1}{4}-(1+\tilde{\sigma}_0^2)(b_{3,k}+8p(1+\sigma_0^2)b_{4,k})\eta_k\Big)\|\bar{\bsg}^0_{k}\|^2\nonumber\\
&\quad+2pn\sigma^2_1b_{4,k}\eta_k^2+n\sigma^2_2(b_{3,k}+4p(1+\sigma_0^2)b_{4,k})\eta_k^2
+b_{5,k}\eta_k\delta_k^2,
\label{zerosg:vkLya}
\end{align}
where the first inequality holds due to \eqref{zerosg:v1k}, \eqref{zerosg:v2k}, \eqref{zerosg:v3k}, and \eqref{zerosg:v4k}; the second inequality holds due to $\alpha_k=\kappa_1\beta_k$,  $\eta_k=\frac{\kappa_2}{\beta_k}$ and $\|\bar{\bsg}^e_{k}\|^2\le\|\bsg^e_{k}\|^2$; the last inequality holds due to \eqref{zerosg:rand-grad-esti2} and \eqref{zerosg:rand-grad-esti4}.

From \eqref{nonconvex:KL-L-eq2}, $\alpha_k=\kappa_1\beta_k$, $\kappa_1>1$, $\beta_k\ge\varepsilon_{1}\ge1+3L_f^2$, and $\eta_k=\frac{\kappa_2}{\beta_k}$, we have
\begin{align}\label{zerosg:m1-rand-pd}
\eta_k\bsM_{1,k}\ge\varepsilon_2\kappa_2\bsK.
\end{align}
From \eqref{nonconvex:KL-L-eq2}, $\alpha_k=\kappa_1\beta_k$, $\beta_k\ge\varepsilon_{1}\ge(8+8p(1+\sigma_0^2)(1+\tilde{\sigma}_0^2)(6+L_f))^{\frac{1}{2}}L_f$, and $\eta_k=\frac{\kappa_2}{\beta_k}$, we have
\begin{align}\label{zerosg:m2-rand-pd}
\eta_k^2\bsM_{2,k}\le\varepsilon_3\kappa_2^2\bsK.
\end{align}
From \eqref{nonconvex:KL-L-eq2}, $\alpha_k=\kappa_1\beta_k$,  and $\eta_k=\frac{\kappa_2}{\beta_k}$, we have
\begin{align}\label{zerosg:m3-rand-pd}
\bsM_{3}\le\varepsilon_5\bsK.
\end{align}
From $\beta_k\ge\varepsilon_{1}\ge p\kappa_3\ge \kappa_3$ and $\eta_k=\frac{\kappa_2}{\beta_k}$, we have
\begin{align}\label{zerosg:vkLya-b1}
b^0_{2,k}\ge&b_{2,k}.
\end{align}

From \eqref{zerosg:vkLya}--\eqref{zerosg:vkLya-b1}, we know that \eqref{zerosg:sgproof-vkLya} holds.

Similar to the way to get \eqref{zerosg:sgproof-vkLya}, we have \eqref{zerosg:sgproof-vkLya-bounded}.
\end{proof}

\begin{lemma}\label{zerosg:lemma:sg2}
Suppose Assumptions~\ref{zerosg:ass:graph}--\ref{zerosg:ass:fig} hold. Suppose $\alpha_k=\kappa_1\beta_k$, $\beta_k=\kappa_0(k+t_1)^{\theta}$, and $\eta_k=\frac{\kappa_2}{\beta_k}$, where $\theta\in[0,1]$, $\kappa_0\ge \frac{c_0(\kappa_1,\kappa_2)}{t_1^\theta}$, $\kappa_1>c_1$, $\kappa_2\in(0,c_2(\kappa_1))$, and $t_1\ge (c_3(\kappa_1,\kappa_2))^{\frac{1}{\theta}}$. Let $\{\bsx_k\}$ be the sequence generated by Algorithm~\ref{zerosg:algorithm-random-pd}, then
\begin{subequations}
\begin{align}
&\mathbf{E}_{\mathfrak{L}_k}[W_{k+1}]
\le  W_{k}-\varepsilon_4\|\bsx_k\|^2_{\bsK}
-\varepsilon_6\Big\|\bm{v}_k+\frac{1}{\beta_k}\bsg_{k}^0\Big\|^2_{\bsK}
-\frac{1}{16}\eta_k\|\bar{\bsg}^0_{k}\|^2
+pn\varepsilon_{12}\eta_k^2+pn\varepsilon_{11}\eta_k\delta_k^2,
\label{zerosg:sgproof-vkLya2}\\
&\mathbf{E}_{\mathfrak{L}_k}[\breve{W}_{k+1}]
\le  \breve{W}_{k}-\varepsilon_4\|\bsx_k\|^2_{\bsK}
-\varepsilon_6\Big\|\bm{v}_k+\frac{1}{\beta_k}\bsg_{k}^0\Big\|^2_{\bsK}
+p\varepsilon_{13}\eta_k^2\|\bar{\bsg}^0_{k}\|^2
+pn\varepsilon_{12}\eta_k^2+pn\varepsilon_{11}\eta_k\delta_k^2,
\label{zerosg:sgproof-vkLya2-bounded}\\
&\mathbf{E}_{\mathfrak{L}_k}[W_{4,k+1}]\le W_{4,k}+\|\bsx_k\|^2_{2\eta_kL_f^2\bsK}
-\frac{3}{16}\eta_k\|\bar{\bsg}_{k}^0\|^2+p\eta_k^2\varepsilon_{15}+(n+p)L_f^2\eta_k\delta^2_k.
\label{zerosg:v4kspeed-diminishing-2}
\end{align}
\end{subequations}
\end{lemma}
\begin{proof}
(i) Noting that $\kappa_1>c_1>1$ and $\beta_k=\kappa_0(k+t_1)^{\theta}\ge\kappa_0t_1^{\theta}\ge c_0(\kappa_1,\kappa_2)\ge\varepsilon_{1}\ge1$, we know that all conditions needed in Lemma~\ref{zerosg:lemma:sg} are satisfied, so \eqref{zerosg:sgproof-vkLya} and \eqref{zerosg:sgproof-vkLya-bounded} hold.

From $\kappa_1>c_1=\frac{1}{\rho_2(L)}+1$, we have
\begin{align}
\varepsilon_2>0.\label{zerosg:varepsilon1and2}
\end{align}

From \eqref{zerosg:varepsilon1and2} and $\kappa_2\in(0,\min\{\frac{\varepsilon_2}{\varepsilon_3},~\frac{1}{5}\})$, we have
\begin{align}
\varepsilon_4>0~\text{and}~
\varepsilon_6>0.\label{zerosg:kappa4-6}
\end{align}

From $t_1\ge(c_3(\kappa_1,\kappa_2))^{\frac{1}{\theta}}$ and $c_3(\kappa_1,\kappa_2)=\frac{24(1+\tilde{\sigma}_0^2)\kappa_3}{\kappa_2}$, we have
\begin{align}\label{zerosg:b3k-1}
\frac{3\kappa_3}{2\kappa_2t_1^{\theta}}\le\frac{1}{16(1+\tilde{\sigma}_0^2)}.
\end{align}

From $\kappa_0\ge \frac{c_0(\kappa_1,\kappa_2)}{t_1^\theta}
\ge\frac{24(1+\tilde{\sigma}_0^2)\kappa_4}{\kappa_2t_1^{\theta}}
\ge\frac{24(1+\tilde{\sigma}_0^2)\kappa_4}{\kappa_2t_1^{3\theta}}$, we have
\begin{align}\label{zerosg:b3k-2}
\frac{3\kappa_4}{2\kappa_2\kappa_0t_1^{3\theta}}\le\frac{1}{16(1+\tilde{\sigma}_0^2)}.
\end{align}

From $\beta_k=\kappa_0(k+t_1)^{\theta}$, we have
\begin{align}\label{zerosg:omegak}
\omega_k&=\frac{1}{\beta_{k}}-\frac{1}{\beta_{k+1}}
=\frac{1}{\kappa_0}(\frac{1}{(k+t_1)^{\theta}}-\frac{1}{(k+t_1+1)^{\theta}})
\nonumber\\
&\le\frac{1}{\kappa_0(k+t_1)^{\theta}(k+t_1+1)^{\theta}}
\le\frac{\kappa_0}{\beta_k^2}\le1.
\end{align}

From \eqref{zerosg:omegak}, $\eta_k=\frac{\kappa_2}{\beta_k}$, $\beta_k\ge1$, $\omega_k\le1$, and $\kappa_0\ge \frac{c_0(\kappa_1,\kappa_2)}{t_1^\theta}
\ge(\frac{2p(1+\sigma_0^2)(1+\tilde{\sigma}_0^2)\varepsilon_7}{\varepsilon_4t_1^{3\theta}})^{\frac{1}{2}}$, we have
\begin{align}
b_{1,k}\le\frac{p(1+\sigma_0^2)(1+\tilde{\sigma}_0^2)\varepsilon_7}{\kappa_0^2t_1^{3\theta}}\le\frac{\varepsilon_4}{2} .\label{zerosg:b1k}
\end{align}

From \eqref{zerosg:omegak}, \eqref{zerosg:b1k}, $\kappa_0\ge \frac{c_0(\kappa_1,\kappa_2)}{t_1^\theta}\ge\frac{2\varepsilon_5}{\varepsilon_4t_1^{\theta}}$, and \eqref{zerosg:kappa4-6}, we have
\begin{align}
2\varepsilon_4-\varepsilon_5\omega_k-b_{1,k}
\ge2\varepsilon_4-\frac{\varepsilon_5}{\kappa_0t_1^{\theta}}-\frac{\varepsilon_4}{2}
\ge\varepsilon_4>0.\label{zerosg:varepsilon4}
\end{align}

From \eqref{zerosg:omegak}, $\eta_k=\frac{\kappa_2}{\beta_k}$, $\kappa_0\ge \frac{c_0(\kappa_1,\kappa_2)}{t_1^\theta}\ge\frac{\varepsilon_8}{2\varepsilon_6t_1^{\theta}}\ge\frac{\varepsilon_8}{2\varepsilon_6t_1^{2\theta}}$, and \eqref{zerosg:kappa4-6}, we have
\begin{align}
b_{2,k}\ge2\varepsilon_6-\frac{\varepsilon_8}{2\kappa_0t_1^{2\theta}}
\ge\varepsilon_6>0.\label{zerosg:b2k}
\end{align}

From \eqref{zerosg:b3k-1}--\eqref{zerosg:omegak} and $\eta_k=\frac{\kappa_2}{\beta_k}$, we have
\begin{align}\label{zerosg:b3keta}
b_{3,k}\eta_k\le\frac{3\kappa_3}{2\kappa_2t_1^{\theta}}
+\frac{3\kappa_4}{2\kappa_2\kappa_0t_1^{3\theta}}\le\frac{1}{8}.
\end{align}

From $\beta_k\ge1$ and $\omega_k\le1$, we have
\begin{subequations}
\begin{align}
b_{3,k}&\le \varepsilon_9,\label{zerosg:b3k}\\
b_{4,k}&\le \varepsilon_{10}.\label{zerosg:b4k}
\end{align}
\end{subequations}

From \eqref{zerosg:b3keta}, \eqref{zerosg:b4k}, and $\kappa_0\ge \frac{c_0(\kappa_1,\kappa_2)}{t_1^\theta}\ge\frac{128p(1+\sigma_0^2)(1+\tilde{\sigma}_0^2)\kappa_2\varepsilon_{10}}{t_1^{\theta}}$, we have
\begin{align}\label{zerosg:b3kb4keta}
&\frac{1}{4}-(1+\tilde{\sigma}_0^2)(b_{3,k}+8p(1+\sigma_0^2)b_{4,k})\eta_k\nonumber\\
&\ge\frac{1}{8}-8p(1+\sigma_0^2)(1+\tilde{\sigma}_0^2)b_{4,k}\eta_k
\ge\frac{1}{8}-\frac{8p(1+\sigma_0^2)(1+\tilde{\sigma}_0^2)\kappa_2\varepsilon_{10}}{\kappa_0t_1^{\theta}}
\ge\frac{1}{16}.
\end{align}

From \eqref{zerosg:b3kb4keta}, $\eta_k=\frac{\kappa_2}{\beta_k}$, $\beta_k\ge1$, and $\omega_k\le1$, we have
\begin{align}
b_{5,k}\le pn\varepsilon_{11}.\label{zerosg:b5k}
\end{align}

From \eqref{zerosg:sgproof-vkLya}, \eqref{zerosg:varepsilon4}, \eqref{zerosg:b2k}, and \eqref{zerosg:b3k}--\eqref{zerosg:b5k}, we know that \eqref{zerosg:sgproof-vkLya2} holds.

(ii) From \eqref{zerosg:sgproof-vkLya-bounded}, \eqref{zerosg:varepsilon4}, \eqref{zerosg:b2k}, \eqref{zerosg:b3k}, \eqref{zerosg:b4k}, and \eqref{zerosg:b5k}, we have \eqref{zerosg:sgproof-vkLya2-bounded}.

(iii) From \eqref{zerosg:v4k}, \eqref{zerosg:rand-grad-esti5}, and \eqref{zerosg:rand-grad-esti2}, we have
\begin{align}
\mathbf{E}_{\mathfrak{L}_k}[W_{4,k+1}]
&\le W_{4,k}-\frac{1}{4}\eta_k\|\bar{\bsg}^s_{k}\|^2
+\|\bsx_k\|^2_{\eta_kL_f^2\bsK}+nL_f^2\eta_k\delta^2_k\nonumber\\
&\quad-\frac{1}{4}\eta_k\|\bar{\bsg}_{k}^0\|^2
+\frac{1}{2}\eta_k^2L_f\Big(\frac{16p(1+\sigma_0^2)(1+\tilde{\sigma}_0^2)}{n}\|\bar{\bsg}_{k}^0\|^2
+\frac{16p(1+\sigma_0^2)(1+\tilde{\sigma}_0^2)}{n}L_f^2\|\bsx_{k}\|^2_{\bsK}\nonumber\\
&\quad+4p\sigma^2_1+8p(1+\sigma_0^2)\sigma^2_2+\frac{1}{2}p^2L_f^2\delta_k^2
+\|\bar{\bsg}^s_{k}\|^2\Big).\label{zerosg:v4kspeed-diminishing}
\end{align}

From $\eta_k=\frac{\kappa_2}{\beta_k}\le\frac{\kappa_2}{\kappa_0t_1^\theta}$ and $\kappa_0t_1^\theta\ge c_0(\kappa_1,\kappa_2)\ge128p(1+\sigma_0^2)(1+\tilde{\sigma}_0^2)\kappa_2\varepsilon_{10}
>128p(1+\sigma_0^2)(1+\tilde{\sigma}_0^2)\kappa_2L_f$, we have
\begin{subequations}
 \begin{align}
&\frac{8p(1+\sigma_0^2)(1+\tilde{\sigma}_0^2)}{n}\eta_k^2L_f
\le\frac{8p(1+\sigma_0^2)(1+\tilde{\sigma}_0^2)\kappa_2}{\kappa_0t_1^\theta}\eta_kL_f
<\frac{1}{16}\eta_k,\label{zerosg:v4kspeed-diminishing-1.1}\\
&\frac{8p(1+\sigma_0^2)(1+\tilde{\sigma}_0^2)}{n}\eta_k^2L_f^3
<\frac{1}{16}\eta_kL_f^2,\label{zerosg:v4kspeed-diminishing-1.2}\\
&\frac{1}{2}\eta_k^2L_f<\frac{1}{16}\eta_k,\label{zerosg:v4kspeed-diminishing-1.3}\\
&\frac{1}{4}p^2\eta_k^2L_f^3<pL_f^2\eta_k.\label{zerosg:v4kspeed-diminishing-1.4}
\end{align}
\end{subequations}

From \eqref{zerosg:v4kspeed-diminishing}--\eqref{zerosg:v4kspeed-diminishing-1.4}, we have \eqref{zerosg:v4kspeed-diminishing-2}.
\end{proof}

Now it is ready to prove Theorem~\ref{zerosg:thm-random-pd-sm}.

Denote
\begin{align*}
\hat{V}_k=\|\bm{x}_k\|^2_{\bsK}+\Big\|\bsv_k
+\frac{1}{\beta_k}\bsg_k^0\Big\|^2_{\bsK}+n(f(\bar{x}_k)-f^*).
\end{align*}
We have
\begin{align}
W_{k}
&=\frac{1}{2}\|\bsx_{k}\|^2_{\bsK}
+\frac{1}{2}\Big\|\bsv_k+\frac{1}{\beta_k}\bsg_k^0\Big\|^2_{\bsQ+\kappa_1\bsK}+\bsx_k^\top\bsK\Big(\bm{v}_k+\frac{1}{\beta_k}\bsg_k^0\Big)+n(f(\bar{x}_k)-f^*)\nonumber\\
&\ge\frac{1}{2}\|\bsx_{k}\|^2_{\bsK}
+\frac{1}{2}\Big(\frac{1}{\rho(L)}+\kappa_1\Big)
\Big\|\bsv_k+\frac{1}{\beta_k}\bsg_k^0\Big\|^2_{\bsK}
-\frac{1}{2\kappa_1}\|\bsx_{k}\|^2_{\bsK}
-\frac{1}{2}\kappa_1\Big\|\bsv_k+\frac{1}{\beta_k}\bsg_k^0\Big\|^2_{\bsK}
+n(f(\bar{x}_k)-f^*)\nonumber\\
&\ge\kappa_7\Big(\|\bsx_{k}\|^2_{\bsK}+\Big\|\bsv_k+\frac{1}{\beta_k}\bsg_k^0\Big\|^2_{\bsK}\Big)
+n(f(\bar{x}_k)-f^*)\label{zerosg:vkLya3.2}\\
&\ge\kappa_7\hat{V}_k\ge0,\label{zerosg:vkLya3}
\end{align}
where the first inequality holds due to \eqref{nonconvex:lemma-eq2} and the Cauchy--Schwarz inequality; and the last inequality holds due to $0<\kappa_7<\frac{1}{2}$. Similarly, we have
\begin{align}\label{zerosg:vkLya3.1}
W_k\le\kappa_6\hat{V}_k.
\end{align}

From \eqref{zerosg:sgproof-vkLya2} and \eqref{zerosg:kappa4-6}, we have
\begin{align}\label{zerosg:vkLya4}
\mathbf{E}_{\mathfrak{L}_k}[W_{k+1}]
&\le   W_{k}-\varepsilon_4\|\bsx_k\|^2_{\bsK}
-\frac{1}{16}\eta_k\|\bar{\bsg}^0_{k}\|^2+pn\varepsilon_{12}\eta_k^2+pn\varepsilon_{11}\eta_k\delta_k^2.
\end{align}
Then, taking expectation in $\calL_{T}$, summing \eqref{zerosg:vkLya4} over $ k\in[0,T-1]$,  noting $\eta_k=\frac{\kappa_2}{\kappa_0(k+t_1)^\theta}$, $\theta\in(0.5,1)$, and $\delta_k\le \frac{\kappa_\delta\sqrt{p\eta_k}}{\sqrt{n+p}}$ as stated in \eqref{zerosg:step:eta1-sm}, and
using \eqref{zerosg:serise:lemma:sum-equ} yield
\begin{align}\label{zerosg:vkLya4.1}
&\mathbf{E}[W_{T}]+\sum_{k=0}^{T-1}\mathbf{E}\Big[\varepsilon_4\|\bsx_k\|^2_{\bsK}
+\frac{1}{16}\eta_k\|\bar{\bsg}^0_{k}\|^2\Big]
\le W_{0}+\frac{pn(\varepsilon_{11}\kappa_\delta^2
+\varepsilon_{12})\kappa_2^2}{\kappa_0^2}\sum_{k=0}^{T-1}\frac{1}{(k+t_1)^{2\theta}}
\le n\varepsilon_{14}.
\end{align}

Noting that $t_1^\theta=\mathcal{O}(\sqrt{p})$, we have
\begin{align}\label{zerosg:k0}
\kappa_0=\mathcal{O}\Big(\frac{p}{t_1^\theta}\Big)=\mathcal{O}(\sqrt{p}).
\end{align}

From $W_0=\mathcal{O}(n)$ and \eqref{zerosg:k0}, we have
\begin{align}
\varepsilon_{14}=\frac{W_{0}}{n}+\frac{p(\varepsilon_{11}\kappa_\delta^2
+\varepsilon_{12})\kappa_2^2}{(2\theta-1)\kappa_0^2}=\mathcal{O}(1).\label{zerosg:c10}
\end{align}

From \eqref{zerosg:vkLya4.1}, \eqref{zerosg:vkLya3.2}, and $\kappa_7>0$, we have
\begin{align}\label{zerosg:thm-sg-sm-equ2p}
\mathbf{E}[f(\bar{x}_{T})]-f^*\le\frac{1}{n}\mathbf{E}[W_{T}]
\le \varepsilon_{14},~\forall T\in\mathbb{N}_0,
\end{align}
which gives \eqref{zerosg:thm-sg-sm-equ2}.

From \eqref{zerosg:vkLya4.1}, \eqref{zerosg:vkLya3}, and \eqref{zerosg:kappa4-6}, we have
\begin{align}\label{zerosg:thm-sg-sm-equ1.1p}
\sum_{k=0}^{T-1}\mathbf{E}[\|\bsx_k\|^2_{\bsK}]
\le\frac{n\varepsilon_{14}}{\varepsilon_4},~\forall T\in\mathbb{N}_+.
\end{align}

From \eqref{zerosg:rand-grad-smooth} and \eqref{zerosg:thm-sg-sm-equ2p}, we have
\begin{align}\label{zerosg:thm-sg-sm-bounded}
\mathbf{E}[\|\bar{\bsg}^0_k\|^2]\le2nL_f(\mathbf{E}[f(\bar{x}_k)]-f^*)
\le 2nL_f\varepsilon_{14}.
\end{align}

From \eqref{zerosg:rand-grad-esti2}, \eqref{zerosg:thm-sg-sm-equ1.1p}, and \eqref{zerosg:thm-sg-sm-bounded}, we know that $\mathbf{E}[\|\bsg^e_k\|^2]$ is bounded. Then, same as the proof of the first part of Theorem~1 in \cite{tang2020distributedzero}, we have \eqref{zerosg:thm-sg-sm-equ1bounded}.

From \eqref{zerosg:vkLya3.2} and \eqref{zerosg:vkLya3.1}, we have
\begin{align}\label{zerosg:vkLya3.2-bound}
0\le2\kappa_7(W_{1,k}+W_{2,k})\le\breve{W}_k\le2\kappa_6(W_{1,k}+W_{2,k}).
\end{align}

Denote $\breve{z}_k=\mathbf{E}[\breve{W}_k]$. From \eqref{zerosg:sgproof-vkLya2-bounded}, \eqref{zerosg:thm-sg-sm-bounded}, \eqref{zerosg:vkLya3.2-bound}, and \eqref{zerosg:step:eta1-sm}, we have
\begin{align}\label{zerosg:vkLya4-bound}
\breve{z}_{k+1}\le(1-a_1)\breve{z}_k+\frac{a_2}{(k+t_1)^{2\theta}}.
\end{align}

From $\kappa_1>1$, we have $\kappa_6>1$. From $0<\kappa_2<\frac{1}{5}$, we have $\varepsilon_6=\frac{1}{4}(\kappa_2-5\kappa_2^2)\le\frac{1}{80}$.
Thus,
\begin{align}\label{zerosg:vkLya2-a1-bounded}
0<a_1\le\frac{\varepsilon_6}{\kappa_6}\le\frac{1}{80}.
\end{align}

From \eqref{zerosg:kappa4-6}, we know that
\begin{align}\label{zerosg:vkLya2-pl-a1a2-bounded}
a_1>0~\text{and}~a_2>0.
\end{align}

From \eqref{zerosg:vkLya3.2-bound}--\eqref{zerosg:vkLya2-pl-a1a2-bounded} and \eqref{zerosg:serise:lemma:sequence-equ6}, we have
\begin{align}\label{zerosg:lemma:sequence-equ6-bounded}
2\kappa_7\mathbf{E}[\|\bsx_k\|^2_{\bsK}]\le\breve{z}_{k}\le \phi_3(k,t_1,a_1,a_2,2\theta,\breve{z}_{0}),~\forall k\in\mathbb{N}_+,
\end{align}
where the function $\phi_3$ is defined in \eqref{zerosg:serise:lemma:sequence-equ6-phi4}.

Noting that $\phi_3(k,t_1,a_1,a_2,2\theta,\breve{z}_0)=\mathcal{O}(\frac{n}{k^{2\theta}})$ due to \eqref{zerosg:k0}, from \eqref{zerosg:lemma:sequence-equ6-bounded}, we have \eqref{zerosg:thm-sg-sm-equ1.1bounded}.

From \eqref{zerosg:v4kspeed-diminishing-2}, we have
\begin{align}\label{zerosg:v4kspeed-diminishing-2-thm1}
\Big(\frac{1}{\eta_k}-\frac{1}{\eta_{k+1}}+\frac{1}{\eta_{k+1}}\Big)\mathbf{E}_{\mathfrak{L}_k}[W_{4,k+1}]
\le \frac{W_{4,k}}{\eta_k}+\|\bsx_k\|^2_{2L_f^2\bsK}
-\frac{3}{16}\|\bar{\bsg}_{k}^0\|^2+p\eta_k\varepsilon_{15}+(n+p)L_f^2\delta^2_k.
\end{align}
Then, taking expectation in $\calL_{T}$, summing \eqref{zerosg:v4kspeed-diminishing-2-thm1} over $ k\in[0,T-1]$,  noting \eqref{zerosg:thm-sg-sm-equ2p}, $\eta_k=\frac{\kappa_2}{\kappa_0(k+t_1)^\theta}$, $\theta\in(0.5,1)$, and $\delta_k\le \frac{\kappa_\delta\sqrt{p\eta_k}}{\sqrt{n+p}}$ as stated in \eqref{zerosg:step:eta1-sm}, and
using \eqref{zerosg:serise:lemma:sum-equ} yield
\begin{align}\label{zerosg:v4kspeed-diminishing-2-thm1-1}
&\frac{3}{16}\sum_{k=0}^{T-1}\mathbf{E}[\|\bar{\bsg}_{k}^0\|^2]\nonumber\\
&\le \frac{W_{4,0}}{\eta_0}
+\sum_{k=0}^{T-1}\Big(\frac{1}{\eta_{k+1}}-\frac{1}{\eta_{k}}\Big)\mathbf{E}[W_{4,k+1}]
+\sum_{k=0}^{T-1}\mathbf{E}[\|\bsx_k\|^2_{2L_f^2\bsK}]
+\frac{p(\varepsilon_{15}+L_f^2\kappa_\delta^2)\kappa_2(T+t_1)^{1-\theta}}{\kappa_0(1-\theta)}\nonumber\\
&\le \frac{n\varepsilon_{14}}{\eta_0}
+\sum_{k=0}^{T-1}\Big(\frac{1}{\eta_{k+1}}-\frac{1}{\eta_{k}}\Big)n\varepsilon_{14}
+\sum_{k=0}^{T-1}\mathbf{E}[\|\bsx_k\|^2_{2L_f^2\bsK}]
+\frac{p(\varepsilon_{15}+L_f^2\kappa_\delta^2)\kappa_2(T+t_1)^{1-\theta}}{\kappa_0(1-\theta)}\nonumber\\
&= \frac{n\varepsilon_{14}\kappa_0(T+t_1)^\theta}{\kappa_2}
+\sum_{k=0}^{T-1}\mathbf{E}[\|\bsx_k\|^2_{2L_f^2\bsK}]
+\frac{p(\varepsilon_{15}+L_f^2\kappa_\delta^2)\kappa_2(T+t_1)^{1-\theta}}{\kappa_0(1-\theta)}.
\end{align}
From \eqref{zerosg:v4kspeed-diminishing-2-thm1-1}, \eqref{zerosg:lemma:sequence-equ6-bounded}, \eqref{zerosg:k0}, and $\theta\in(0.5,1)$, we have \eqref{zerosg:thm-sg-sm-equ1}.

\subsection{Proof of Theorem~\ref{zerosg:thm-sg-smT}}\label{zerosg:proof-thm-sg-smT}
In addition to the notations defined in Appendix~\ref{zerosg:proof-thm-random-pd-sm},
we also denote the following notations.
\begin{align*}
&\tilde{c}_0(\kappa_1,\kappa_2)=\max\Big\{\varepsilon_{1},
~\Big(\frac{p(1+\sigma_0^2)(1+\tilde{\sigma}_0^2)\tilde{\varepsilon}_7}{\varepsilon_4}\Big)^{\frac{1}{3}},
~64p(1+\sigma_0^2)(1+\tilde{\sigma}_0^2)\kappa_2\tilde{\varepsilon}_{10}\Big\},\\
&\tilde{\varepsilon}_7
=8(1+3\kappa_2+\kappa_4+2\kappa_2\kappa_4)\kappa_2L_f^4,\\
&\tilde{\varepsilon}_{10}=6+L_f+\frac{1}{\kappa_2}(\kappa_4+1)L_f^2
+(3\kappa_4+3)L_f^2,\\
&\tilde{\varepsilon}_{11}=L_f^2\Big(\frac{1}{256(1+\sigma_0^2)(1+\tilde{\sigma}_0^2)}+\frac{8\kappa_2+3}{p}\Big),\\
&\tilde{\varepsilon}_{12}=2(\sigma^2_1+2(1+\sigma_0^2)\sigma^2_2)\tilde{\varepsilon}_{10},\\
&\tilde{\varepsilon}_{13}=8(1+\sigma_0^2)(1+\tilde{\sigma}_0^2)\tilde{\varepsilon}_{10},\\
&\tilde{a}_1=\frac{1}{\kappa_6}\min\{\varepsilon_{4},~2\varepsilon_{6}\}.
\end{align*}

To prove Theorem~\ref{zerosg:thm-sg-smT}, the following lemma is used.
\begin{lemma}\label{zerosg:lemma:sg2-T}
Suppose Assumptions~\ref{zerosg:ass:graph}--\ref{zerosg:ass:fig} hold. Suppose $\alpha_k=\alpha=\kappa_1\beta$, $\beta_k=\beta$, and $\eta_k=\eta=\frac{\kappa_2}{\beta}$, where
$\beta\ge\tilde{c}_0(\kappa_1,\kappa_2)$, $\kappa_1>c_1$, and $\kappa_2\in(0,c_2(\kappa_2))$ are constants. Let $\{\bsx_k\}$ be the sequence generated by Algorithm~\ref{zerosg:algorithm-random-pd}, then
\begin{subequations}
\begin{align}
\mathbf{E}_{\mathfrak{L}_k}[W_{k+1}]
&\le   W_{k}-\varepsilon_4\|\bsx_k\|^2_{\bsK}
-2\varepsilon_6\Big\|\bm{v}_k+\frac{1}{\beta}\bsg_{k}^0\Big\|^2_{\bsK}
-\frac{1}{8}\eta\|\bar{\bsg}^0_{k}\|^2+pn\tilde{\varepsilon}_{12}\eta^2+pn\tilde{\varepsilon}_{11}\eta\delta_k^2,
\label{zerosg:sgproof-vkLya2T}\\
\mathbf{E}_{\mathfrak{L}_k}[\breve{W}_{k+1}]&\le   \breve{W}_{k}-\varepsilon_4\|\bsx_k\|^2_{\bsK}-2\varepsilon_6\Big\|\bm{v}_k+\frac{1}{\beta}\bsg_{k}^0\Big\|^2_{\bsK}
+p\tilde{\varepsilon}_{13}\eta^2\|\bar{\bsg}^0_{k}\|^2+pn\tilde{\varepsilon}_{12}\eta^2+pn\tilde{\varepsilon}_{11}\eta\delta_k^2,
\label{zerosg:sgproof-vkLya2T-bounded}\\
\mathbf{E}_{\mathfrak{L}_k}[W_{4,k+1}]
&\le  W_{4,k}+\|\bsx_k\|^2_{2\eta L_f^2\bsK}-\frac{1}{8}\eta\|\bar{\bsg}_{k}^0\|^2
+p\varepsilon_{15}\eta^2
+(n+p)L_f^2\eta\delta^2_k.\label{zerosg:v4kspeed}
\end{align}
\end{subequations}
\end{lemma}
\begin{proof}
(i) Substituting $\alpha_k=\alpha=\kappa_1\beta$, $\beta_k=\beta$, $\eta_k=\eta=\frac{\kappa_2}{\beta}$, and $\omega_k=0$ into \eqref{zerosg:v1k}, \eqref{zerosg:v2k}, \eqref{zerosg:v3k}, and \eqref{zerosg:v4k}, similar to the way to get \eqref{zerosg:vkLya}, we have
\begin{align}
\mathbf{E}_{\mathfrak{L}_k}[W_{k+1}]
&\le W_{k}-\|\bsx_k\|^2_{\eta\tilde{\bsM}_{1}-\eta^2\tilde{\bsM}_{2}-\tilde{b}_{1}\bsK}
-\Big\|\bm{v}_k+\frac{1}{\beta_k}\bsg_{k}^0\Big\|^2_{\tilde{b}^0_{2}\bsK}\nonumber\\
&\quad-\eta\Big(\frac{1}{4}-8p(1+\sigma_0^2)(1+\tilde{\sigma}_0^2)\tilde{b}_{4}\eta\Big)\|\bar{\bsg}^0_{k}\|^2
+2pn(\sigma^2_1+2(1+\sigma_0^2)\sigma^2_2)\tilde{b}_{4}\eta^2+\tilde{b}_{5}\eta\delta_k^2,
\label{zerosg:vkLyaT}
\end{align}
where
\begin{align*}
\tilde{\bsM}_{1}&=(\alpha-\beta)\bsL-(1+3L_f^2)\bsK,\\
\tilde{\bsM}_{2}&=\beta^2\bsL+(2\alpha^2+\beta^2)\bsL^2+8L_f^2\bsK
+8p(1+\sigma_0^2)(1+\tilde{\sigma}_0^2)(6+ L_f)L_f^2\bsK,\\
\tilde{b}_{1}&=8p(1+\sigma_0^2)(1+\tilde{\sigma}_0^2)\kappa_3L_f^4\frac{\eta}{\beta^2}
+16p(1+\sigma_0^2)(1+\tilde{\sigma}_0^2)\kappa_5L_f^4\frac{\eta^2}{\beta^2},\\
\tilde{b}^0_{2}&=\frac{1}{2}\eta(2\beta-\kappa_3)-\frac{5}{2}\kappa_2^2,\\
\tilde{b}_{4}&=6+L_f+\frac{\kappa_3}{\kappa_2}L_f^2\frac{1}{\beta}
+2\kappa_5L_f^2\frac{1}{\beta^2},\\
\tilde{b}_{5}&=nL_f^2\Big(\frac{1}{4}p^2\tilde{b}_{4}\eta+3+8\eta\Big).
\end{align*}

From \eqref{zerosg:vkLyaT}, similar to the way to get \eqref{zerosg:sgproof-vkLya}, we have
\begin{align}\label{zerosg:sgproof-vkLyaT}
\mathbf{E}_{\mathfrak{L}_k}[W_{k+1}]
&\le  W_{k}-\|\bsx_k\|^2_{(2\varepsilon_4-\tilde{b}_{1})\bsK}
-\|\bm{v}_k+\frac{1}{\beta_k}\bsg_{k}^0\|^2_{2\varepsilon_6\bsK}\nonumber\\
&\quad-\eta\Big(\frac{1}{4}-8p(1+\sigma_0^2)(1+\tilde{\sigma}_0^2)\tilde{b}_{4}\eta\Big)\|\bar{\bsg}^0_{k}\|^2
+2pn(\sigma^2_1+2(1+\sigma_0^2)\sigma^2_2)\tilde{b}_{4}\eta^2+\tilde{b}_{5}\eta\delta_k^2.
\end{align}

From \eqref{zerosg:sgproof-vkLyaT}, similar to the way to get \eqref{zerosg:sgproof-vkLya2}, we have \eqref{zerosg:sgproof-vkLya2T}.

(ii) Similarly, we know that \eqref{zerosg:sgproof-vkLya2T-bounded} holds.

(iii) Noting $\eta_k=\eta$, $\beta\ge\tilde{c}_0(\kappa_1,\kappa_2)\ge64p(1+\sigma_0^2)(1+\tilde{\sigma}_0^2)\kappa_2\tilde{\varepsilon}_{10}
\ge64p(1+\sigma_0^2)(1+\tilde{\sigma}_0^2)\kappa_2L_f$, and $\eta=\frac{\kappa_2}{\beta}$, similar to the way to get \eqref{zerosg:v4kspeed-diminishing-2}, we have \eqref{zerosg:v4kspeed}.
\end{proof}

We are now ready to prove Theorem~\ref{zerosg:thm-sg-smT}.

From $\beta_k=\beta=\frac{\kappa_2\sqrt{pT}}{\sqrt{n}}$ and $T\ge  \frac{n(\tilde{c}_0(\kappa_1,\kappa_2))^2}{p\kappa_2^2}$, we have $\beta\ge\tilde{c}_0(\kappa_1,\kappa_2)$. Thus, all conditions needed in Lemma~\ref{zerosg:lemma:sg2-T} are satisfied. So \eqref{zerosg:sgproof-vkLya2T}--\eqref{zerosg:v4kspeed} hold.

Taking expectation in $\calL_{T}$, summing \eqref{zerosg:sgproof-vkLya2T} over $ k\in[0,T-1]$, noting $\eta_k=\eta=\frac{\sqrt{n}}{\sqrt{pT}}$ and $\delta_{i,k}\le\frac{p^{\frac{1}{4}}n^{\frac{1}{4}}\kappa_\delta}
{\sqrt{n+p}(k+1)^{\frac{1}{4}}}$ as stated in \eqref{zerosg:step:eta2-sm}, and
using \eqref{zerosg:serise:lemma:sum-equ} yield
\begin{align}
&\frac{1}{nT}\sum_{k=0}^{T-1}\mathbf{E}[\|\bsx_k\|^2_{\bsK}]
\le\frac{1}{\varepsilon_4}\Big(\frac{W_{0}}{nT}
+\frac{n\tilde{\varepsilon}_{12}}{T}
+\frac{2n\tilde{\varepsilon}_{11}\kappa_\delta^2}{T}\Big).
\label{zerosg:thm-sg-sm-equ3.1p}
\end{align}
Similarly, from \eqref{zerosg:v4kspeed} and \eqref{zerosg:step:eta2-sm},  we have
\begin{align}\label{zerosg:thm-sg-sm-equ3p}
\frac{1}{T}\sum_{k=0}^{T}\mathbf{E}[\|\nabla f(\bar{x}_k)\|^2]&=\frac{1}{nT}\sum_{k=0}^{T}\mathbf{E}[\|\bar{\bsg}_{k}^0\|^2]
\le 8\Big(\frac{W_{4,0}}{nT\eta}
+\frac{2L_f^2}{nT}\sum_{k=0}^{T}\mathbf{E}[\|\bsx_k\|^2_{\bsK}]+\frac{p\varepsilon_{15}\eta}{n}
+\frac{2\sqrt{p}L_f^2\kappa_\delta^2}{\sqrt{nT}}\Big).
\end{align}
Noting that $\eta=\frac{\kappa_2}{\beta_k}=\frac{\sqrt{n}}{\sqrt{pT}}$, from \eqref{zerosg:thm-sg-sm-equ3p} and \eqref{zerosg:thm-sg-sm-equ3.1p}, we have
\begin{align*}
\frac{1}{T}\sum_{k=0}^{T-1}\mathbf{E}[\|\nabla f(\bar{x}_k)\|^2]
&=8(f(\bar{x}_0)-f^*+\varepsilon_{15}+2L_f^2\kappa_\delta^2)\frac{\sqrt{p}}{\sqrt{nT}}
+\mathcal{O}\Big(\frac{n}{T}\Big),
\end{align*}
which gives \eqref{zerosg:coro-sg-sm-equ3}.

Taking expectation in $\calL_{T}$, summing \eqref{zerosg:v4kspeed} over $ k\in[0,T]$, and using \eqref{zerosg:step:eta2-sm}  yield
\begin{align}\label{zerosg:thm-sg-sm-equ4p}
n(\mathbf{E}[f(\bar{x}_{T})]-f^*)&=\mathbf{E}[W_{4,T}]
\le W_{4,0}+\frac{2L_f^2\sqrt{n}}{\sqrt{pT}} \sum_{k=0}^{T-1}\mathbf{E}[\|\bsx_k\|^2_{\bsK}]+n\varepsilon_{15}
+2nL_f^2\kappa_\delta^2.
\end{align}

Noting that $W_{4,0}=\mathcal{O}(n)$ and $\frac{\sqrt{n}n}{\sqrt{pT}}\le1$ due to $T\ge \frac{n^3}{p}$, from \eqref{zerosg:thm-sg-sm-equ3.1p} and \eqref{zerosg:thm-sg-sm-equ4p}, we have \eqref{zerosg:coro-sg-sm-equ4}. Then, from \eqref{zerosg:rand-grad-smooth}, we know that there exists a constant $\tilde{c}_g>0$, such that
\begin{align}\label{zerosg:coro-sg-sm-gbark0}
\mathbf{E}[\|\bar{\bsg}^0_k\|^2]\le n\tilde{c}_g,~\forall k\in\mathbb{N}_0.
\end{align}

From \eqref{zerosg:sgproof-vkLya2T-bounded}, \eqref{zerosg:coro-sg-sm-gbark0}, \eqref{zerosg:vkLya3.2-bound}, and \eqref{zerosg:step:eta2-sm}, we have
\begin{align}\label{zerosg:vkLya4-bound-tilde}
\breve{z}_{k+1}\le(1-\tilde{a}_1)\breve{z}_k
+\frac{n^2\tilde{\varepsilon}_{11}\kappa_\delta^2}{\sqrt{T(k+1)}}
+\frac{n^2(\tilde{\varepsilon}_{13}\tilde{c}_g+\tilde{\varepsilon}_{12})}{T},~\forall 0\le k\le T.
\end{align}

From \eqref{zerosg:vkLya2-a1-bounded}, we have
\begin{align}\label{zerosg:vkLya2-a1-bounded-thm2}
0<\tilde{a}_1\le\frac{2\varepsilon_6}{\kappa_6}\le\frac{1}{40}.
\end{align}

From \eqref{zerosg:vkLya4-bound-tilde} and \eqref{zerosg:vkLya2-a1-bounded-thm2}, similar to the way to get \eqref{zerosg:serise:lemma:sequence-equ6}, we have $\breve{z}_T=\mathcal{O}(\frac{n^2}{T})$. Then, from \eqref{zerosg:vkLya3.2-bound}, we have \eqref{zerosg:coro-sg-sm-equ3.1}.

Similar to the proof of \eqref{zerosg:thm-sg-sm-equ1bounded}, we have \eqref{zerosg:coro-sg-sm-equ3bounded}.


\subsection{Proof of Theorem~\ref{zerosg:thm-sg-diminishing}}\label{zerosg:proof-thm-sg-diminishing}
In addition to the notations defined in Appendix~\ref{zerosg:proof-thm-random-pd-sm}, we also denote the following notations.
\begin{align*}
&\varepsilon_{16}=\frac{1}{\kappa_6}\min\Big\{\frac{\varepsilon_4\kappa_0t_1^\theta}{\kappa_2},
~\frac{\varepsilon_6\kappa_0t_1^\theta}{\kappa_2},~\frac{\nu}{8}\Big\},\\
&\varepsilon_{17}=\frac{16\theta4^\theta\kappa_2(\varepsilon_{15}+L_f^2\kappa_\delta^2)}{3\nu\kappa_0},\\
&\breve{a}_2=pn(\varepsilon_{11}\kappa_\delta^2+\varepsilon_{12}
+\varepsilon_{13}\breve{c}_g)\frac{\kappa_2^2}{\kappa^2_0},\\
&a_3=\frac{\kappa_2\varepsilon_{16}}{\kappa_0},\\
&a_4=pn(\varepsilon_{11}\kappa_\delta^2+\varepsilon_{12})\frac{\kappa_2^2}{\kappa^2_0}.
\end{align*}


All conditions needed in Lemma~\ref{zerosg:lemma:sg2} are satisfied, so \eqref{zerosg:sgproof-vkLya2}--\eqref{zerosg:v4kspeed-diminishing-2} hold.

From \eqref{nonconvex:equ:plc}, we have that
\begin{align}\label{nonconvex:gg3}
\|\bar{\bsg}^0_k\|^2=n\|\nabla f(\bar{x}_k)\|^2\ge2\nu n(f(\bar{x}_k)-f^*)
=2\nu W_{4,k}.
\end{align}

From \eqref{zerosg:sgproof-vkLya2}, \eqref{nonconvex:gg3}, \eqref{zerosg:vkLya3}, and \eqref{zerosg:vkLya3.1}, we have
\begin{align}\label{zerosg:vkLya2-pl}
\mathbf{E}_{\mathfrak{L}_k}[W_{k+1}]
&\le W_{k}-\varepsilon_4\|\bsx_k\|^2_{\bsK}
-\varepsilon_6\|\bm{v}_k+\frac{1}{\beta_k}\bsg_{k}^0\|^2_{\bsK}-\frac{\eta_k\nu n}{8}W_{4,k}
+pn\varepsilon_{12}\eta_k^2+pn\varepsilon_{11}\eta_k\delta_k^2\nonumber\\
&\le W_{k}-\frac{\eta_k}{\kappa_6}\min\Big\{\frac{\varepsilon_4}{\eta_k},
~\frac{\varepsilon_6}{\eta_k},~\frac{\nu}{8}\Big\}W_{k}
+pn\varepsilon_{12}\eta_k^2+pn\varepsilon_{11}\eta_k\delta_k^2\nonumber\\
&\le W_{k}-\eta_k\varepsilon_{16}W_{k}+pn\varepsilon_{12}\eta_k^2
+pn\varepsilon_{11}\eta_k\delta_k^2,~\forall k\in\mathbb{N}_0.
\end{align}

Denote $z_k=\mathbf{E}[W_k]$, $r_{1,k}=\eta_k\varepsilon_{16}$, and $r_{2,k}=pn\varepsilon_{12}\eta_k^2
+pn\varepsilon_{11}\eta_k\delta_k^2$. From \eqref{zerosg:vkLya2-pl}, we have
\begin{align}
z_{k+1}
\le (1-r_{1,k})z_k+r_{2,k},~\forall k\in\mathbb{N}_0.
\label{zerosg:vkLya2-pl-z}
\end{align}

From \eqref{zerosg:step:eta1}, we have
\begin{align}
r_{1,k}&=\eta_k\varepsilon_{16}
=\frac{a_3}{(k+t_1)^\theta},\label{zerosg:vkLya2-pl-r1}\\
r_{2,k}&=pn\varepsilon_{12}\eta_k^2
+pn\varepsilon_{11}\eta_k\delta_k^2
\le\frac{a_4}{(k+t_1)^{2\theta}}.\label{zerosg:vkLya2-pl-r2}
\end{align}

From \eqref{zerosg:vkLya2-a1-bounded}, we have
\begin{align}\label{zerosg:vkLya2-pl-r1.1}
r_{1,k}\le\frac{\varepsilon_6}{\kappa_6}\le\frac{1}{80}.
\end{align}

From \eqref{zerosg:kappa4-6}, we know that
\begin{align}\label{zerosg:vkLya2-pl-a1a2}
a_3>0~\text{and}~a_4>0.
\end{align}

Then, from $\theta\in(0,1)$, \eqref{zerosg:vkLya2-pl-z}--\eqref{zerosg:vkLya2-pl-a1a2}, and \eqref{zerosg:serise:lemma:sequence-equ4}, we have
\begin{align}\label{zerosg:vkLya2-pl-theta0}
z_{k}\le\phi_1(k,t_1,a_3,a_4,\theta,2\theta,z_0),~\forall k\in\mathbb{N}_+,
\end{align}
where the function $\phi_1$ is defined in \eqref{zerosg:serise:lemma:sequence-equ4-phi2}.

From $t_1\ge(pc_3(\kappa_1,\kappa_2))^{1/\theta}$, we have
\begin{align}\label{zerosg:t1-pl}
t_1^\theta=\mathcal{O}(p).
\end{align}

From $\kappa_0\ge\frac{c_0(\kappa_1,\kappa_2)}{t_1^\theta}$, $t_1\le(pc_4c_3(\kappa_1,\kappa_2))^{\frac{1}{\theta}}$, $c_0(\kappa_1,\kappa_2)\ge\varepsilon_1\ge p\kappa_3$, and $c_3(\kappa_1,\kappa_2)=\frac{24(1+\tilde{\sigma}_0^2)\kappa_3}{\kappa_2}$, we have
\begin{align}\label{zerosg:k0-pl}
\frac{\kappa_2}{\kappa_0}\le\frac{\kappa_2t_1^\theta}{c_0(\kappa_1,\kappa_2)}
\le\frac{\kappa_2 pc_4c_3(\kappa_1,\kappa_2)}{p\kappa_3}\le 24(1+\tilde{\sigma}_0^2)c_4.
\end{align}
Thus,
\begin{align}\label{zerosg:phi1-pl}
\phi_1(k,t_1,a_3,a_4,\theta,2\theta,z_0)=\mathcal{O}\Big(\frac{pn}{(k+t_1)^\theta}\Big).
\end{align}

From \eqref{zerosg:vkLya3}, we have
\begin{align}\label{zerosg:vkLya5}
\|\bsx_{k}\|^2_{\bsK}+W_{4,k}
\le\hat{V}_k\le\frac{W_{k}}{\kappa_7}.
\end{align}

From \eqref{zerosg:rand-grad-smooth}, \eqref{zerosg:vkLya2-pl-theta0}, \eqref{zerosg:phi1-pl}, and \eqref{zerosg:vkLya5}, we get
\begin{align}\label{zerosg:gbark0-pl-1}
\mathbf{E}[\|\bar{\bsg}^0_k\|^2]
=\mathcal{O}\Big(\frac{pn}{(k+t_1)^\theta}\Big),~\forall k\in\mathbb{N}_+.
\end{align}

From \eqref{zerosg:t1-pl} and \eqref{zerosg:gbark0-pl-1}, we know that there exists a constant $\breve{c}_g>0$, such that
\begin{align}\label{zerosg:gbark0-pl}
\mathbf{E}[\|\bar{\bsg}^0_k\|^2]\le n\breve{c}_g,~\forall k\in\mathbb{N}_0.
\end{align}

From \eqref{zerosg:sgproof-vkLya2-bounded}, \eqref{zerosg:gbark0-pl}, \eqref{zerosg:vkLya3.2-bound}, and \eqref{zerosg:step:eta1}, we have
\begin{align}\label{zerosg:vkLya4-bound-pl}
\breve{z}_{k+1}\le(1-a_1)\breve{z}_k+\frac{\breve{a}_2}{(t+t_1)^{2\theta}}.
\end{align}

Using \eqref{zerosg:serise:lemma:sequence-equ6}, from \eqref{zerosg:vkLya2-a1-bounded} and \eqref{zerosg:vkLya4-bound-pl}, we have
\begin{align}\label{zerosg:lemma:sequence-equ6-bounded-pl}
\breve{z}_{k}&\le \phi_3(k,t_1,a_1,\breve{a}_2,2\theta,\breve{z}_{0}),~\forall k\in\mathbb{N}_+,
\end{align}
where the function $\phi_3$ is defined in \eqref{zerosg:serise:lemma:sequence-equ6-phi4}.
From \eqref{zerosg:lemma:sequence-equ6-bounded-pl}, \eqref{zerosg:vkLya3.2-bound}, \eqref{zerosg:serise:lemma:sequence-equ6-phi4}, and \eqref{zerosg:k0-pl}, we have
\begin{align}\label{zerosg:vkLya4-bound-brevez}
&\mathbf{E}[\|\bsx_k\|^2_{\bsK}]\le\frac{1}{\kappa_7}\breve{z}_{k}
\le\frac{1}{\kappa_7}\phi_3(k,t_1,a_1,\breve{a}_2,2\theta,\breve{z}_0)
=\mathcal{O}\Big(\frac{pn}{(k+t_1)^{2\theta}}\Big),
\end{align}
which yields \eqref{zerosg:thm-sg-diminishing-equ1.1bounded}.

From \eqref{zerosg:v4kspeed-diminishing-2}, \eqref{nonconvex:gg3}, and $\delta_k\le \frac{\kappa_\delta\sqrt{p\eta_k}}{\sqrt{n+p}}$ as stated in \eqref{zerosg:step:eta1}, we have
\begin{align}
\mathbf{E}[W_{4,k+1}]\le \mathbf{E}[W_{4,k}]-\frac{3\nu}{8}\eta_k\mathbf{E}[W_{4,k}]
+\|\bsx_k\|^2_{2\eta_kL_f^2\bsK}
+p\varepsilon_{15}\eta_k^2+pL_f^2\kappa_\delta^2\eta_k^2.
\label{zerosg:v4kspeed-diminishing-3}
\end{align}

Similar to the way to prove \eqref{zerosg:serise:lemma:sequence-equ4}, from \eqref{zerosg:vkLya4-bound-brevez} and \eqref{zerosg:v4kspeed-diminishing-3}, we have
\begin{align}\label{zerosg:v4kspeed-diminishing-4}
\mathbf{E}[f(\bar{x}_{T})-f^*]
\le\frac{\varepsilon_{17}p}{n(T+t_1)^\theta}+\mathcal{O}\Big(\frac{p}{(T+t_1)^{2\theta}}\Big).
\end{align}
From \eqref{zerosg:k0-pl}, we have
\begin{align}\label{zerosg:v4kspeed-diminishing-5}
\varepsilon_{17}=\frac{16\theta4^\theta\kappa_2(\varepsilon_{15}+L_f^2\kappa_\delta^2)}{3\nu\kappa_0}
\le\frac{128\theta4^\theta\kappa_2(\varepsilon_{15}+L_f^2\kappa_\delta^2)
(1+\tilde{\sigma}_0^2)c_4}
{\nu}.
\end{align}
Thus, from \eqref{zerosg:v4kspeed-diminishing-4} and \eqref{zerosg:v4kspeed-diminishing-5}, we have \eqref{zerosg:thm-sg-diminishing-equ1bounded}.

\subsection{Proof of Theorem~\ref{zerosg:thm-sg-diminishingt}}\label{zerosg:proof-thm-sg-diminishingt}
In addition to the notations defined in Appendices~\ref{zerosg:proof-thm-random-pd-sm} and \ref{zerosg:proof-thm-sg-diminishing}, we also denote the following notations.
\begin{align*}
&\hat{c}_3(\kappa_0,\kappa_1,\kappa_2)=\max\Big\{\frac{c_0(\kappa_1,\kappa_2)}{\kappa_0},
~c_3(\kappa_1,\kappa_2),~\frac{2}{3\varepsilon_4},
~\frac{2}{3\varepsilon_6}\Big\},\\
&\hat{\varepsilon}_{17}
=\frac{4\kappa_2^2(\varepsilon_{15}+L_f^2\kappa_\delta^2)}{\kappa_0^2(\frac{3\nu\kappa_2}{8\kappa_0}-1)},\\
&\hat{a}_2=pn(\varepsilon_{11}\kappa_\delta^2+\varepsilon_{12}
+\varepsilon_{13}\hat{c}_g)\frac{\kappa_2^2}{\kappa_0^2},\\
&\hat{a}_3=\frac{2}{3\kappa_6}.
\end{align*}


From $t_1\ge\hat{c}_3(\kappa_0,\kappa_1,\kappa_2)\ge\frac{c_0(\kappa_1,\kappa_2)}{\kappa_0}$, we have
\begin{align*}
\kappa_0\ge\frac{c_0(\kappa_1,\kappa_2)}{t_1}.
\end{align*}
Thus, all conditions needed in Lemma~\ref{zerosg:lemma:sg2} are satisfied, so \eqref{zerosg:vkLya2-pl-z}--\eqref{zerosg:vkLya2-pl-a1a2} still hold when $\theta=1$.

From rom $t_1\ge\hat{c}_3(\kappa_0,\kappa_1,\kappa_2)\ge \frac{2}{3\varepsilon_4}$, we have
\begin{align}\label{zerosg:vkLya2-pl-a1-1}
\frac{\varepsilon_4t_1}{\kappa_6}\ge\frac{2}{3\kappa_6}.
\end{align}

From $t_1\ge\hat{c}_3(\kappa_0,\kappa_1,\kappa_2)\ge \frac{2}{3\varepsilon_6}$, we have
\begin{align}\label{zerosg:vkLya2-pl-a1-2}
\frac{\varepsilon_6t_1}{\kappa_6}\ge\frac{2}{3\kappa_6}.
\end{align}

From $\kappa_0\in[\frac{3\hat{c}_0\nu\kappa_2}{16},\frac{3\nu\kappa_2}{16})$, we have
\begin{align}\label{zerosg:vkLya2-pl-a1-3}
\frac{16}{3\nu}<\frac{\kappa_2}{\kappa_0}\le\frac{16}{3\hat{c}_0\nu}.
\end{align}
Thus,
\begin{align}\label{zerosg:vkLya2-pl-a1-2.1}
\frac{\nu\kappa_2}{8\kappa_6\kappa_0}>\frac{2}{3\kappa_6}.
\end{align}

Hence, from \eqref{zerosg:vkLya2-pl-a1-1}, \eqref{zerosg:vkLya2-pl-a1-2}, \eqref{zerosg:vkLya2-pl-a1-2.1}, and $\kappa_6>1$ due to $\kappa_1>1$, we have
\begin{align}\label{zerosg:vkLya2-pl-a1}
a_3>\hat{a}_3>0~\text{and}~\hat{a}_3<\frac{2}{3}.
\end{align}

Then from $\theta=1$, \eqref{zerosg:vkLya2-pl-z}--\eqref{zerosg:vkLya2-pl-a1a2}, \eqref{zerosg:vkLya2-pl-a1}, and \eqref{zerosg:serise:lemma:sequence-equ5}, we have
\begin{align}\label{zerosg:vkLya2-pl-theta0t}
z_{k}\le\phi_2(k,t_1,\hat{a}_3,a_4,2,z_0),~\forall k\in\mathbb{N}_+,
\end{align}
where the function $\phi_2$ is defined in \eqref{zerosg:serise:lemma:sequence-equ5-phi3}.

From \eqref{zerosg:vkLya2-pl-a1} and \eqref{zerosg:vkLya2-pl-a1-3}, we have $\phi_2(k,t_1,\hat{a}_3,a_4,2,z_0)
=\mathcal{O}(\frac{nt_1^{\hat{a}_3}}{(k+t_1)^{\hat{a}_3}}
+\frac{pn}{(k+t_1)^{\hat{a}_3}t_1^{1-\hat{a}_3}})$. Hence, from \eqref{zerosg:rand-grad-smooth}, \eqref{zerosg:vkLya2-pl-theta0t}, and \eqref{zerosg:vkLya5}, we get
\begin{align}\label{zerosg:gbark0-pl-speed-1}
\mathbf{E}[\|\bar{\bsg}^0_k\|^2]
=\mathcal{O}\Big(\frac{nt_1^{\hat{a}_3}}{(k+t_1)^{\hat{a}_3}}
+\frac{pn}{(k+t_1)^{\hat{a}_3}t_1^{1-\hat{a}_3}}\Big),~\forall k\in\mathbb{N}_+.
\end{align}

Noting that $t_1>\hat{c}_3(\kappa_0,\kappa_1,\kappa_2)\ge \frac{c_0(\kappa_1,\kappa_2)}{\kappa_0}
\ge\frac{p\kappa_3}{\kappa_0}$, from \eqref{zerosg:gbark0-pl-speed-1} and \eqref{zerosg:vkLya2-pl-a1-3}, we know that there exists a constant $\hat{c}_g>0$, such that
\begin{align}\label{zerosg:gbark0-pl-speed}
\mathbf{E}[\|\bar{\bsg}^0_k\|^2]\le n\hat{c}_g,~\forall k\in\mathbb{N}_0.
\end{align}

From \eqref{zerosg:sgproof-vkLya2-bounded}, \eqref{zerosg:gbark0-pl-speed}, \eqref{zerosg:vkLya3.2-bound}, and \eqref{zerosg:step:eta1t1}, we have
\begin{align}\label{zerosg:vkLya4-bound-pl-1}
\breve{z}_{k+1}\le(1-a_1)\breve{z}_k+\frac{\hat{a}_2}{(t+t_1)^{2}}.
\end{align}

Using \eqref{zerosg:serise:lemma:sequence-equ6}, from \eqref{zerosg:vkLya2-a1-bounded} and \eqref{zerosg:vkLya4-bound-pl-1}, we have
\begin{align}\label{zerosg:lemma:sequence-equ6-bounded-pl-speed}
\breve{z}_{k}&\le \phi_3(k,t_1,a_1,\hat{a}_2,2,\breve{z}_{0}),~\forall k\in\mathbb{N}_+,
\end{align}
where the function $\phi_3$ is defined in \eqref{zerosg:serise:lemma:sequence-equ6-phi4}.
From \eqref{zerosg:lemma:sequence-equ6-bounded-pl-speed}, \eqref{zerosg:vkLya3.2-bound}, \eqref{zerosg:serise:lemma:sequence-equ6-phi4}, and \eqref{zerosg:vkLya2-pl-a1-3}, we have
\begin{align}\label{zerosg:vkLya4-bound-brevez-speed}
&\mathbf{E}[\|\bsx_k\|^2_{\bsK}]\le\frac{1}{\kappa_7}\breve{z}_{k}
\le\frac{1}{\kappa_7}\phi_3(k,t_1,a_1,\hat{a}_2,2,\breve{z}_0)
=\mathcal{O}\Big(\frac{pn}{(k+t_1)^{2}}\Big),
\end{align}
which yields \eqref{zerosg:thm-sg-diminishing-equ2.1bounded}.

From $\kappa_0<\frac{3\nu\kappa_2}{16}$, we have
\begin{align}\label{zerosg:vkLya2-pl-a1-3-bounded}
\frac{3\nu\kappa_2}{8\kappa_0}>2.
\end{align}

Same to the way to prove \eqref{zerosg:serise:lemma:sequence-equ5}, from \eqref{zerosg:vkLya4-bound-brevez-speed}, \eqref{zerosg:vkLya2-pl-a1-3-bounded}, and \eqref{zerosg:v4kspeed-diminishing-3}, we have \begin{align}\label{zerosg:v4kspeed-diminishing-6}
\mathbf{E}[f(\bar{x}_{T})-f^*]
\le\frac{\hat{\varepsilon}_{17}p}{n(T+t_1)}+\mathcal{O}\Big(\frac{p}{(T+t_1)^{2}}\Big).
\end{align}
From \eqref{zerosg:vkLya2-pl-a1-3}, we have
\begin{align}\label{zerosg:v4kspeed-diminishing-7}
&\hat{\varepsilon}_{17}
=\frac{4\kappa_2^2(\varepsilon_{15}+L_f^2\kappa_\delta^2)}{\kappa_0^2(\frac{3\nu\kappa_2}{8\kappa_0}-1)}
\le\frac{1024(\varepsilon_{15}+L_f^2\kappa_\delta^2)}{9\hat{c}_0(2-\hat{c}_0)\nu^2}.
\end{align}
Thus, from \eqref{zerosg:v4kspeed-diminishing-6} and \eqref{zerosg:v4kspeed-diminishing-7}, we have \eqref{zerosg:thm-sg-diminishing-equ2bounded}.

\subsection{Proof of Theorem~\ref{zerosg-a5:thm-sg-diminishingt}}\label{zerosg-a5:proof-thm-sg-diminishingt}
In addition to the notations defined in Appendices~\ref{zerosg:proof-thm-random-pd-sm}, \ref{zerosg:proof-thm-sg-diminishing}, and \ref{zerosg:proof-thm-sg-diminishingt}, we also denote the following notations.
\begin{align*}
&\check{c}_3(\kappa_0,\kappa_1,\kappa_2)=\max\Big\{
\frac{2}{3\varepsilon_4},~\frac{2}{3\varepsilon_6},
~\frac{\check{\varepsilon}_{1}}{\kappa_0},
~\frac{2\varepsilon_5}{\kappa_0\varepsilon_4},
~\Big(\frac{2p(1+\sigma_0^2)\check{\varepsilon}_7}{\kappa_0^2\varepsilon_4}\Big)^{\frac{1}{3}},
~\Big(\frac{\varepsilon_8}{2\kappa_0\varepsilon_6}\Big)^{\frac{1}{2}},\\
&\quad\quad\quad\quad\quad\quad~\frac{16L_f\kappa_3}{\nu\kappa_2},
~\Big(\frac{16L_f\kappa_4}{\nu\kappa_0\kappa_2}\Big)^{\frac{1}{3}},
~\frac{64p(1+\sigma_0^2)L_f\kappa_2\check{\varepsilon}_{10}}{\nu\kappa_0},
~\frac{p(1+\sigma_0^2)\kappa_2\check{\varepsilon}_{10}}{4\kappa_0}\Big\},\\
&\check{\varepsilon}_{1}=\max\{1+3L_f^2,~(8+4p(1+\sigma_0^2)
(6+L_f))^{\frac{1}{2}}L_f,~p\kappa_3\},\\
&\check{\varepsilon}_7
=4(6+5\kappa_2+2\kappa_4+8\kappa_2\kappa_4)\kappa_2L_f^4
+\frac{(1+2L_f^2)\kappa_2}{2p(1+\sigma_0^2)}
+\Big(\frac{5}{p(1+\sigma_0^2)}+16\Big)L_f^2\kappa_2^2,\\
&\check{\varepsilon}_{10}=10+L_f+\frac{1}{\kappa_2}(2\kappa_4+1)L_f^2
+(8\kappa_4+5)L_f^2,\\
&\check{\varepsilon}_{11}=L_f^2\Big(\frac{1}{1+\sigma_0^2}+\frac{13\kappa_2+4}{p}\Big),\\
&\check{\varepsilon}_{12}=2\check{\varepsilon}_{10}\sigma^2_1
+\frac{2\varepsilon_9\check{\sigma}^2_2}{3p}
+8(1+\sigma_0^2)\check{\varepsilon}_{10}\check{\sigma}^2_2,\\
&\check{\varepsilon}_{13}=\frac{4L_f\varepsilon_9}{3p}
+8L_f(1+\sigma_0^2)\check{\varepsilon}_{10},\\
&\check{\sigma}^2_2=2L_ff^*-\frac{2L_f}{n}\sum_{i=1}^{n}f_i^*\ge0.
\end{align*}

To prove Theorem~\ref{zerosg-a5:thm-sg-diminishingt}, the following lemma is used.
\begin{lemma}\label{zerosg-a5:lemma:sg2}
Suppose Assumptions~\ref{zerosg:ass:graph}--\ref{zerosg:ass:zeroth-variance} hold and each $f_i^*>-\infty$. Suppose $\alpha_k=\kappa_1\beta_k$, $\beta_k=\kappa_0(k+t_1)^{\theta}$, and $\eta_k=\frac{\kappa_2}{\beta_k}$, where $\theta\in[0,1]$, $\kappa_0>0$, $\kappa_1>c_1$, $\kappa_2\in(0,c_2(\kappa_1))$, and $t_1^\theta\ge \check{c}_3(\kappa_0,\kappa_1,\kappa_2)$ with any constant $\nu>0$. Let $\{\bsx_k\}$ be the sequence generated by Algorithm~\ref{zerosg:algorithm-random-pd}, then
\begin{subequations}
\begin{align}
\mathbf{E}_{\mathfrak{L}_k}[W_{k+1}]
&\le  W_{k}-\varepsilon_4\|\bsx_k\|^2_{\bsK}
-\varepsilon_6\Big\|\bm{v}_k+\frac{1}{\beta_k}\bsg_{k}^0\Big\|^2_{\bsK}
-\frac{1}{4}\eta_k\|\bar{\bsg}^0_{k}\|^2\nonumber\\
&\quad+\frac{3\nu}{8}\eta_kW_{4,k}
+pn\check{\varepsilon}_{12}\eta_k^2+pn\check{\varepsilon}_{11}\eta_k\delta_k^2,
\label{zerosg-a5:sgproof-vkLya2}\\
\mathbf{E}_{\mathfrak{L}_k}[\breve{W}_{k+1}]
&\le  \breve{W}_{k}-\varepsilon_4\|\bsx_k\|^2_{\bsK}
-\varepsilon_6\Big\|\bm{v}_k+\frac{1}{\beta_k}\bsg_{k}^0\Big\|^2_{\bsK}
\nonumber\\
&\quad+p\check{\varepsilon}_{13}\eta_k^2W_{4,k}
+pn\check{\varepsilon}_{12}\eta_k^2+pn\check{\varepsilon}_{11}\eta_k\delta_k^2,
\label{zerosg-a5:sgproof-vkLya2-bounded}\\
\mathbf{E}_{\mathfrak{L}_k}[W_{4,k+1}]
&\le W_{4,k}+\|\bsx_k\|^2_{2\eta_kL_f^2\bsK}
-\frac{1}{4}\eta_k\|\bar{\bsg}_{k}^0\|^2+\frac{\nu}{8}\eta_kW_{4,k}\nonumber\\
&\quad+2p\eta_k^2L_f(\sigma^2_1+2(1+\sigma_0^2)\check{\sigma}^2_2)
+(n+p)L_f^2\eta_k\delta^2_k.
\label{zerosg-a5:v4kspeed-diminishing-2}
\end{align}
\end{subequations}
\end{lemma}
\begin{proof} The proof of this lemma is similar to the proof of Lemma~\ref{zerosg:lemma:sg2}. For the sake of completeness, the proof is included here.

We know that \eqref{zerosg:rand-grad-esti1}--\eqref{zerosg:rand-grad-smooth} and \eqref{zerosg:rand-grad2-esti1} still hold since  Assumptions~\ref{zerosg:ass:zeroth-smooth} and \ref{zerosg:ass:zeroth-variance} hold. Moreover, \eqref{zerosg:vkLya-1} still holds.

We have
\begin{align}
&\|\bsg^0_{k}\|^2
=\sum_{i=1}^{n}\|\nabla f_i(\bar{x}_k)\|^2
\le\sum_{i=1}^{n}2L_f(f_i(\bar{x}_k)-f_i^*)=2L_fn(f(\bar{x}_k)-f^*)+n\check{\sigma}^2_2,\label{zerosg-a5:vkLya-2g0k}
\end{align}
where the inequality holds due to \eqref{nonconvex:lemma:lipschitz2}.

We have
\begin{align}\label{zerosg-a5:rand-grad2-esti1.2}
&\|\bsg_{k}\|^2
=\|\bsg_{k}-\bsg^0_{k}+\bsg^0_{k}\|^2
\le 2(\|\bsg_{k}-\bsg^0_{k}\|^2
+\|\bsg^0_{k}\|^2)\le 2(L_f^2\|\bsx_{k}\|^2_{\bsK}+2L_fW_{4,k}+n\check{\sigma}^2_2),
\end{align}
where the first inequality holds due to the Cauchy--Schwarz inequality; and the last inequality holds due to \eqref{zerosg:gg1} and \eqref{zerosg-a5:vkLya-2g0k}.

From \eqref{zerosg:rand-grad2-esti1} and \eqref{zerosg-a5:rand-grad2-esti1.2}, we have
\begin{align}
\mathbf{E}_{\mathfrak{L}_k}[\|\bsg^e_k\|^2]
&\le  16p(1+\sigma_0^2)L_fW_{4,k}+8p(1+\sigma_0^2)L_f^2\|\bsx_{k}\|^2_{\bsK}+4np\sigma^2_1
+8np(1+\sigma_0^2)\check{\sigma}^2_2+\frac{1}{2}np^2L_f^2\delta_k^2.
\label{zerosg-a5:rand-grad-esti2}
\end{align}

From  the Cauchy--Schwarz inequality, \eqref{zerosg:gg-rand-pd}, and \eqref{zerosg-a5:vkLya-2g0k}, we have
\begin{align}
&\|\bsg^0_{k+1}\|^2
=\|\bsg^0_{k+1}-\bsg^0_{k}+\bsg^0_{k}\|^2
\le2(\|\bsg^0_{k+1}-\bsg^0_{k}\|^2+\|\bsg^0_{k}\|^2)
\le 2(\eta^2_kL_f^2\|\bsg^e_k\|^2+2L_fW_{4,k}+n\check{\sigma}^2_2).\label{zerosg-a5:vkLya-2}
\end{align}

Then, from \eqref{zerosg:vkLya-1}, \eqref{zerosg-a5:rand-grad-esti2}, and \eqref{zerosg-a5:vkLya-2}, we get
\begin{align}\label{zerosg:vkLya-thm5}
\mathbf{E}_{\mathfrak{L}_k}[W_{k+1}]
&\le  W_{k}-\|\bsx_k\|^2_{\eta_k\bsM_{1,k}-\eta_k^2\check{\bsM}_{2,k}-\omega_k\bsM_{3}-\check{b}_{1,k}\bsK}
-\Big\|\bm{v}_k+\frac{1}{\beta_k}\bsg_{k}^0\Big\|^2_{b^0_{2,k}\bsK}-\frac{1}{4}\eta_k\|\bar{\bsg}^0_{k}\|^2\nonumber\\
&\quad+L_f\Big(\frac{4}{3}b_{3,k}+8p(1+\sigma_0^2)\check{b}_{4,k}\Big)\eta_k^2W_{4,k}
+2pn\sigma^2_1\check{b}_{4,k}\eta_k^2\nonumber\\
&\quad+n\check{\sigma}^2_2\Big(\frac{2}{3}b_{3,k}+4p(1+\sigma_0^2)\check{b}_{4,k}\Big)\eta_k^2
+\check{b}_{5,k}\eta_k\delta_k^2,
\end{align}
where
\begin{align*}
\check{\bsM}_{2,k}&=\bsM^0_{2,k}
+4p(1+\sigma_0^2)(6+L_f)L_f^2\bsK,\\
\check{b}_{1,k}
&=4p(1+\sigma_0^2)\kappa_3L_f^4\frac{\eta_k}{\beta_k^2}
+8p(1+\sigma_0^2)\kappa_5L_f^4\frac{\eta_k^2}{\beta_k^2}+\Big(\frac{1}{2}+L_f^2\Big)\eta_k\omega_k\\
&\quad+4p(1+\sigma_0^2)\kappa_4L_f^4\frac{\eta_k\omega_k}{\beta_k^2}
+(5+16p(1+\sigma_0^2)+8p(1+\sigma_0^2)\kappa_3L_f^2)L_f^2\eta_k^2\omega_k\\
&\quad+8p(1+\sigma_0^2)\kappa_4L_f^4\eta_k^2\omega_k^2
+8p(1+\sigma_0^2)\kappa_4L_f^4\frac{\eta_k^2\omega_k}{\beta_k^2},\\
\check{b}_{4,k}&=6+L_f+\frac{\kappa_3L_f^2}{\kappa_2}\frac{1}{\beta_k}+2\kappa_5L_f^2\frac{1}{\beta_k^2}
+\frac{\kappa_4L_f^2}{\kappa_2}\frac{\omega_k}{\beta_k}\\
&\quad+(4+2\kappa_3L_f^2)\omega_k+2\kappa_4L_f^2\omega_k^2+2\kappa_4L_f^2\frac{\omega_k}{\beta_k^2},\\
\check{b}_{5,k}
&=nL_f^2\Big(\frac{1}{4}p^2\check{b}_{4,k}\eta_k+3+\omega_k+8\eta_k+5\eta_k\omega_k\Big).
\end{align*}
Then, from \eqref{zerosg:vkLya-thm5} and $t_1^\theta\ge \check{c}_3(\kappa_0,\kappa_1,\kappa_2)\ge\frac{\check{\varepsilon}_{1}}{\kappa_0}$, similar to the way to get \eqref{zerosg:sgproof-vkLya}--\eqref{zerosg:sgproof-vkLya-bounded}, we have
\begin{subequations}
\begin{align}
\mathbf{E}_{\mathfrak{L}_k}[W_{k+1}]
&\le  W_{k}-\|\bsx_k\|^2_{(2\varepsilon_4-\varepsilon_5\omega_k-\check{b}_{1,k})\bsK}
-\Big\|\bm{v}_k+\frac{1}{\beta_k}\bsg_{k}^0\Big\|^2_{b_{2,k}\bsK}-\frac{1}{4}\eta_k\|\bar{\bsg}^0_{k}\|^2\nonumber\\
&\quad+L_f\Big(\frac{4}{3}b_{3,k}+8p(1+\sigma_0^2)\check{b}_{4,k}\Big)\eta_k^2W_{4,k}
+2pn\sigma^2_1\check{b}_{4,k}\eta_k^2\nonumber\\
&\quad+n\check{\sigma}^2_2\Big(\frac{2}{3}b_{3,k}+4p(1+\sigma_0^2)\check{b}_{4,k}\Big)\eta_k^2
+\check{b}_{5,k}\eta_k\delta_k^2,
\label{zerosg:sgproof-vkLya-thm5}\\
\mathbf{E}_{\mathfrak{L}_k}[\breve{W}_{k+1}]
&\le  \breve{W}_{k}-\|\bsx_k\|^2_{(2\varepsilon_4-\varepsilon_5\omega_k-\check{b}_{1,k})\bsK}
-\Big\|\bm{v}_k+\frac{1}{\beta_k}\bsg_{k}^0\Big\|^2_{b_{2,k}\bsK}\nonumber\\
&\quad+L_f\Big(\frac{4}{3}b_{3,k}+8p(1+\sigma_0^2)\check{b}_{4,k}\Big)\eta_k^2W_{4,k}
+2pn\sigma^2_1\check{b}_{4,k}\eta_k^2\nonumber\\
&\quad+n\check{\sigma}^2_2\Big(\frac{2}{3}b_{3,k}+4p(1+\sigma_0^2)\check{b}_{4,k}\Big)\eta_k^2
+\check{b}_{5,k}\eta_k\delta_k^2.
\label{zerosg:sgproof-vkLya-bounded-thm5}
\end{align}
\end{subequations}

From $t_1^\theta\ge \check{c}_3(\kappa_0,\kappa_1,\kappa_2)\ge\max\{\frac{2\varepsilon_5}{\kappa_0\varepsilon_4},
~(\frac{2p(1+\sigma_0^2)\check{\varepsilon}_7}{\kappa_0^2\varepsilon_4})^{\frac{1}{3}}\}$, similar to the way to prove \eqref{zerosg:varepsilon4}, we have
\begin{align}
2\varepsilon_4-\varepsilon_5\omega_k-\check{b}_{1,k}
\ge2\varepsilon_4-\frac{\varepsilon_5}{\kappa_0t_1^{\theta}}-
\frac{p(1+\sigma_0^2)\check{\varepsilon}_7}{\kappa_0^2t_1^{3\theta}}
\ge\varepsilon_4>0.\label{zerosg:varepsilon4-thm5}
\end{align}

From $t_1^\theta\ge \check{c}_3(\kappa_0,\kappa_1,\kappa_2)\ge(\frac{\varepsilon_8}{2\kappa_0\varepsilon_6})^{\frac{1}{2}}$, similar to the way to prove \eqref{zerosg:b2k}, we have
\begin{align}
b_{2,k}\ge2\varepsilon_6-\frac{\varepsilon_8}{2\kappa_0t_1^{2\theta}}
\ge\varepsilon_6>0.\label{zerosg:b2k-thm5}
\end{align}

From $\beta_k\ge1$ and $\omega_k\le1$, we have
\begin{align}
\check{b}_{4,k}&\le \check{\varepsilon}_{10}.\label{zerosg:b4k-thm5}
\end{align}

From
$t_1^\theta\ge \check{c}_3(\kappa_0,\kappa_1,\kappa_2)\ge\max\{\frac{16L_f\kappa_3}{\nu\kappa_2},
~(\frac{16L_f\kappa_4}{\nu\kappa_0\kappa_2})^{\frac{1}{3}},
~\frac{64p(1+\sigma_0^2)L_f\kappa_2\check{\varepsilon}_{10}}{\nu\kappa_0}\}$, similar to the way to get \eqref{zerosg:b3kb4keta}, we have
\begin{align}\label{zerosg-a5:b3kb4keta}
\frac{4}{3}L_f(b_{3,k}+6p(1+\sigma_0^2)\check{b}_{4,k})\eta_k
\le\frac{2\kappa_3}{\kappa_2t_1^{\theta}}
+\frac{2\kappa_4}{\kappa_2\kappa_0t_1^{3\theta}}
+\frac{8p(1+\sigma_0^2)L_f\kappa_2\check{\varepsilon}_{10}}{\kappa_0t_1^{\theta}}
\le\frac{3\nu}{8}.
\end{align}

From $t_1^\theta\ge \check{c}_3(\kappa_0,\kappa_1,\kappa_2)\ge\frac{p(1+\sigma_0^2)\kappa_2\check{\varepsilon}_{10}}{4\kappa_0}$, we have
\begin{align}
\check{b}_{5,k}\le pn\check{\varepsilon}_{11}.\label{zerosg:b5k-thm5}
\end{align}

From \eqref{zerosg:sgproof-vkLya-thm5}--\eqref{zerosg:sgproof-vkLya-bounded-thm5} and \eqref{zerosg:varepsilon4-thm5}--\eqref{zerosg:b5k-thm5}, we have \eqref{zerosg-a5:sgproof-vkLya2}--\eqref{zerosg-a5:sgproof-vkLya2-bounded}.

From \eqref{zerosg:v4k}, \eqref{zerosg:rand-grad-esti5}, and \eqref{zerosg-a5:rand-grad-esti2}, we have
\begin{align}
\mathbf{E}_{\mathfrak{L}_k}[W_{4,k+1}]
&\le W_{4,k}-\frac{1}{4}\eta_k\|\bar{\bsg}^s_{k}\|^2
+\|\bsx_k\|^2_{\eta_kL_f^2\bsK}+nL_f^2\eta_k\delta^2_k\nonumber\\
&\quad-\frac{1}{4}\eta_k\|\bar{\bsg}_{k}^0\|^2
+\frac{1}{2}\eta_k^2L_f\Big(\frac{16p(1+\sigma_0^2)L_f}{n}W_{4,k}
+\frac{8p(1+\sigma_0^2)}{n}L_f^2\|\bsx_{k}\|^2_{\bsK}\nonumber\\
&\quad+4p\sigma^2_1+8p(1+\sigma_0^2)\check{\sigma}^2_2+\frac{1}{2}p^2L_f^2\delta_k^2
+\|\bar{\bsg}^s_{k}\|^2\Big).\label{zerosg:v4kspeed-diminishing-thm5}
\end{align}

From $\eta_k=\frac{\kappa_2}{\beta_k}\le\frac{\kappa_2}{\kappa_0t_1^\theta}$, $t_1^\theta\ge \check{c}_3(\kappa_0,\kappa_1,\kappa_2)
\ge\max\{\frac{64p(1+\sigma_0^2)L_f\kappa_2\check{\varepsilon}_{10}}{\nu\kappa_0},
~\frac{p(1+\sigma_0^2)\kappa_2\check{\varepsilon}_{10}}{4\kappa_0}\}$, and $\check{\varepsilon}_{10}>(1+2\sqrt{10(8\kappa_4+5)}) L_f>23L_f$ due to $\kappa_4>1$, we have
\begin{subequations}
 \begin{align}
&\frac{8p(1+\sigma_0^2)L^2_f}{n}\eta_k^2
\le\frac{8p(1+\sigma_0^2)\kappa_2L^2_f}{\kappa_0t_1^\theta}\eta_k
<\frac{\nu}{8}\eta_k,\label{zerosg:v4kspeed-diminishing-1.1-thm5}\\
&\frac{4p(1+\sigma_0^2)L_f^3}{n}\eta_k^2
\le\frac{4p(1+\sigma_0^2)\kappa_2L_f}{\kappa_0t_1^\theta}\eta_kL_f^2
<\frac{16}{23}\eta_kL_f^2,\label{zerosg:v4kspeed-diminishing-1.2-thm5}\\
&\frac{1}{2}\eta_k^2L_f<\frac{2}{23}\eta_k,\label{zerosg:v4kspeed-diminishing-1.3-thm5}\\
&\frac{1}{4}p^2\eta_k^2L_f^3<pL_f^2\eta_k.\label{zerosg:v4kspeed-diminishing-1.4-thm5}
\end{align}
\end{subequations}

From \eqref{zerosg:v4kspeed-diminishing-thm5}--\eqref{zerosg:v4kspeed-diminishing-1.4-thm5}, we have \eqref{zerosg-a5:v4kspeed-diminishing-2}.
\end{proof}

Now we are ready to prove Theorem~\ref{zerosg-a5:thm-sg-diminishingt}

All conditions needed in Lemma~\ref{zerosg-a5:lemma:sg2} are satisfied, so \eqref{zerosg-a5:sgproof-vkLya2}--\eqref{zerosg-a5:v4kspeed-diminishing-2} still hold when $\theta=1$.

From \eqref{zerosg-a5:sgproof-vkLya2}--\eqref{zerosg-a5:v4kspeed-diminishing-2}, \eqref{nonconvex:gg3}, \eqref{zerosg:vkLya3.2-bound}, \eqref{zerosg:vkLya3}, \eqref{zerosg:vkLya3.1}, \eqref{zerosg-a5:step:eta1t1}, and $t_1\ge \check{c}_3(\kappa_0,\kappa_1,\kappa_2)\ge\max\{\frac{2}{3\varepsilon_4},~\frac{2}{3\varepsilon_6}\}$, similar to the way to get \eqref{zerosg:thm-sg-diminishing-equ2.1bounded} and \eqref{zerosg:thm-sg-diminishing-equ2bounded}, we have \eqref{zerosg-a5:thm-sg-diminishing-equ2.1bounded} and \eqref{zerosg-a5:thm-sg-diminishing-equ2bounded}.

\subsection{Proof of Theorem~\ref{zerosg:thm-random-pd-fixed}}\label{zerosg:proof-thm-random-pd-fixed}

In addition to the notations defined in Appendix~\ref{zerosg:proof-thm-sg-smT},
we also denote the following notations.
\begin{align*}
&\varepsilon=\frac{1}{2}+\frac{1}{2}\max\{1-\tilde{\varepsilon}_{16},~\tilde{\varepsilon}^2\},\\
&\tilde{\varepsilon}_{16}=\frac{1}{4\kappa_6}\min\{4\varepsilon_4,~8\varepsilon_6,
~\eta\nu\}.
\end{align*}

All conditions needed in Lemma~\ref{zerosg:lemma:sg2-T} are satisfied, so \eqref{zerosg:sgproof-vkLya2T} still holds.

(i) Taking expectation in $\calL_{T}$, summing \eqref{zerosg:sgproof-vkLya2T} over $ k\in[0,T-1]$, and using $\delta_{i,k}\in(0,\kappa_\delta\tilde{\varepsilon}^{k}]$ yield
\begin{align*}
\mathbf{E}[W_{T}]+\varepsilon_4\sum_{k=0}^{T-1}\mathbf{E}[\|\bsx_k\|^2_{\bsK}]
+\frac{1}{8}\eta\sum_{k=0}^{T-1}\mathbf{E}[\|\bar{\bsg}^0_{k}\|^2]
\le  W_{0}
+pn\tilde{\varepsilon}_{12}\eta^2T
+\frac{pn\tilde{\varepsilon}_{11}\kappa_\delta^2\eta}{1-\tilde{\varepsilon}^2},
\end{align*}
which further implies
\begin{align}\label{zerosg:sgproof-vkLya2T-fixed}
\sum_{k=0}^{T-1}\mathbf{E}[\|\bsx_k\|^2_{\bsK}]
\le \frac{1}{\varepsilon_4}\Big( W_{0}
+pn\tilde{\varepsilon}_{12}\eta^2T
+\frac{pn\tilde{\varepsilon}_{11}\kappa_\delta^2\eta}{1-\tilde{\varepsilon}^2}\Big).
\end{align}
Therefore, \eqref{zerosg:thm-sg-fixed-equ3.1} holds due to $\eta=\mathcal{O}(\frac{1}{p})$.

From \eqref{zerosg:v4kspeed}, we have
\begin{align}\label{zerosg:v4kspeed-fixed}
\mathbf{E}[W_{4,T}]
&\le  W_{4,0}+\sum_{k=0}^{T-1}\mathbf{E}[\|\bsx_k\|^2_{2\eta L_f^2\bsK}]-\frac{1}{8}\eta\sum_{k=0}^{T-1}\mathbf{E}[\|\bar{\bsg}_{k}^0\|^2]
+p\varepsilon_{15}\eta^2T
+\frac{(n+p)L_f^2\kappa_\delta^2\eta}{1-\tilde{\varepsilon}^2}.
\end{align}

From \eqref{zerosg:sgproof-vkLya2T-fixed} and \eqref{zerosg:v4kspeed-fixed}, we have
\begin{align*}
\sum_{k=0}^{T-1}\mathbf{E}[\|\bar{\bsg}_{k}^0\|^2]&\le
\frac{8W_{4,0}}{\eta}+\sum_{k=0}^{T-1}\mathbf{E}[\|\bsx_k\|^2_{16 L_f^2\bsK}]
+8p\varepsilon_{15}\eta T
+\frac{8(n+p)L_f^2\kappa_\delta^2}{1-\tilde{\varepsilon}^2}\\
&\le
\frac{8W_{4,0}}{\eta}+\frac{16 L_f^2}{\varepsilon_4}\Big(W_{0}
+pn\tilde{\varepsilon}_{12}\eta^2T
+\frac{pn\tilde{\varepsilon}_{11}\kappa_\delta^2\eta}{1-\tilde{\varepsilon}^2}\Big)
+8p\varepsilon_{15}\eta T
+\frac{8(n+p)L_f^2\kappa_\delta^2}{1-\tilde{\varepsilon}^2},
\end{align*}
which gives \eqref{zerosg:thm-sg-fixed-equ3}.

From \eqref{zerosg:rand-grad-smooth}, \eqref{zerosg:sgproof-vkLya2T-fixed}, and \eqref{zerosg:v4kspeed-fixed}, we have
\begin{align}\label{zerosg:sgproof-vkLya2T-fixed1}
&\mathbf{E}[\|\bar{\bsg}^0_k\|^2]
\le2L_f\mathbf{E}[W_{4,k}]\nonumber\\
&\le 2L_f\Big( W_{4,0}
+p\varepsilon_{15}\eta^2k
+\frac{(n+p)L_f^2\kappa_\delta^2\eta}{1-\tilde{\varepsilon}^2}\Big)
+\frac{4L_f^3\eta}{\varepsilon_4}\Big(W_0+pn\tilde{\varepsilon}_{12}\eta^2k
+\frac{pn\tilde{\varepsilon}_{11}\kappa_\delta^2\eta}{1-\tilde{\varepsilon}^2}\Big),~\forall k\in\mathbb{N}_0.
\end{align}

From \eqref{zerosg:sgproof-vkLya2T-bounded}, \eqref{zerosg:vkLya3.2-bound}, \eqref{zerosg:vkLya2-a1-bounded-thm2}, $\delta_{i,k}\in(0,\kappa_\delta\tilde{\varepsilon}^{k}]$, \eqref{zerosg:sgproof-vkLya2T-fixed1}, and , we have
\begin{align*}
\mathbf{E}[\breve{W}_{k+1}]&\le \mathbf{E}[ (1-\tilde{a}_1) \breve{W}_{k}
+p\tilde{\varepsilon}_{13}\eta^2\|\bar{\bsg}^0_{k}\|^2
+pn\tilde{\varepsilon}_{12}\eta^2
+pn\tilde{\varepsilon}_{11}\eta\kappa_\delta^2\tilde{\varepsilon}^{2k}]\nonumber\\
&\le (1-\tilde{a}_1)^k \breve{W}_{0}
+p\tilde{\varepsilon}_{13}\eta^2\sum_{\tau=0}^{k}(1-\tilde{a}_1)^\tau\mathbf{E}[\|\bar{\bsg}^0_{k-\tau}\|^2]\nonumber\\
&\quad+pn\tilde{\varepsilon}_{12}\eta^2\sum_{\tau=0}^{k}(1-\tilde{a}_1)^\tau
+pn\tilde{\varepsilon}_{11}\eta\kappa_\delta^2\sum_{\tau=0}^{k}(1-\tilde{a}_1)^\tau
\tilde{\varepsilon}^{2(k-\tau)}\nonumber\\
&=\mathcal{O}(pn\eta^2+\varepsilon_{15}p^2\eta^4(k+1)+\tilde{\varepsilon}_{12}p^2n\eta^5(k+1)),
\end{align*}
which gives \eqref{zerosg:thm-sg-fixed-equ3.2}.

(ii) If Assumption~\ref{zerosg:ass:fil} also holds, then \eqref{nonconvex:gg3} holds.
From  \eqref{zerosg:sgproof-vkLya2T}, \eqref{nonconvex:gg3},  and \eqref{zerosg:vkLya3.1}, for any $k\in\mathbb{N}_0$, we have
\begin{align}\label{zerosg:vkLya2-pl-fixed}
\mathbf{E}[W_{k+1}]
&\le \mathbf{E}\Big[W_{k}-\varepsilon_4\|\bsx_k\|^2_{\bsK}
-2\varepsilon_6\|\bm{v}_k+\frac{1}{\beta}\bsg_{k}^0\|^2_{\bsK}
-\frac{\eta\nu n}{4}(f(\bar{x}_k)-f^*)\Big]\nonumber\\
&\quad
+2pn(\sigma^2_1+2(1+\sigma_0^2)\sigma^2_2)\tilde{\varepsilon}_{10}\eta^2+pn\tilde{\varepsilon}_{11}\eta\delta_k^2\nonumber\\
&\le \mathbf{E}[W_{k}-\tilde{\varepsilon}_{16}W_{k}]+2pn(\sigma^2_1+2(1+\sigma_0^2)\sigma^2_2)\tilde{\varepsilon}_{10}\eta^2
+pn\tilde{\varepsilon}_{11}\eta\delta_k^2.
\end{align}

From \eqref{zerosg:vkLya2-a1-bounded}, we have
\begin{align}\label{zerosg:vkLya2-pl-r1.1-fixed}
0<\tilde{\varepsilon}_{16}\le\frac{2\varepsilon_6}{\kappa_6}\le\frac{1}{40}.
\end{align}

From \eqref{zerosg:vkLya2-pl-fixed}, \eqref{zerosg:vkLya3},  \eqref{zerosg:vkLya2-pl-r1.1-fixed}, and $\delta_{i,k}\in(0,\kappa_\delta\tilde{\varepsilon}^{k}]$, we have
\begin{align}\label{zerosg:vkLya2-pl-fixed2}
\mathbf{E}[W_{k+1}]
&\le (1-\tilde{\varepsilon}_{16})^{k+1}W_{0}
+2pn(\sigma^2_1+2(1+\sigma_0^2)\sigma^2_2)\tilde{\varepsilon}_{10}\eta^2
\sum_{\tau=0}^{k}(1-\tilde{\varepsilon}_{16})^\tau\nonumber\\
&\quad+pn\tilde{\varepsilon}_{11}\kappa_\delta^2\eta
\sum_{\tau=0}^{k}(1-\tilde{\varepsilon}_{16})^\tau
\tilde{\varepsilon}^{2(k-\tau)},~\forall k\in\mathbb{N}_0.
\end{align}

From \eqref{zerosg:vkLya3}, \eqref{zerosg:vkLya2-pl-fixed2}, \eqref{zerosg:lemma:sumgeo-equ}, and $\varepsilon>\max\{1-\tilde{\varepsilon}_{16},~\tilde{\varepsilon}^2\}$,  we have
\begin{align}\label{zerosg:vkLya2-pl-fixed3}
&\mathbf{E}[\|\bm{x}_k\|^2_{\bsK}+n(f(\bar{x}_k)-f^*)]
\le\frac{1}{\kappa_7}\mathbf{E}[W_{k}]\nonumber\\
&\le \frac{n}{\kappa_7}\Big(\frac{W_{0}}{n}
+\frac{p\tilde{\varepsilon}_{11}\kappa_\delta^2\eta}{\varepsilon-\tilde{\varepsilon}^2}\Big)\epsilon^{k}
+\frac{2n\tilde{\varepsilon}_{10}\eta}{\kappa_7\tilde{\varepsilon}_{16}}
(\sigma^2_1+2(1+\sigma_0^2)\sigma^2_2)p\eta,~\forall k\in\mathbb{N}_0,
\end{align}
which gives \eqref{zerosg:thm-sg-fixed-equ1}.



\subsection{Proof of Theorem~\ref{zerosg-p:thm-random-sm}}\label{zerosg-p:proof-thm-random-sm}
We denote the following notations.
\begin{align*}
&d_1=\frac{\rho_2(L)}{2\rho(L^2)},\\
&d_2(\gamma)=\min\Big\{\frac{4\epsilon_1}{9L_f^2},
~\frac{1}{64p(1+\sigma_0^2)(1+\tilde{\sigma}_0^2)(2\epsilon_2+L_f)}\Big\},\\
&\epsilon_{1}=\frac{1}{2}\gamma\rho_2(L)-\gamma^2\rho(L^2),\\
&\epsilon_{2}=\frac{1+2\gamma\rho_2(L)}{2\gamma\rho_2(L)},\\
&\epsilon_3=\Big(2\epsilon_{2}+\frac{L_f}{n}\Big)\epsilon_5,\\
&\epsilon_4=\frac{L_f^2}{4}\Big(\frac{1}{64(1+\sigma_0^2)(1+\tilde{\sigma}_0^2)}+\frac{4}{p}\Big),\\
&\epsilon_5=2(\sigma^2_1+2(1+\sigma_0^2)\sigma^2_2),\\
&\epsilon_6=\frac{ W_{1,0}+W_{4,0}}{n}
+\frac{p(\epsilon_3+\kappa^2_\delta \epsilon_4)\kappa^2_\eta}{2\theta-1},\\
&\epsilon_7=pn\kappa_\eta^2(32(1+\sigma_0^2)(1+\tilde{\sigma}_0^2)
L_f\epsilon_{2}\epsilon_6
+2\epsilon_2\epsilon_5+\epsilon_4\kappa_\delta^2).
\end{align*}

To prove Theorem~\ref{zerosg-p:thm-random-sm}, the following lemma is used.
\begin{lemma}\label{zerosg-p:lemma:sg2}
Suppose Assumptions~\ref{zerosg:ass:graph}--\ref{zerosg:ass:fig} hold. Suppose $\gamma\in(0,d_1)$ and $\eta_k\in(0,d_2(\gamma)]$. Let $\{\bsx_k\}$ be the sequence generated by Algorithm~\ref{zerosg-p:algorithm-random}, then
\begin{subequations}
\begin{align}
&\mathbf{E}_{\mathfrak{L}_k}[W_{1,k+1}+W_{4,k+1}]
\le W_{1,k}+W_{4,k}-\|\bsx_k\|^2_{\frac{\epsilon_{1}}{2}\bsK}
-\frac{1}{8}\eta_k\|\bar{\bsg}_{k}^0\|^2
+pn\epsilon_3\eta^2_k+pn\epsilon_4\eta_k\delta_k^2,\label{zerosg-p:w1w4k2}\\
&\mathbf{E}_{\mathfrak{L}_k}[W_{1,k+1}]
\le W_{1,k}-\|\bsx_k\|^2_{\frac{\epsilon_{1}}{2}\bsK}
+16p(1+\sigma_0^2)(1+\tilde{\sigma}_0^2)\epsilon_2\eta_k^2\|\bar{\bsg}_{k}^0\|^2
+2pn\epsilon_2\epsilon_5\eta^2_k
+pn\epsilon_4\eta_k\delta_k^2,
\label{zerosg-p:w1w4k1}\\
&\mathbf{E}_{\mathfrak{L}_k}[W_{4,k+1}]\le W_{4,k}+\|\bsx_k\|^2_{2L_f^2\eta_k\bsK}
-\frac{1}{8}\eta_k\|\bar{\bsg}_{k}^0\|^2+pL_f\epsilon_5\eta^2_k
+(p+n)L_f^2\eta_k\delta_k^2.\label{zerosg-p:w1w4k3}
\end{align}
\end{subequations}
\end{lemma}
\begin{proof} It is straightforward to see that for $\{\bsx_k\}$ generated by Algorithm~\ref{zerosg-p:algorithm-random}, Lemma~\ref{zerosg:lemma:grad-st} and \eqref{zerosg:v4k} still hold. Thus, \eqref{zerosg:v4kspeed-diminishing} still holds.

We have
\begin{align}
\mathbf{E}_{\mathfrak{L}_k}[W_{1,k+1}]
&=\mathbf{E}_{\mathfrak{L}_k}\Big[\frac{1}{2}\|\bm{x}_{k+1} \|^2_{\bsK}\Big]\nonumber\\
&=\mathbf{E}_{\mathfrak{L}_k}\Big[\frac{1}{2}\|\bm{x}_k-(\gamma\bsL\bm{x}_k+\eta_k\bsg^e_k) \|^2_{\bsK}\Big]\nonumber\\
&=\mathbf{E}_{\mathfrak{L}_k}\Big[\frac{1}{2}\|\bm{x}_k\|^2_{\bsK}-\gamma\|\bsx_k\|^2_{\bsL}
+\frac{1}{2}\gamma^2\|\bsx_k\|^2_{\bsL^2}
\nonumber\\
&\quad-\eta_k\bsx^\top_k({\bm I}_{np}-\gamma\bsL)\bsK\bsg^e_k
+\frac{1}{2}\eta^2_k\|\bsg^e_k\|^2_{\bsK}\Big]\nonumber\\
&\le\mathbf{E}_{\mathfrak{L}_k}\Big[\frac{1}{2}\|\bm{x}_k\|^2_{\bsK}-\gamma\|\bsx_k\|^2_{\bsL}
+\frac{1}{2}\gamma^2\|\bsx_k\|^2_{\bsL^2}\nonumber\\
&\quad+\frac{1}{2}\gamma\rho_2(L)\|\bm{x}_k\|^2_{\bsK}
+\frac{1}{2\gamma\rho_2(L)}\eta^2_k\|\bsg^e_k\|^2\nonumber\\
&\quad+\frac{1}{2}\gamma^2\|\bm{x}_k\|^2_{\bsL^2}
+\frac{1}{2}\eta^2_k\|\bsg^e_k\|^2
+\frac{1}{2}\eta^2_k\|\bsg^e_k\|^2\Big]\nonumber\\
&\le\mathbf{E}_{\mathfrak{L}_k}\Big[\frac{1}{2}\|\bm{x}_k\|^2_{\bsK}
-\gamma\|\bsx_k\|^2_{\rho_2(L)\bsK}
+\gamma^2\|\bsx_k\|^2_{\rho(L^2)\bsK}\nonumber\\
&\quad+\frac{1}{2}\gamma\rho_2(L)\|\bm{x}_k\|^2_{\bsK}
+\frac{1+2\gamma\rho_2(L)}{2\gamma\rho_2(L)}\eta^2_k
\|\bsg^e_k\|^2\Big]\nonumber\\
&=\frac{1}{2}\|\bm{x}_k\|^2_{\bsK}-\|\bsx_k\|^2_{\epsilon_{1}\bsK}
+\epsilon_{2}\eta^2_k\mathbf{E}_{\mathfrak{L}_k}[\|\bsg^e_k\|^2]\label{zerosg-p:v1k-1}\\
&\le\frac{1}{2}\|\bm{x}_k\|^2_{\bsK}-\|\bsx_k\|^2_{\epsilon_{1}\bsK}
+\epsilon_{2}\eta^2_k\Big(16p(1+\sigma_0^2)(1+\tilde{\sigma}_0^2)
(\|\bar{\bsg}_{k}^0\|^2+L_f^2\|\bsx_{k}\|^2_{\bsK})\nonumber\\
&\quad+4np\sigma^2_1
+8np(1+\sigma_0^2)\sigma^2_2+\frac{np^2L_f^2}{2}\delta_k^2\Big)\nonumber\\
&=\frac{1}{2}\|\bm{x}_k\|^2_{\bsK}-\|\bsx_k\|^2_{\epsilon_{1}\bsK
-16p(1+\sigma_0^2)(1+\tilde{\sigma}_0^2)L_f^2\epsilon_{2}\eta^2_k\bsK}
+\epsilon_{2}\eta^2_k\Big(16p(1+\sigma_0^2)(1+\tilde{\sigma}_0^2)
\|\bar{\bsg}_{k}^0\|^2\nonumber\\
&\quad+4np\sigma^2_1
+8np(1+\sigma_0^2)\sigma^2_2+\frac{np^2L_f^2}{2}\delta_k^2\Big),
\label{zerosg-p:v1k}
\end{align}
where the second equality holds due to \eqref{zerosg-p:alg:random}; the third equality holds due to \eqref{nonconvex:KL-L-eq}; the first inequality holds due to the Cauchy--Schwarz inequality and $\rho(\bsK)=1$; the second  inequality holds due to \eqref{nonconvex:KL-L-eq2}; the second last  equality holds since that $x_{i,k}$ is independent of $\mathfrak{L}_k$; and the last  inequality holds due to \eqref{zerosg:rand-grad-esti2}.

From \eqref{zerosg:v4kspeed-diminishing} and \eqref{zerosg-p:v1k}, we have
\begin{align}\label{zerosg-p:w1w4k}
\mathbf{E}_{\mathfrak{L}_k}[W_{1,k+1}+W_{4,k+1}]
&\le W_{1,k}+W_{4,k}-\|\bsx_k\|^2_{\epsilon_{1}\bsK
-(L_f^2\eta_k+16p(1+\sigma_0^2)(1+\tilde{\sigma}_0^2)L_f^2\epsilon_{2}\eta^2_k
+\frac{8p(1+\sigma_0^2)(1+\tilde{\sigma}_0^2)L_f^3}{n}\eta^2_k)\bsK}\nonumber\\
&\quad-\frac{1}{4}\Big(1-64p(1+\sigma_0^2)(1+\tilde{\sigma}_0^2)\epsilon_{2}\eta_k
-\frac{32p(1+\sigma_0^2)(1+\tilde{\sigma}_0^2)L_f}{n}\eta_k\Big)
\eta_k\|\bar{\bsg}_{k}^0\|^2\nonumber\\
&\quad-\frac{1}{4}(1-2L_f\eta_k)\eta_k\|\bar{\bsg}^s_{k}\|^2
+2pn\Big(2\epsilon_{2}+\frac{L_f}{n}\Big)(\sigma^2_1
+2(1+\sigma_0^2)\sigma^2_2)\eta^2_k\nonumber\\
&\quad
+\frac{pnL_f^2}{4}\Big(2p\epsilon_{2}\eta_k+\frac{pL_f}{n}\eta_k+\frac{4}{p}\Big)\eta_k\delta_k^2.
\end{align}

From $\gamma\in(0,d_1)$ and $\rho_2(L)\le\rho(L)$, we have
\begin{align}\label{zerosg-p:epsilon1}
0<\epsilon_{1}<\frac{1}{16}.
\end{align}

From $\eta_k\le d_2(\gamma)\le \frac{1}{64p(1+\sigma_0^2)(1+\tilde{\sigma}_0^2)(2\epsilon_2+L_f)}$, we have
\begin{subequations}
\begin{align}
&32p(1+\sigma_0^2)(1+\tilde{\sigma}_0^2)\Big(2\epsilon_{2}
+\frac{L_f}{n}\Big)\eta_k
\le32p(1+\sigma_0^2)(1+\tilde{\sigma}_0^2)(2\epsilon_2+L_f)d_2(\gamma)\le\frac{1}{2},\label{zerosg-p:epsilon2}\\
&2L_f\eta_k\le\frac{2L_f}{64p(1+\sigma_0^2)(1+\tilde{\sigma}_0^2)(2\epsilon_2+L_f)}
<\frac{1}{32p}<1,\label{zerosg-p:epsilon2.1}\\
&\frac{L_f^2}{4}\Big(2p\epsilon_{2}\eta_k+\frac{pL_f}{n}\eta_k+\frac{4}{p}\Big)
\le\epsilon_4.\label{zerosg-p:epsilon2.2}
\end{align}
\end{subequations}

From $\eta_k\le d_2(\gamma)\le \frac{4\epsilon_1}{9L_f^2}$ and \eqref{zerosg-p:epsilon2}, we have
\begin{align}\label{zerosg-p:epsilon3}
&L_f^2\eta_k+16p(1+\sigma_0^2)(1+\tilde{\sigma}_0^2)L_f^2\epsilon_{2}\eta^2_k
+\frac{8p(1+\sigma_0^2)(1+\tilde{\sigma}_0^2)L_f^3}{n}\eta^2_k\nonumber\\
&\le(1+8p(1+\sigma_0^2)(1+\tilde{\sigma}_0^2)(2\epsilon_2+L_f)d_2(\gamma))L_f^2d_2(\gamma)
\le\frac{9L_f^2d_2(\gamma)}{8}\le\frac{1}{2}\epsilon_1.
\end{align}

From \eqref{zerosg-p:w1w4k}--\eqref{zerosg-p:epsilon3}, we have \eqref{zerosg-p:w1w4k2}.

Similarly, we get \eqref{zerosg-p:w1w4k1} and \eqref{zerosg-p:w1w4k3}.
\end{proof}

Now it is ready to prove Theorem~\ref{zerosg-p:thm-random-sm}. The proof is similar to the proof of Theorem~\ref{zerosg:thm-random-pd-sm}. For the sake of completeness, the proof is included here.

From $\kappa_\eta\in(0,d_2(\gamma)t_1^\theta]$ and $\eta_k=\frac{\kappa_\eta}{(k+t_1)^\theta}$, we have $\eta_k\le d_2(\gamma)$. Thus, all conditions needed in Lemma~\ref{zerosg-p:lemma:sg2} are satisfied. So \eqref{zerosg-p:w1w4k2}--\eqref{zerosg-p:w1w4k1} hold.

Taking expectation in $\calL_{T}$, summing \eqref{zerosg-p:w1w4k2} over $ k\in[0,T-1]$, noting $\eta_k=\frac{\kappa_\eta}{(k+t_1)^\theta}$, $\theta\in(0.5,1)$, and $\delta_k\le \frac{\kappa_\delta\sqrt{p\eta_k}}{\sqrt{n+p}}$ as stated in \eqref{zerosg-p:thm-random-sm-akbk}, and
using \eqref{zerosg:serise:lemma:sum-equ} yield
\begin{align}\label{zerosg-p:vkLya4.1}
&\mathbf{E}[W_{1,T}+W_{4,T}]
+\sum_{k=0}^{T-1}\mathbf{E}\Big[\frac{\epsilon_1}{2}\|\bsx_k\|^2_{\bsK}
+\frac{1}{8}\eta_k\|\bar{\bsg}^0_{k}\|^2\Big]\nonumber\\
&\le W_{1,0}+W_{4,0}+pn(\epsilon_3+\kappa^2_\delta \epsilon_4)\kappa^2_\eta\sum_{k=0}^{T-1}\frac{1}{(k+t_1)^{2\theta}}
\le n\epsilon_6.
\end{align}

Noting that $t_1^\theta=\mathcal{O}(\sqrt{p})$, we have
\begin{align}\label{zerosg-p:k0}
\kappa_\eta=\mathcal{O}\Big(\frac{t_1^\theta}{p}\Big)=\mathcal{O}\Big(\frac{1}{\sqrt{p}}\Big).
\end{align}

From $W_{1,0}+W_{4,0}=\mathcal{O}(n)$ and \eqref{zerosg-p:k0}, we have
\begin{align}\label{zerosg-p:epsilon5}
\epsilon_6=\frac{ W_{1,0}+W_{4,0}}{n}
+\frac{p(\epsilon_3+\kappa^2_\delta \epsilon_4)\kappa^2_\eta}{2\theta-1}=\mathcal{O}(1).
\end{align}

From \eqref{zerosg-p:vkLya4.1} and \eqref{zerosg-p:epsilon1}, we have
\begin{align}\label{zerosg-p:thm-sg-sm-equ2p}
\mathbf{E}[f(\bar{x}_{T})]-f^*\le\frac{1}{n}W_{4,T}
\le \epsilon_6.
\end{align}

From \eqref{zerosg-p:thm-sg-sm-equ2p} and \eqref{zerosg-p:epsilon5}, we have \eqref{zerosg-p:thm-sg-sm-equ2}.

From \eqref{zerosg-p:vkLya4.1} and \eqref{zerosg-p:epsilon1}, we have
\begin{align}\label{zerosg-p:thm-sg-sm-equ1.1p}
\sum_{k=0}^{T-1}\mathbf{E}[\|\bsx_k\|^2_{\bsK}]
\le\frac{2n\epsilon_6}{\epsilon_1}.
\end{align}

From \eqref{zerosg:rand-grad-smooth} and \eqref{zerosg-p:thm-sg-sm-equ2p}, we have
\begin{align}\label{zerosg-p:thm-sg-sm-bounded}
\|\bar{\bsg}^0_k\|^2\le 2nL_f\epsilon_6.
\end{align}

From \eqref{zerosg:rand-grad-esti2}, \eqref{zerosg-p:thm-sg-sm-equ1.1p}, and \eqref{zerosg-p:thm-sg-sm-bounded}, we know that $\mathbf{E}[\|\bsg^e_k\|^2]$ is bounded. Then, same as the proof of the first part of Theorem~1 in \cite{tang2020distributedzero}, we have  \eqref{zerosg-p:thm-sg-sm-equ1bounded}.

From \eqref{zerosg-p:w1w4k1}, \eqref{zerosg-p:thm-sg-sm-bounded}, and \eqref{zerosg-p:thm-random-sm-akbk}, we have
\begin{align}\label{zerosg-p:vkLya4-bound}
\mathbf{E}[W_{1,k+1}]\le(1-\epsilon_{1})\mathbf{E}[W_{1,k}]+\frac{\epsilon_{7}}{(t+t_1)^{2\theta}}.
\end{align}

From \eqref{zerosg-p:vkLya4-bound}, \eqref{zerosg-p:epsilon1}, and \eqref{zerosg:serise:lemma:sequence-equ6}, we have
\begin{align}\label{zerosg-p:lemma:sequence-equ6-bounded}
\mathbf{E}[W_{1,k}]&\le \phi_3(k,t_1,\epsilon_{1},\epsilon_{7},2\theta,W_{1,0}),~\forall k\in\mathbb{N}_+,
\end{align}
where the function $\phi_3$ is defined in \eqref{zerosg:serise:lemma:sequence-equ6-phi4}.

Noting that $\phi_3(k,t_1,\epsilon_{1},\epsilon_{7},2\theta,W_{1,0})=\mathcal{O}(\frac{n}{k^{2\theta}})$ due to \eqref{zerosg-p:k0}, from \eqref{zerosg-p:lemma:sequence-equ6-bounded}, we have \eqref{zerosg-p:thm-sg-sm-equ1.1bounded}.

From \eqref{zerosg-p:w1w4k3}, we have
\begin{align}\label{zerosg-p:w1w4k3-thm7}
\Big(\frac{1}{\eta_k}-\frac{1}{\eta_{k+1}}+\frac{1}{\eta_{k+1}}\Big)\mathbf{E}_{\mathfrak{L}_k}[W_{4,k+1}]
\le \frac{W_{4,k}}{\eta_k}+\|\bsx_k\|^2_{2L_f^2\bsK}
-\frac{1}{8}\|\bar{\bsg}_{k}^0\|^2+pL_f\epsilon_5\eta_k
+(p+n)L_f^2\delta_k^2.
\end{align}
Then, taking expectation in $\calL_{T}$, summing \eqref{zerosg-p:w1w4k3-thm7} over $ k\in[0,T-1]$,  noting \eqref{zerosg-p:thm-sg-sm-equ2p}, $\eta_k=\frac{\kappa_\eta}{(k+t_1)^\theta}$, $\theta\in(0.5,1)$, and $\delta_k\le \frac{\kappa_\delta\sqrt{p\eta_k}}{\sqrt{n+p}}$ as stated in \eqref{zerosg-p:thm-random-sm-akbk}, and
using \eqref{zerosg:serise:lemma:sum-equ} yield
\begin{align}\label{zerosg-p:w1w4k3-thm7-1}
&\frac{1}{8}\sum_{k=0}^{T-1}\mathbf{E}[\|\bar{\bsg}_{k}^0\|^2]\nonumber\\
&\le \frac{W_{4,0}}{\eta_0}
+\sum_{k=0}^{T-1}\Big(\frac{1}{\eta_{k+1}}-\frac{1}{\eta_{k}}\Big)\mathbf{E}[W_{4,k+1}]
+\sum_{k=0}^{T-1}\mathbf{E}[\|\bsx_k\|^2_{2L_f^2\bsK}]
+\frac{pL_f(\epsilon_5+L_f\kappa_\delta^2)\kappa_\eta(T+t_1)^{1-\theta}}{1-\theta}\nonumber\\
&\le \frac{n\epsilon_6}{\eta_0}
+\sum_{k=0}^{T-1}\Big(\frac{1}{\eta_{k+1}}-\frac{1}{\eta_{k}}\Big)n\epsilon_6
+\sum_{k=0}^{T-1}\mathbf{E}[\|\bsx_k\|^2_{2L_f^2\bsK}]
+\frac{pL_f(\epsilon_5+L_f\kappa_\delta^2)\kappa_\eta(T+t_1)^{1-\theta}}{1-\theta}\nonumber\\
&= \frac{n\epsilon_6(T+t_1)^\theta}{\kappa_\eta}
+\sum_{k=0}^{T-1}\mathbf{E}[\|\bsx_k\|^2_{2L_f^2\bsK}]
+\frac{pL_f(\epsilon_5+L_f\kappa_\delta^2)\kappa_\eta(T+t_1)^{1-\theta}}{1-\theta}.
\end{align}
From \eqref{zerosg-p:w1w4k3-thm7-1}, \eqref{zerosg-p:lemma:sequence-equ6-bounded}, \eqref{zerosg-p:k0}, and $\theta\in(0.5,1)$, we have \eqref{zerosg-p:thm-sg-sm-equ1}.

\subsection{Proof of Theorem~\ref{zerosg-p:thm-sg-smT}}\label{zerosg-p:proof-thm-sg-smT}
We use the notations defined in Appendix~\ref{zerosg-p:proof-thm-random-sm}.

From $\eta_k=\eta=\frac{\sqrt{n}}{\sqrt{pT}}$ and $T\ge \frac{n}{pd_2^2(\gamma)}$, we have $\eta_k\le d_2(\gamma)$. Thus, all conditions needed in Lemma~\ref{zerosg-p:lemma:sg2} are satisfied. So \eqref{zerosg-p:w1w4k2}--\eqref{zerosg-p:w1w4k3} hold.

Taking expectation in $\calL_{T}$, summing \eqref{zerosg-p:w1w4k2} over $ k\in[0,T-1]$, noting $\eta_k=\eta=\frac{\sqrt{n}}{\sqrt{pT}}$ and $\delta_{i,k}\le\frac{p^{\frac{1}{4}}n^{\frac{1}{4}}\kappa_\delta}
{\sqrt{n+p}(k+1)^{\frac{1}{4}}}$ as stated in \eqref{zerosg-p:step:eta2-sm}, and
using \eqref{zerosg:serise:lemma:sum-equ} yield
\begin{align}
&\frac{1}{nT}\sum_{k=0}^{T-1}\mathbf{E}[\|\bsx_k\|^2_{\bsK}]
\le\frac{2}{\epsilon_1}\Big(\frac{W_{1,0}+W_{4,0}}{nT}
+\frac{n\epsilon_{3}}{T}
+\frac{2n\kappa_\delta^2\epsilon_{4}}{T}\Big).
\label{zerosg-p:thm-sg-sm-equ3.1p}
\end{align}
Similarly, from \eqref{zerosg-p:w1w4k3} and \eqref{zerosg-p:step:eta2-sm}, we have
\begin{align}\label{zerosg-p:thm-sg-sm-equ3p}
\frac{1}{T}\sum_{k=0}^{T-1}\mathbf{E}[\|\nabla f(\bar{x}_k)\|^2]&=\frac{1}{nT}\sum_{k=0}^{T}\mathbf{E}[\|\bar{\bsg}_{k}^0\|^2]
\le 8\Big(\frac{W_{4,0}}{nT\eta}
+\frac{2L_f^2}{nT}\sum_{k=0}^{T}\mathbf{E}[\|\bsx_k\|^2_{\bsK}]
+\frac{pL_f\epsilon_5\eta}{n}+\frac{2\sqrt{p}L_f^2\kappa_\delta^2}{\sqrt{nT}}\Big).
\end{align}
Noting that $\eta=\frac{\sqrt{n}}{\sqrt{pT}}$, from \eqref{zerosg-p:thm-sg-sm-equ3.1p} and \eqref{zerosg-p:thm-sg-sm-equ3p}, we have
\begin{align*}
\frac{1}{T}\sum_{k=0}^{T-1}\mathbf{E}[\|\nabla f(\bar{x}_k)\|^2]
&=8(f(\bar{x}_0)-f^*+\epsilon_5L_f
+2L_f^2\kappa_\delta^2)\frac{\sqrt{p}}{\sqrt{nT}}
+\mathcal{O}(\frac{n}{T}),
\end{align*}
which gives \eqref{zerosg-p:thm-sg-sm-equ3}.

Taking expectation in $\calL_{T}$, summing \eqref{zerosg-p:w1w4k3} over $ k\in[0,T-1]$, and using \eqref{zerosg-p:step:eta2-sm}  yield
\begin{align}\label{zerosg-p:thm-sg-sm-equ4p}
n(\mathbf{E}[f(\bar{x}_{T})]-f^*)&=\mathbf{E}[W_{4,T}]
\le W_{4,0}+\frac{2L_f^2\sqrt{n}}{\sqrt{pT}}\sum_{k=0}^{T-1}\|\bsx_k\|^2_{\bsK}+nL_f\epsilon_5
+2nL_f^2\kappa_\delta^2.
\end{align}

Noting that $W_{4,0}=\mathcal{O}(n)$ and $\frac{\sqrt{n}n}{\sqrt{pT}}\le1$ due to $T\ge \frac{n^3}{p}$, from \eqref{zerosg-p:thm-sg-sm-equ3.1p} and \eqref{zerosg-p:thm-sg-sm-equ4p}, we have \eqref{zerosg-p:thm-sg-sm-equ4}.
Then, from \eqref{zerosg:rand-grad-smooth}, we know that there exists a constant $\tilde{d}_g>0$, such that
\begin{align}\label{zerosg-p:coro-sg-sm-gbark0}
\mathbf{E}[\|\bar{\bsg}^0_k\|^2]\le n\tilde{d}_g,~\forall k\in\mathbb{N}_0.
\end{align}

From \eqref{zerosg-p:w1w4k1}, \eqref{zerosg-p:coro-sg-sm-gbark0}, and \eqref{zerosg-p:step:eta2-sm}, we have
\begin{align}\label{zerosg-p:vkLya4-bound-tilde}
\mathbf{E}[W_{1,k+1}]\le(1-\epsilon_{1})\mathbf{E}[W_{1,k}]
+\frac{n^2\epsilon_4\kappa_\delta^2}{\sqrt{T(k+1)}}
+\frac{n^2(16(1+\sigma_0^2)(1+\tilde{\sigma}_0^2)\epsilon_2\tilde{d}_g
+2\epsilon_2\epsilon_5)}{T},~\forall 0\le k\le T.
\end{align}

From \eqref{zerosg-p:vkLya4-bound-tilde} and \eqref{zerosg-p:epsilon1}, similar to the way to get \eqref{zerosg:serise:lemma:sequence-equ6}, we have $\mathbf{E}[W_{1,T}]=\mathcal{O}(\frac{n^2}{T})$, which gives \eqref{zerosg-p:thm-sg-sm-equ3.1}.

Similar to the proof of \eqref{zerosg-p:thm-sg-sm-equ1bounded}, we have \eqref{zerosg-p:coro-sg-sm-equ3bounded}.

\subsection{Proof of Theorem~\ref{zerosg-p:thm-sg-diminishing}}\label{zerosg-p:proof-thm-sg-diminishing}
In addition to the notations defined in Appendix~\ref{zerosg-p:proof-thm-random-sm}, we also denote the following notations.
\begin{align*}
&\breve{\epsilon}_7=pn\kappa_\eta^2(16(1+\sigma_0^2)(1+\tilde{\sigma}_0^2)
\epsilon_{2}d_g+2\epsilon_2\epsilon_5+\epsilon_4\kappa_\delta^2),\\
&\epsilon_{8}=\min\Big\{\frac{\epsilon_1t_1^\theta}{\kappa_\eta},~\frac{\nu}{4}\Big\},\\
&b_1=\epsilon_{8}\kappa_\eta,\\
&b_2=pn(\epsilon_{3}+\epsilon_{4}\kappa_\delta^2)\kappa_\eta^2.
\end{align*}

All conditions needed in Lemma~\ref{zerosg-p:lemma:sg2} are satisfied, so \eqref{zerosg-p:w1w4k2}--\eqref{zerosg-p:w1w4k3} hold.

Denote $\check{W}_{k}=W_{1,k}+W_{4,k}$. From \eqref{zerosg-p:w1w4k2} and \eqref{nonconvex:gg3},  we have
\begin{align}\label{zerosg-p:vkLya2-pl}
\mathbf{E}_{\mathfrak{L}_k}[\check{W}_{k+1}]
&\le \check{W}_{k}-\|\bsx_k\|^2_{\frac{\epsilon_{1}}{2}\bsK}
-\frac{\nu}{4}\eta_kW_{4,k}+pn\epsilon_3\eta^2_k+pn\epsilon_4\eta_k\delta_k^2\nonumber\\
&\le \Big(1-\eta_k\min\Big\{\frac{\epsilon_1}{\eta_k},~\frac{\nu}{4}\Big\}\Big)
\check{W}_{k}+pn\epsilon_3\eta^2_k+pn\epsilon_4\eta_k\delta_k^2\nonumber\\
&\le (1-\eta_k\epsilon_{8})\check{W}_{k}+pn\epsilon_3\eta^2_k+pn\epsilon_4\eta_k\delta_k^2
,~\forall k\in\mathbb{N}_0.
\end{align}

Denote $\check{z}_k=\mathbf{E}[\check{W}_k]$, $s_{1,k}=\eta_k\epsilon_{8}$, and $s_{2,k}=pn\epsilon_3\eta^2_k+pn\epsilon_4\eta_k\delta_k^2$. From \eqref{zerosg-p:vkLya2-pl}, we have
\begin{align}
\check{z}_{k+1}\le (1-s_{1,k})\check{z}_k+s_{2,k},~\forall k\in\mathbb{N}_0.
\label{zerosg-p:vkLya2-pl-z}
\end{align}

From \eqref{zerosg-p:step:eta1}, we have
\begin{align}
s_{1,k}&=\eta_k\epsilon_{8}
=\frac{b_1}{(k+t_1)^\theta},\label{zerosg-p:vkLya2-pl-r1}\\
s_{2,k}&=pn\epsilon_3\eta^2_k+pn\epsilon_4\eta_k\delta_k^2
\le\frac{b_2}{(k+t_1)^{2\theta}}.\label{zerosg-p:vkLya2-pl-r2}
\end{align}

From \eqref{zerosg-p:epsilon1}, we have
\begin{align}\label{zerosg-p:vkLya2-pl-r1.1}
0<s_{1,k}\le\epsilon_1\le\frac{1}{16}.
\end{align}

Then, from $\theta\in(0,1)$, \eqref{zerosg-p:vkLya2-pl-z}--\eqref{zerosg-p:vkLya2-pl-r1.1}, and \eqref{zerosg:serise:lemma:sequence-equ4}, we have
\begin{align}\label{zerosg-p:vkLya2-pl-theta0}
\check{z}_{k}\le\phi_1(k,t_1,b_1,b_2,\theta,2\theta,\check{z}_0),~\forall k\in\mathbb{N}_+,
\end{align}
where the function $\phi_1$ is defined in \eqref{zerosg:serise:lemma:sequence-equ4-phi2}.

From $\kappa_\eta\in(0,d_2(\gamma)t_1^\theta]$, $\theta\in(0,1)$, and $t_1\in[ p^{\frac{1}{\theta}},d_3p^{\frac{1}{\theta}}]$, we have
\begin{align}\label{zerosg-p:k0-pl}
0<\kappa_\eta\le d_3pd_2(\gamma)<\frac{d_3}{64(1+\sigma_0^2)(1+\tilde{\sigma}_0^2)L_f}.
\end{align}

From \eqref{zerosg:rand-grad-smooth}, \eqref{zerosg-p:vkLya2-pl-theta0}, and \eqref{zerosg-p:k0-pl}, we get
\begin{align}\label{zerosg-p:gbark0-pl-1}
\mathbf{E}[\|\bar{\bsg}^0_k\|^2]
=\mathcal{O}\Big(\frac{pn}{(k+t_1)^\theta}\Big),~\forall k\in\mathbb{N}_+.
\end{align}
From \eqref{zerosg-p:gbark0-pl-1} and $t_1^\theta=\mathcal{O}(p)$, we know that there exists a constant $\breve{d}_g>0$, such that
\begin{align}\label{zerosg-p:gbark0-pl}
\mathbf{E}[\|\bar{\bsg}^0_k\|^2]\le n\breve{d}_g,~\forall k\in\mathbb{N}_0.
\end{align}

From \eqref{zerosg-p:w1w4k1}, \eqref{zerosg-p:gbark0-pl}, and \eqref{zerosg-p:step:eta1}, we have
\begin{align}\label{zerosg-p:vkLya4-bound-pl}
\mathbf{E}[W_{1,k+1}]\le(1-\epsilon_{1})\mathbf{E}[W_{1,k}]
+\frac{\breve{\epsilon}_{7}}{(t+t_1)^{2\theta}}.
\end{align}

Using \eqref{zerosg:serise:lemma:sequence-equ6}, from \eqref{zerosg-p:epsilon1} and \eqref{zerosg-p:vkLya4-bound-pl}, we have
\begin{align}\label{zerosg-p:lemma:sequence-equ6-bounded-pl}
\mathbf{E}[W_{1,k}]&\le \phi_3(k,t_1,\epsilon_{1},\breve{\epsilon}_{7},2\theta,W_{0,k}),~\forall k\in\mathbb{N}_+,
\end{align}
where the function $\phi_3$ is defined in \eqref{zerosg:serise:lemma:sequence-equ6-phi4}.
From \eqref{zerosg-p:lemma:sequence-equ6-bounded-pl}, \eqref{zerosg:serise:lemma:sequence-equ6-phi4}, and \eqref{zerosg-p:k0-pl}, we have
\begin{align}\label{zerosg-p:vkLya4-bound-brevez}
&\mathbf{E}[\|\bsx_k\|^2_{\bsK}]\le2\mathbf{E}[W_{1,k}]
\le2\phi_3(k,t_1,\epsilon_{1},\breve{\epsilon}_{7},2\theta,W_{0,k})
=\mathcal{O}\Big(\frac{pn}{(k+t_1)^{2\theta}}\Big),
\end{align}
which yields \eqref{zerosg-p:thm-sg-diminishing-equ1.1bounded}.

From \eqref{zerosg-p:w1w4k3}, \eqref{nonconvex:gg3}, and $\delta_k\le \frac{\kappa_\delta\sqrt{p\eta_k}}{\sqrt{n+p}}$ we have
\begin{align}
\mathbf{E}[W_{4,k+1}]
&\le \mathbf{E}[W_{4,k}]-\frac{\nu}{4}\eta_k\mathbf{E}[W_{4,k}]
+\|\bsx_k\|^2_{2L_f^2\eta_k\bsK}
+pL_f\epsilon_5\eta^2_k+pL_f^2\kappa_\delta^2\eta_k^2.
\label{zerosg-p:v4kspeed-diminishing-3}
\end{align}

Similar to the way to prove \eqref{zerosg:serise:lemma:sequence-equ4}, from \eqref{zerosg-p:vkLya4-bound-brevez} and \eqref{zerosg-p:v4kspeed-diminishing-3}, we have \begin{align}\label{zerosg-p:v4kspeed-diminishing-4}
\mathbf{E}[f(\bar{x}_{T})-f^*]
\le\frac{\epsilon_9p}{n(T+t_1)^\theta}+\mathcal{O}\Big(\frac{p}{(T+t_1)^{2\theta}}\Big).
\end{align}
From \eqref{zerosg-p:k0-pl}, we have
\begin{align}\label{zerosg-p:v4kspeed-diminishing-5}
\epsilon_9=\frac{8\theta4^\theta\kappa_\eta(\epsilon_5+L_f^2\kappa_\delta^2)}{\nu}
<\frac{8\theta4^\theta(\epsilon_5+L_f^2\kappa_\delta^2)d_3}{64\nu(1+\sigma_0^2)(1+\tilde{\sigma}_0^2)L_f}.
\end{align}
Thus, from \eqref{zerosg-p:v4kspeed-diminishing-4} and \eqref{zerosg-p:v4kspeed-diminishing-5}, we have\eqref{zerosg-p:thm-sg-diminishing-equ1bounded}.

\subsection{Proof of Theorem~\ref{zerosg-p:thm-sg-diminishingt}}\label{zerosg-p:proof-thm-sg-diminishingt}
In addition to the notations defined in Appendices~\ref{zerosg-p:proof-thm-random-sm} and \ref{zerosg-p:proof-thm-sg-diminishing},
we also denote
\begin{align*}
&\hat{d}_2(\gamma)=\max\Big\{\frac{1}{\epsilon_1},~\frac{\kappa_\eta}{d_2(\gamma)}\Big\}.
\end{align*}

From $t_1>\hat{d}_2(\gamma)\ge\frac{\kappa_\eta}{d_2(\gamma)}$, we have
\begin{align*}
\eta_k=\frac{\kappa_\eta}{k+t_1}\le\frac{\kappa_\eta}{t_1}<d_2(\gamma).
\end{align*}
Thus, all conditions needed in Lemma~\ref{zerosg-p:lemma:sg2} are satisfied, so \eqref{zerosg-p:vkLya2-pl-z}--\eqref{zerosg-p:vkLya2-pl-r1.1} still hold when $\theta=1$.

From $t_1>\hat{d}_2(\gamma)\ge\frac{1}{\epsilon_1}$, we have
\begin{align}\label{zerosg-p:vkLya2-pl-a1-1}
\epsilon_1t_1>1.
\end{align}

From $\kappa_\eta\in(\frac{8}{\nu},\frac{8d_3}{\nu}]$, we have
\begin{align}\label{zerosg-p:vkLya2-pl-a1-3}
2d_3\ge\frac{\nu\kappa_\eta}{4}>2.
\end{align}

Hence, from \eqref{zerosg-p:vkLya2-pl-a1-1} and \eqref{zerosg-p:vkLya2-pl-a1-3}, we have
\begin{align}\label{zerosg-p:vkLya2-pl-a1}
b_1=\epsilon_8\kappa_\eta>1.
\end{align}

Then from $\theta=1$, \eqref{zerosg-p:vkLya2-pl-z}--\eqref{zerosg-p:vkLya2-pl-r1.1}, \eqref{zerosg-p:vkLya2-pl-a1}, and \eqref{zerosg:serise:lemma:sequence-equ5}, we have
\begin{align}\label{zerosg-p:vkLya2-pl-theta0t}
\check{z}_{k}\le\phi_2(k,t_1,b_1,b_2,2,\check{z}_0),~\forall k\in\mathbb{N}_+,
\end{align}
where the function $\phi_2$ is defined in \eqref{zerosg:serise:lemma:sequence-equ5-phi3}.

From \eqref{zerosg-p:vkLya2-pl-a1-3} and \eqref{zerosg-p:vkLya2-pl-a1}, we know  $\phi_2(k,t_1,b_1,b_2,2,\check{z}_0)=\mathcal{O}(n+\frac{pn}{k+t_1})$. Hence, from \eqref{zerosg:rand-grad-smooth} and \eqref{zerosg-p:vkLya2-pl-theta0t}, we get
\begin{align}\label{zerosg-p:gbark0-pl-speed-1}
\mathbf{E}[\|\bar{\bsg}^0_k\|^2]
=\mathcal{O}\Big(n+\frac{pn}{k+t_1}\Big),~\forall k\in\mathbb{N}_+.
\end{align}

Noting that $t_1>\hat{d}_2(\gamma)\ge\frac{\kappa_\eta}{d_2(\gamma)}\ge64p\kappa_\eta(1+\sigma_0^2)
(1+\tilde{\sigma}_0^2)(2\epsilon_2+L_f)$, from \eqref{zerosg-p:gbark0-pl-speed-1} and \eqref{zerosg-p:vkLya2-pl-a1-3}, we know that there exists a constant $\hat{d}_g>0$, such that
\begin{align}\label{zerosg-p:gbark0-pl-speed}
\mathbf{E}[\|\bar{\bsg}^0_k\|^2]\le n\hat{d}_g,~\forall k\in\mathbb{N}_0.
\end{align}
Then, similar to the way to get \eqref{zerosg-p:thm-sg-diminishing-equ1.1bounded} and \eqref{zerosg-p:thm-sg-diminishing-equ1bounded}, we get \eqref{zerosg-p:thm-sg-diminishing-equ2.1bounded} and \eqref{zerosg-p:thm-sg-diminishing-equ2bounded}.

\subsection{Proof of Theorem~\ref{zerosg-p-a5:thm-sg-diminishingt}}\label{zerosg-p-a5:proof-thm-sg-diminishingt}
In addition to the notations defined in Appendices~\ref{zerosg-p:proof-thm-random-sm}, \ref{zerosg-p:proof-thm-sg-diminishing}, and \ref{zerosg-p:proof-thm-sg-diminishingt},
we also denote
\begin{align*}
&\tilde{d}_2(\gamma)=\min\Big\{\frac{\epsilon_1}{4L_f^2},~
\frac{1}{4p(1+\sigma_0^2)(2\epsilon_2+L_f)}\Big\},\\
&\check{d}_2(\gamma)=\max\Big\{\frac{1}{\epsilon_1},~\frac{\kappa_\eta}{\tilde{d}_2(\gamma)},
~\frac{4\kappa_\eta\epsilon_{10}}{\nu}\Big\},\\
&\check{\epsilon}_3=2\Big(2\epsilon_{2}+\frac{L_f}{n}\Big)\check{\epsilon}_5,\\
&\check{\epsilon}_4=\frac{L_f^2}{4}\Big(\frac{1}{1+\sigma_0^2}+\frac{4}{p}\Big),\\
&\check{\epsilon}_5=2(\sigma^2_1+2(1+\sigma_0^2)\check{\sigma}^2_2),\\
&\epsilon_{10}=8p(1+\sigma_0^2)\Big(2\epsilon_{2}+\frac{L_f}{n}\Big)L_f.
\end{align*}

To prove Theorem~\ref{zerosg-p-a5:thm-sg-diminishingt}, the following lemma is used.
\begin{lemma}\label{zerosg-p-a5:lemma:sg2}
Suppose Assumptions~\ref{zerosg:ass:graph}--\ref{zerosg:ass:zeroth-variance} hold and each $f_i^*>-\infty$. Suppose $\gamma\in(0,d_1)$ and $\eta_k\in(0,\tilde{d}_2(\gamma)]$. Let $\{\bsx_k\}$ be the sequence generated by Algorithm~\ref{zerosg-p:algorithm-random}, then
\begin{subequations}
\begin{align}
\mathbf{E}_{\mathfrak{L}_k}[W_{1,k+1}+W_{4,k+1}]
&\le W_{1,k}+W_{4,k}-\|\bsx_k\|^2_{\frac{\epsilon_{1}}{2}\bsK}
-\frac{1}{4}\eta_k\|\bar{\bsg}_{k}^0\|^2\nonumber\\
&\quad+\epsilon_{10}\eta^2_kW_{4,k}+pn\check{\epsilon}_3\eta^2_k
+pn\check{\epsilon}_4\eta_k\delta_k^2,\label{zerosg-p-a5:w1w4k2}\\
\mathbf{E}_{\mathfrak{L}_k}[W_{1,k+1}]
&\le W_{1,k}-\|\bsx_k\|^2_{\frac{\epsilon_{1}}{2}\bsK}
+16p(1+\sigma_0^2)\epsilon_2L_f\eta_k^2W_{4,k}\nonumber\\
&\quad+2pn\epsilon_2\check{\epsilon}_5\eta^2_k+pn\check{\epsilon}_4\eta_k\delta_k^2,
\label{zerosg-p-a5:w1w4k1}\\
\mathbf{E}_{\mathfrak{L}_k}[W_{4,k+1}]
&\le W_{4,k}+\|\bsx_k\|^2_{2L_f^2\eta_k\bsK}
-\frac{1}{4}\eta_k\|\bar{\bsg}_{k}^0\|^2+\frac{8p(1+\sigma_0^2)L_f^2}{n}\eta_k^2W_{4,k}\nonumber\\
&\quad+pL_f\check{\epsilon}_5\eta^2_k+(p+n)L_f^2\eta_k\delta_k^2.
\label{zerosg-p-a5:w1w4k3}
\end{align}
\end{subequations}
\end{lemma}
\begin{proof}
We know that \eqref{zerosg:rand-grad-esti1}--\eqref{zerosg:rand-grad-smooth} and \eqref{zerosg-a5:rand-grad-esti2} still hold since  Assumptions~\ref{zerosg:ass:zeroth-smooth} and \ref{zerosg:ass:zeroth-variance} hold, and each $f_i^*>-\infty$. Therefore, \eqref{zerosg:v4kspeed-diminishing-thm5} also holds.

From \eqref{zerosg:v4kspeed-diminishing-thm5}, \eqref{zerosg-p:v1k-1}, and \eqref{zerosg-a5:rand-grad-esti2}, we have
\begin{align}\label{zerosg-p:w1w4k-thm11}
\mathbf{E}_{\mathfrak{L}_k}[W_{1,k+1}+W_{4,k+1}]
&\le W_{1,k}+W_{4,k}-\|\bsx_k\|^2_{\epsilon_{1}\bsK
-(L_f^2\eta_k+8p(1+\sigma_0^2)L_f^2\epsilon_{2}\eta^2_k
+\frac{4p(1+\sigma_0^2)L_f^3}{n}\eta^2_k)\bsK}\nonumber\\
&\quad-\frac{1}{4}\eta_k\|\bar{\bsg}_{k}^0\|^2+
8p(1+\sigma_0^2)\Big(2\epsilon_{2}+\frac{L_f}{n}\Big)L_f\eta_k^2W_{4,k}
\nonumber\\
&\quad-\frac{1}{4}(1-2L_f\eta_k)\eta_k\|\bar{\bsg}^s_{k}\|^2
+2pn\Big(2\epsilon_{2}+\frac{L_f}{n}\Big)(\sigma^2_1
+2(1+\sigma_0^2)\check{\sigma}^2_2)\eta^2_k\nonumber\\
&\quad
+\frac{pnL_f^2}{4}\Big(2p\epsilon_{2}\eta_k+\frac{L_fp}{n}\eta_k+\frac{4}{p}\Big)\eta_k\delta_k^2.
\end{align}
Then, similar to the rest of the proof of Lemma~\ref{zerosg-p:lemma:sg2}, we get Lemma~\ref{zerosg-p-a5:lemma:sg2}.
\end{proof}

Now we are ready to prove Theorem~\ref{zerosg-p-a5:thm-sg-diminishingt}.

From $t_1>\check{d}_2(\gamma)\ge\max\{\frac{\kappa_\eta}{\tilde{d}_2(\gamma)},
~\frac{4\kappa_\eta\epsilon_{10}}{\nu}\}$, we have
\begin{align}\label{zerosg-p-a5:etak-pl-speed}
\eta_k=\frac{\kappa_\eta}{k+t_1}\le\frac{\kappa_\eta}{t_1}
<\min\Big\{\tilde{d}_2(\gamma),~\frac{\nu}{4\epsilon_{10}}\Big\}.
\end{align}
Thus, all conditions needed in Lemma~\ref{zerosg-p-a5:lemma:sg2} are satisfied, so \eqref{zerosg-p-a5:w1w4k2}--\eqref{zerosg-p-a5:w1w4k3} hold

From \eqref{zerosg-p-a5:w1w4k2}, \eqref{nonconvex:gg3}, and \eqref{zerosg-p-a5:etak-pl-speed}, we know that \eqref{zerosg-p:vkLya2-pl} still holds when $\epsilon_3$ and $\epsilon_4$ are replaced by $\check{\epsilon}_3$ and $\check{\epsilon}_4$, respectively.

Then, similar to the way to get \eqref{zerosg-p:thm-sg-diminishing-equ2.1bounded} and \eqref{zerosg-p:thm-sg-diminishing-equ2bounded}, we have \eqref{zerosg-p-a5:thm-sg-diminishing-equ2.1bounded} and \eqref{zerosg-p-a5:thm-sg-diminishing-equ2bounded}.

\subsection{Proof of Theorem~\ref{zerosg-p:thm-random-pd-fixed}}\label{zerosg-p:proof-thm-random-pd-fixed}

In addition to the notations defined in Appendix~\ref{zerosg-p:proof-thm-random-sm},
we also denote the following notations.
\begin{align*}
&\epsilon=\frac{1}{2}+\frac{1}{2}\max\{1-\tilde{\epsilon}_8,~\tilde{\epsilon}^2\},\\
&\tilde{\epsilon}_8=\min\Big\{\epsilon_1,~\frac{\nu\eta}{4}\Big\}.
\end{align*}

All conditions needed in Lemma~\ref{zerosg-p:lemma:sg2} are satisfied, so \eqref{zerosg-p:w1w4k2} still holds.

(i) Taking expectation in $\calL_{T}$, summing \eqref{zerosg-p:w1w4k2} over $ k\in[0,T-1]$, and using $\eta_k=\eta$ and $\delta_{i,k}\in(0,\frac{\kappa_\delta\tilde{\varepsilon}^{k}}{\sqrt{p}}]$ yield
\begin{align*}
&\mathbf{E}[W_{1,T}+W_{4,T}]+\frac{\epsilon_1}{2}\sum_{k=0}^{T-1}\|\bsx_k\|^2_{\bsK}
+\frac{1}{8}\eta\sum_{k=0}^{T-1}\|\bar{\bsg}^0_{k}\|^2\le  W_{1,0}+W_{4,0}
+pn\epsilon_3\eta^2T
+\frac{pn\epsilon_4\kappa_\delta^2\eta}{1-\tilde{\epsilon}^2},
\end{align*}
which further implies
\begin{align}\label{zerosg-p:sgproof-vkLya2T-fixed}
\sum_{k=0}^{T-1}\mathbf{E}[\|\bsx_k\|^2_{\bsK}]
\le \frac{2}{\epsilon_1}\Big( W_{1,0}
+pn\epsilon_3\eta^2T
+\frac{pn\epsilon_4\kappa_\delta^2\eta}{1-\tilde{\varepsilon}^2}\Big).
\end{align}
Therefore, \eqref{zerosg-p:thm-sg-fixed-equ3.1} holds due to $\eta=\mathcal{O}(\frac{1}{p})$.

From \eqref{zerosg-p:w1w4k3}, we have
\begin{align}\label{zerosg-p:v4kspeed-fixed}
\mathbf{E}[W_{4,T}]
&\le  W_{4,0}+\sum_{k=0}^{T-1}\mathbf{E}[\|\bsx_k\|^2_{2\eta L_f^2\bsK}]-\frac{1}{8}\eta\sum_{k=0}^{T-1}\mathbf{E}[\|\bar{\bsg}_{k}^0\|^2]
+pL_f\epsilon_5\eta^2T
+\frac{(n+p)L_f^2\kappa_\delta^2\eta}{1-\tilde{\varepsilon}^2}.
\end{align}

From \eqref{zerosg-p:sgproof-vkLya2T-fixed} and \eqref{zerosg-p:v4kspeed-fixed}, we have
\begin{align*}
\sum_{k=0}^{T-1}\mathbf{E}[\|\bar{\bsg}_{k}^0\|^2]&\le
\frac{8W_{4,0}}{\eta}+\sum_{k=0}^{T-1}\mathbf{E}[\|\bsx_k\|^2_{16 L_f^2\bsK}]
+8pL_f\epsilon_5\eta T
+\frac{8(n+p)L_f^2\kappa_\delta^2}{1-\tilde{\varepsilon}^2}\\
&\le
\frac{8W_{4,0}}{\eta}+\frac{16 L_f^2}{\epsilon_1}\Big( W_{1,0}
+pn\epsilon_3\eta^2T
+\frac{pn\epsilon_4\kappa_\delta^2\eta}{1-\tilde{\varepsilon}^2}\Big)
+8pL_f\epsilon_5\eta T
+\frac{8(n+p)L_f^2\kappa_\delta^2}{1-\tilde{\varepsilon}^2},
\end{align*}
which gives \eqref{zerosg-p:thm-sg-fixed-equ3}.

From \eqref{zerosg:rand-grad-smooth}, \eqref{zerosg-p:sgproof-vkLya2T-fixed}, and \eqref{zerosg-p:v4kspeed-fixed}, we have
\begin{align}\label{zerosg-p:sgproof-vkLya2T-fixed1}
&\mathbf{E}[\|\bar{\bsg}^0_k\|^2]
\le2L_f\mathbf{E}[W_{4,k}]\nonumber\\
&\le 2L_f\Big( W_{4,0}
+pL_f\epsilon_5\eta^2k
+\frac{(n+p)L_f^2\kappa_\delta^2\eta}{1-\tilde{\varepsilon}^2}\Big)
+\frac{4L_f^3\eta}{\epsilon_1}\Big( W_{1,0}
+pn\epsilon_3\eta^2k
+\frac{pn\epsilon_4\kappa_\delta^2\eta}{1-\tilde{\varepsilon}^2}\Big),~\forall k\in\mathbb{N}_0.
\end{align}

From \eqref{zerosg-p:w1w4k1}, \eqref{zerosg-p:epsilon1}, $\delta_{i,k}\in(0,\kappa_\delta\tilde{\varepsilon}^{k}]$, \eqref{zerosg-p:sgproof-vkLya2T-fixed1}, and , we have
\begin{align*}
\mathbf{E}[W_{1,k+1}]&\le \mathbf{E}[ (1-\epsilon_1) W_{1,k}
+16p(1+\sigma_0^2)(1+\tilde{\sigma}_0^2)\epsilon_2\eta^2\|\bar{\bsg}^0_{k}\|^2
+2pn\epsilon_2\epsilon_5\eta^2
+pn\epsilon_4\eta\kappa_\delta^2\tilde{\varepsilon}^{2k}]\nonumber\\
&\le (1-\epsilon_1)^k \breve{W}_{0}
+16p(1+\sigma_0^2)(1+\tilde{\sigma}_0^2)\epsilon_2
\eta^2\sum_{\tau=0}^{k}(1-\epsilon_1)^\tau\mathbf{E}[\|\bar{\bsg}^0_{k-\tau}\|^2]\nonumber\\
&\quad+2pn\epsilon_2\epsilon_5\eta^2\sum_{\tau=0}^{k}(1-\epsilon_1)^\tau
+pn\epsilon_4\eta\kappa_\delta^2\sum_{\tau=0}^{k}(1-\epsilon_1)^\tau
\tilde{\varepsilon}^{2(k-\tau)}\nonumber\\
&=\mathcal{O}(pn\eta^2+\epsilon_5p^2\eta^4(k+1)+\epsilon_3p^2n\eta^5(k+1)),
\end{align*}
which gives \eqref{zerosg-p:thm-sg-fixed-equ3.2}.

(ii) If Assumption~\ref{zerosg:ass:fil} also holds, then \eqref{nonconvex:gg3} holds. Thus, \eqref{zerosg-p:vkLya2-pl} also holds when $\eta_k=\eta$.
From \eqref{zerosg-p:vkLya2-pl} and $\eta_k=\eta$, for all $k\in\mathbb{N}_0$, we have
\begin{align}\label{zerosg-p:vkLya2-pl-fixed}
\mathbf{E}_{\mathfrak{L}_k}[\check{W}_{k+1}]
\le (1-\tilde{\epsilon}_8)\check{W}_{k}+pn\epsilon_3\eta^2+pn\epsilon_4\eta\delta_k^2.
\end{align}

From \eqref{zerosg-p:epsilon1}
\begin{align}\label{zerosg-p:vkLya2-pl-r1.1-fixed}
0<\tilde{\epsilon}_8\le\epsilon_{1}<\frac{1}{16}.
\end{align}

From \eqref{zerosg-p:vkLya2-pl-fixed}, \eqref{zerosg-p:vkLya2-pl-r1.1-fixed}, and $\delta_{i,k}\in(0,\kappa_\delta\tilde{\varepsilon}^k]$, we have
\begin{align}\label{zerosg-p:vkLya2-pl-fixed2}
&\mathbf{E}[\check{W}_{k+1}]\le (1-\tilde{\epsilon}_8)^{k+1}\check{W}_{0}
+pn\epsilon_3\eta^2
\sum_{\tau=0}^{k}(1-\tilde{\epsilon}_8)^\tau
+pn\epsilon_4\kappa_\delta^2\eta\sum_{\tau=0}^{k}(1-\tilde{\epsilon}_8)^\tau
\tilde{\varepsilon}^{2(k-\tau)},~\forall k\in\mathbb{N}_0.
\end{align}

From \eqref{zerosg-p:vkLya2-pl-fixed2}, \eqref{zerosg:lemma:sumgeo-equ}, and $\epsilon>\max\{1-\tilde{\epsilon}_8,~\tilde{\epsilon}^2\}$,  we have
\begin{align}\label{zerosg-p:vkLya2-pl-fixed3}
\mathbf{E}[\check{W}_{k+1}]\le \Big(\frac{W_{1,0}+W_{4,0}}{n}
+\frac{p\epsilon_4\kappa_\delta^2\eta}{\epsilon-\tilde{\epsilon}^2}\Big)n\epsilon^{k+1}
+\frac{2\eta}{\tilde{\epsilon}_7}\Big(2\epsilon_{2}+\frac{1}{n}L_f\Big)
n(\sigma^2_1+2(1+\sigma_0^2)\sigma^2_2)p\eta,~\forall k\in\mathbb{N}_0,
\end{align}
which gives \eqref{zerosg-p:thm-sg-fixed-equ1}.

\end{document}